\documentclass{article}
\usepackage{cite}
\usepackage{amssymb,amsthm}
\usepackage{mathrsfs}
\usepackage[tbtags]{amsmath}
\usepackage{enumerate}
\usepackage{amsmath,amssymb,amsthm,mathrsfs,dsfont}

\usepackage[titletoc]{appendix}

\usepackage{dutchcal}
\usepackage{titlesec,hyperref}
\usepackage{color}
\usepackage{verbatim}
\usepackage{fancyhdr}
\usepackage[margin=1cm]{geometry}
\usepackage{mathrsfs}
\usepackage{boondox-cal}
\usepackage{indentfirst} 

\def\var{\varepsilon}
\def\p{\partial}

\newtheorem{theo}{Theorem}[section]
\newtheorem{lemm}[theo]{Lemma}

\newtheorem{rema}[theo]{Remark}
\numberwithin{equation}{section}

\def\ep{\varepsilon}

\allowdisplaybreaks

\begin{document}
	\title{Global uniform regularity for the 3D incompressible
MHD equations with slip boundary condition near an equilibrium}
	\author{
		$\mbox{Jincheng Gao}^1$ \footnote{Email: gaojch5@mail.sysu.edu.cn}, \quad
        $\mbox{Jiahong Wu}^2$ \footnote{Email: jwu29@nd.edu}, \quad
		$\mbox{Zheng-an Yao}^1$ \footnote{Email: mcsyao@mail.sysu.edu.cn}\quad and\quad
		$\mbox{Xuan Yin}^1$ \footnote{Corresponding author. Email: yinx29@mail2.sysu.edu.cn}\\
		\quad
		$^1\mbox{School}$ of Mathematics, Sun Yat-sen University,\\
		Guangzhou 510275, China\\
$^2\mbox{Department}$ of Mathematics
University of Notre Dame, \\
Notre Dame 46556, USA\\
	}
	\date{}
	\maketitle
\vspace{-1em}
	\begin{abstract}
\begin{sloppypar}
This paper solves the global conormal regularity problem for the three-dimensional incompressible MHD equations with slip boundary condition near a background magnetic field. Motivated by applications in geophysics, the MHD system considered here is anisotropic with small vertical dissipation and small horizontal magnetic diffusion. By exploiting the enhanced dissipation due to the background magnetic field and introducing three layers of energy functionals, we are able to establish global-in-time uniform bounds that are independent of vertical viscosity and horizontal resistivity. These global  conormal regularity estimates allow us to pass to the limit and obtain the convergence to the MHD system with no vertical dissipation and horizontal magnetic diffusion. In the special case of the 3D incompressible Navier-Stokes, explicit long-time rates are also extracted in the zero vertical viscosity limit. 
\end{sloppypar}

	\end{abstract}
	
	
	%
\vspace{-1em}

	\tableofcontents
	
	\section{Introduction}
	
	The magnetohydrodynamic (MHD) equations play a significant role in the study of many phenomena in geophysics, astrophysics, cosmology and engineering. They reflect the basic physics laws governing the motion of electrically conducting fluids such as plasmas, liquid metals and electrolytes. The fluid velocity obeys the Naiver-Stokes equations with Lorentz forcing while the magnetic field satisfies the Maxwell's equations of electromagnetism \cite{D1993,D2001,MagneticReconnection2000}. Mathematically the coupling and interaction in the MHD systems leads to much richer structures such as various wave phenomena than those in the Navier-Stokes counterparts.
This paper focuses on the following 3D incompressible MHD equations with anisotropic dissipation
	\begin{equation}\label{eq0}
		\left\{\begin{array}{*{4}{ll}}
			\partial_t u^\ep + u^\ep \cdot \nabla u^\ep -
\Delta_h u^\ep + \nabla p^\ep = \varepsilon \partial_3^2 u^\ep
 + b^\ep \cdot \nabla b^\ep &{\rm in} ~~ \mathbb{R}_+^3,\\
 \nabla \cdot u^\ep = 0 \quad & {\rm in} ~~\mathbb{R}_+^3,\\
			\partial_t b^\ep + u^\ep \cdot \nabla b^\ep - \partial_3^2 b^\ep
= \varepsilon \Delta_h b^\ep + b^\ep\cdot \nabla u^\ep &{\rm in} ~~ \mathbb{R}_+^3,\\
 \nabla \cdot b^\ep= 0 \quad & {\rm in} ~~ \mathbb{R}_+^3.
		\end{array}\right.
	\end{equation}
where $\mathbb{R}_+^3=\{ (x_1,x_2,x_3) \in \mathbb{R} \times \mathbb{R} \times \mathbb{R_+} \}$, and $u^\ep$, $p^\ep$ and $b^\ep$ represent the velocity, total pressure and magnetic field, respectively. For notational convenience, we write
\begin{equation*}
	\Delta_h =  \partial_1^2+\partial_2^2, \quad \partial_h = (\partial_1, \partial_2).
\end{equation*}
\eqref{eq0} is supplemented with the slip boundary condition
\begin{equation}\label{eq46}
	(u_3^\ep , \partial_3 u_h^\ep , b_3^\ep , \partial_3 b_h^\ep )|_{x_3=0}=0.
\end{equation}
\eqref{eq0} models the reconnecting plasmas (see, e.g., \cite{Craig2009,Craig2007}) and the turbulent diffusion of rotating fluids in Ekman layers (see \cite[Chapter 4]{1987Geophysical}).
Formally, as the parameter $\ep \rightarrow 0^+$, \eqref{eq0}
tends to the following system
	\begin{equation}\label{eq47}
		\left\{\begin{array}{*{4}{ll}}
			\partial_t u^0 + u^0 \cdot \nabla u^0 - \Delta_h u^0 +
 \nabla p^0 =  b^0\cdot \nabla b^0 \quad & {\rm in} ~~\mathbb{R}_+^3,\\
 \nabla \cdot u^0 = 0 \quad & {\rm in} ~~\mathbb{R}_+^3,\\
			\partial_t b^0 + u^0 \cdot \nabla b^0 - \partial_3^2 b^0
=  b^0\cdot \nabla u^0 \quad & {\rm in} ~~ \mathbb{R}_+^3,\\
 \nabla \cdot b^0 = 0 \quad & {\rm in} ~~ \mathbb{R}_+^3,
		\end{array}\right.
	\end{equation}
together with the boundary condition
\begin{equation}
		(u_3^0, b_3^0, \partial_3 b_h^0)|_{x_3=0}=0.
\end{equation}
The goal of this paper is to rigorously
establish that the global solution of system \eqref{eq0} converges to
the one of system \eqref{eq47} as the parameter $\ep$ tends to zero.

Due to the smallness of parameter $\ep$ in \eqref{eq0}, it is challenging to study the
global stability problem. To this end, let us introduce the recent developments on the stability and decay of the 3D anisotropic incompressible Navier-Stokes and MHD equations.
The Navier-Stokes and the MHD equations with anisotropic  viscous dissipation arise in the modeling of the reconnecting plasmas (see, e.g., \cite{Craig2009,Craig2007}) and the turbulent diffusion of rotating fluids in Ekman layers (see \cite[Chapter 4]{1987Geophysical}).
If the fluid is not affected by the magnetic field, then the system
\eqref{eq47} will transform into the incompressible Navier-Stokes equation with only horizontal dissipation. The well-posedness and stability results can be found in many references such as \cite{Chemin2006,Chemin2007,Iftimie2002,Paicu2005}.
On the other hand, there are considerable progress on the stability of the background magnetic field for the ideal MHD equations and the MHD equations with kinematic dissipation and no magnetic diffusion.
The nonlinear stability for the ideal MHD equations can be found in \cite{Longtime1988,CL2018,HXY2018,PZZ2018,WZ2017}. The stability problem for the MHD equations without magnetic diffusion was first investigated in \cite{LXZ2015}. The stability has now been obtained by several mathematicians by various methods (see, e.g., \cite{DZ2018,HW2010,LXZ2015,PZZ2018,RWXZ2014,RXZ2016,TW2018,WW2017,WWX2015}). More recent stability and global regularity results can be found in \cite{CWY2014,DJJW2018,DLW2019,Fefferman2014,HL2013,JNWXY2015,WuJiahong2018} and the references therein. In particular,
Wu and Zhu \cite{Wu2021Advance} successfully obtained the global stability
for the MHD equations with mixed partial dissipation and magnetic diffusion \eqref{eq47} under small perturbations near a background magnetic field in $\mathbb{R}^3$.
We note that the vertical diffusion may be relevant
when the resistivity of electrically conducting fluids such as certain plasmas and liquid mental is anisotropic and only in the vertical direction
(see \cite{1989Magnetoresistance}).
However, due to the appearance of small viscous term $\ep \p_3^2 u^\ep$ in \eqref{eq0}
and domain boundary, it is tricky problem for us to establish the global uniform energy estimate
under the framework of classical Sobolev space as in \cite{Wu2021Advance}.
The reason is that we can't establish the higher order normal 
derivative by the combination of tangential derivative 
estimate and the inherent second order normal derivative in Eq.$\eqref{eq0}_1$.

In order to find the suitable energy framework for the Eq.\eqref{eq0},
let us introduce the progress about the full vanishing viscosity limit of
the incompressible Navier-Stokes equation.
The inviscid limit of Navier-Stokes equation in the whole space case has been examined by many
authors, see, for instance, \cite{Constantin1986,Constantin1988,Kato1972,Masmoudi2007CMP, ConstantinWu1995,ConstantinWu1996}.
However, in the presence of physical boundary, the problem
becomes much more complicated due to the
possible appearance of boundary layer.
Under the Navier-slip boundary condition, considerable progress has also been made.
Indeed, for some special types of Navier boundary conditions, the main part of the boundary layer vanishes and, as a consequence, uniform bounds and uniform existence time (independent of viscosity) can be established (see, e.g., \cite{BeiraodaVeiga2010,BeiraodaVeiga2011,Xiao2007}).
However, as pointed in \cite{Iftimie2011ARMA}, uniform regularity estimates
in Sobolev spaces are not expected for general curved boundaries.
 Masmoudi and Rousset \cite{Masmoudi2012ARMA} successfully established uniform estimates in conormal Sobolev spaces
for the 3D general smooth domains with the Navier-slip
boundary conditions. The convergence
of the viscous solutions to the inviscid ones is then obtained by a compactness argument.
Furthermore, based on these uniform estimates, \cite{Xiao2013,Gie2012} provided some better convergence rates.
We briefly recall the vanishing viscosity limits of the viscous MHD equations.
For the 3D viscous MHD system in certain bounded domains with slip boundary conditions,
Xiao, Xin and Wu \cite{Xiao2009JFA}  established the uniform Sobolev regularity
and specified the convergence rate from the viscous solution to the inviscid one.
Later, Meng and Wang \cite{MR3472518}
derived uniform estimates on strong solutions to the
incompressible MHD system with a slip boundary condition
in a conormal Sobolev space with viscosity weight.
This uniform estimate implies that the solution of the viscous MHD system
converges strongly to a solution of the ideal MHD system from a compactness argument.
The above full vanishing viscosity limit progress and convergence rate of solution
between viscous equation and inviscid equation hold only in a small time interval.

Due to the smallness of parameter $\ep$ in \eqref{eq0}
and appearance of domain boundary, we hope to establish the energy estimate
under the framework of conormal Sobolev space rather than the classical
Sobolev space as in \cite{Wu2021Advance}.
In order to deal with the degeneracy of horizontal dissipation of magnetic field,
we investigate the global stability of solution near a  background magnetic field.
Clearly, a special solution of \eqref{eq0} and \eqref{eq47} is given by the zero velocity field and the background magnetic field $b^{(0)}=(1,0,0)$. The perturbation $(u^\ep,B^\ep)$ around this equilibrium with $B^\ep= b^\ep-b^{(0)}$ satisfies
	 \begin{equation}\label{eq1}
	 	\left\{\begin{array}{*{4}{ll}}
	 		\partial_t u^\ep + u^\ep \cdot \nabla u^\ep - \Delta_h u^\ep + \nabla p^\ep = \varepsilon \partial_3^2 u^\ep + B^\ep\cdot \nabla B^\ep + \partial_1 B^\ep
\quad & {\rm in} ~~\mathbb{R}_+^3,\\
\nabla \cdot u^\ep = 0 \quad & {\rm in} ~~\mathbb{R}_+^3,\\
	 		\partial_t B^\ep + u^\ep \cdot \nabla B^\ep - \partial_3^2 B^\ep = \varepsilon \Delta_h B^\ep + B^\ep\cdot \nabla u^\ep + \partial_1 u^\ep
\quad & {\rm in} ~~ \mathbb{R}_+^3,\\
\nabla \cdot B^\ep = 0 \quad & {\rm in} ~~ \mathbb{R}_+^3,\\
	 		u_3^\ep = 0,\quad \partial_3 u_h^\ep=0, \quad B_3^\ep = 0,\quad \partial_3 B_h^\ep=0\quad & {\rm on} ~~ \mathbb{R}^2 \times \{ x_3=0 \}.
	 	\end{array}\right.
	 \end{equation}	
This paper aims at the uniform stability problem on the perturbation $(u^\ep,B^\ep)$. Equivalently, we obtain a small data global well-posedness result for \eqref{eq1} with the initial condition
\begin{align}\label{eq1111}
		u^\ep(0,x)=u_0(x),\quad B^\ep(0,x)=B_0(x).
\end{align}
Formally, as $\ep \to 0^+$, \eqref{eq1} will
tend to the following system
	 \begin{equation}\label{eq12}
	 	\left\{\begin{array}{*{4}{ll}}
	 		\partial_t u^0 + u^0 \cdot \nabla u^0 - \Delta_h u^0 + \nabla p^0
=B^0 \cdot \nabla B^0+ \partial_1 B^0
\quad & {\rm in} ~~\mathbb{R}_+^3,\\
\nabla \cdot u^0 = 0 \quad & {\rm in} ~~\mathbb{R}_+^3,\\
	 		\partial_t B^0 + u^0 \cdot \nabla B^0 - \partial_3^2 B^0
=B^0 \cdot \nabla u^0 + \partial_1 u^0
\quad & {\rm in} ~~ \mathbb{R}_+^3,\\
\nabla \cdot B^0 = 0 \quad & {\rm in} ~~ \mathbb{R}_+^3,\\
	 	u_3^0 = 0, \quad B_3^0 = 0,
\quad \partial_3 B_h^0=0\quad & {\rm on} ~~\mathbb{R}^2 \times \{ x_3 =0\}.
	 	\end{array}\right.
	 \end{equation}
To give a precise account of our main results, we define
$
		H_{tan}^m(\mathbb{R}_+^3)\overset{def}{=}
		\{ f \in
		L^2(\mathbb{R}_+^3) ~|~ \partial_h^k f \in L^2(\mathbb{R}_+^3),~~\forall 0\le k \le m \}.
$
	To define the conormal Sobolev spaces, we consider $(Z_k)_{1 \le k \le 3}$, a finite set of generators of vector fields tangent to the boundary $\mathbb{R}^2 \times \{x_3=0\}$, with $Z_k = \partial_k$, $k=1,2$ and $Z_3= \varphi (x_3) \partial_3$, where $\varphi(x_3)$ is any smooth bounded function such that $\varphi(0)=0$, $\varphi'(0) \neq 0$ and $\varphi(x_3)>0$ for every $x_3>0$ (for example, $\varphi(x_3)=x_3(1+x_3)^{-1}$ fits). We set
	\begin{equation*}
	\begin{aligned}
H_{co}^m(\mathbb{R}_+^3)\overset{def}{=} \{ f \in
		L^2(\mathbb{R}_+^3)~|~Z^\alpha f \in L^2(\mathbb{R}_+^3),~~\forall 0\le |\alpha| \le m \},
	\end{aligned}
\end{equation*}
	where $Z^\alpha \overset{def}{=} Z^{\alpha_1}_1 Z^{\alpha_2}_2 Z^{\alpha_3}_3$.
	We also use the following notations, for every $m \in \mathbb{N}$:
	\begin{equation*}
	\begin{aligned}
		&\Vert f \Vert_{H^m_{tan}}^2\overset{def}{=} \sum\limits_{0\le k \le m} \Vert \partial_h^k f \Vert_{L^2}^2,\quad \Vert f \Vert_{H^m_{co}}^2\overset{def}{=} \sum\limits_{0\le|\alpha| \le m} \Vert Z^\alpha f \Vert_{L^2}^2.
	\end{aligned}
\end{equation*}
Finally, we also define the operator $\Lambda^{s}_h$, $s \in \mathbb{R}$ by
$
\Lambda_h^s f(x_h) \overset{def}{=} \int_{\mathbb{R}^2} |\xi_h|^s \hat f(\xi_h) {\rm e}^{2\pi i x_h \cdot \xi_h} d\xi_h,
$
where $x_h=(x_1,x_2)$, $\xi_h=(\xi_1,\xi_2)$ and $\hat{f}$ is the Fourier transform of $f$ in $\mathbb{R}^2$.
Now, we state our first result about the global-in-time
uniform estimate as follows.
\begin{theo}\label{th1}
	For every integer $m \ge 3$, we consider \eqref{eq1} with the initial data $(u_0,B_0)\in H^m_{co}(\mathbb{R}_+^3)$ satisfying $\nabla \cdot u_0 = \nabla \cdot B_0=0$ and $(\partial_3 u_0,\partial_3 B_0)\in H^{m-1}_{co}(\mathbb{R}_+^3)$. Then, there exists a small constant $\delta_0>0$ such that, if
\begin{equation*}
	\begin{aligned}
			\Vert (u_0, B_0) \Vert_{H^m_{co}}^2
			+\Vert (\partial_3 u_0, \partial_3 B_0) \Vert_{H^{m-1}_{co}}^2 \le \delta_0,
		\end{aligned}
\end{equation*}
the	Eq.\eqref{eq1} has a unique global
        solution $(u^\ep, B^\ep)$ satisfying
        \begin{equation}\label{uniform_estimate}
		\begin{aligned}
			&\|(u^\ep , B^\ep) (t) \|_{H^m_{co}}^2
            +\| (\partial_3 u^\ep, \partial_3 B^\ep) (t) \|_{H^{m-1}_{co}}^2
			+\int_{0}^{t}
			\left(\|(\partial_h u^\ep, \partial_3 B^\ep) \|_{H^m_{co}}^2
			+\|(\partial_h \partial_3 u^\ep, \partial_3^2 B^\ep)\|_{H^{m-1}_{co}}^2\right)
            d \tau\\
           &+\varepsilon \int_{0}^{t}
           \left( \Vert (\partial_3 u^\varepsilon,\partial_h B^\varepsilon) \Vert_{H^m_{co}}^2
           +\Vert (\partial_3^2 u^\varepsilon,\partial_h \partial_3 B_h^\varepsilon) \Vert_{H^{m-1}_{co}}^2 \right) d \tau
           +\int_{0}^{t} \Vert \partial_1 B^\varepsilon(\tau) \Vert_{H^{m-1}_{co}}^2 d \tau
			\le C \delta_0
		\end{aligned}
       \end{equation}
    for a constant $C>0$ and all $t>0$.
    Furthermore, let $(u^0, B^0)$ be the global solution of system \eqref{eq12}
    supplemented
    with the same initial data $(u_0, B_0)$ as system \eqref{eq1}. Then, it holds
    \begin{equation*}
    (u^\ep, B^\ep) \rightarrow (u^0, B^0) \text{~strongly~in~}
    L^\infty(0, t; H^{m-1,co}_{loc}(\mathbb{R}_+^3)),
    \end{equation*}
    for any $t>0$.
	\end{theo}

\begin{rema}
	Recently, Wu and Zhu \cite{Wu2021Advance} successfully established the global stability of the limit equation \eqref{eq47} in the three-dimensional whole space $\mathbb{R}^3$, assuming the initial data is a small perturbation of the steady-state solution induced by a background magnetic field. However, the half-space setting considered here is significantly more challenging due to the effects of domain boundary. Our result in Theorem \ref{th1} provides a global well-posedness framework for both the viscous equation \eqref{eq0} and the limit equation \eqref{eq47}. In comparison to \cite{Wu2021Advance}, the energy estimate \eqref{uniform_estimate} involves only first-order normal derivatives, reflecting the impact of the boundary layer.
\end{rema}


\begin{rema}
	The anisotropic dissipative structure in \eqref{eq0} eliminates the need for deriving the $W^{1, \infty}$
	estimate, distinguishing our analysis from the local-in-time vanishing full viscosity limit of the well-known incompressible Navier-Stokes equations (cf. \cite{Masmoudi2012ARMA}). Instead, our result in Theorem \ref{th1} addresses global-in-time estimates that are uniform with respect to $\varepsilon$
	in \eqref{eq0}. Given the exclusively vertical dissipative structure of the magnetic field in \eqref{eq0}, we leverage the enhanced dissipation induced by the background magnetic field and introduce three layers of energy functionals. This approach allows us to establish global-in-time uniform bounds that are independent of both the vertical viscosity and horizontal resistivity.
\end{rema}

Theorem \ref{th1} states that
the global solution $(u^\ep, B^\ep)$ of system \eqref{eq1} converges to the solution $(u^0, B^0)$,
which is the global solution of \eqref{eq12} under the same
background magnetic field and small initial data assumptions, as the parameter
$\ep$ tends to zero.
One natural question is to establish the specific convergence rate
between two global solutions.
Since the two global solutions both tend to the the equilibrium state
as the time tends to infinity, the acceptable convergence rate
should be independent of time.
This is an interesting and tricky problem, since we have to
establish some suitable time decay rate estimates under the
conormal derivative framework.
We are able to obtain the convergence rate when the magnetic field vanishes, namely the Navier-Stokes case. More precisely, we consider
   \begin{equation}\label{eq1-3}
  \left\{\begin{array}{*{4}{ll}}
\partial_t u^\var + u^\var \cdot \nabla u^\var - \Delta_h u^\var + \nabla p^\var
= \varepsilon \partial_3^2 u^\var
& {\rm in} ~~\mathbb{R}_+^3,\\
\nabla \cdot u^\var = 0 & {\rm in} ~~\mathbb{R}_+^3,\\
	 		u_3^\var = 0,\quad \partial_3 u_h^\var=0
\quad & {\rm on} ~~ \mathbb{R}^2 \times \{ x_3=0\},\\
u^\var|_{t=0}=u_0, & {\rm in} ~~\mathbb{R}_+^3.
	 	\end{array}\right.
	 \end{equation}
As $\var \rightarrow 0^+$, then $(u^\var, \nabla p^\var)$
converges to the solution $(u^0, \nabla p^0)$ satisfying
\begin{equation}\label{eq1-4}
\left\{\begin{array}{*{4}{ll}}
\partial_t u^0 + u^0 \cdot \nabla u^0 - \Delta_h u^0 + \nabla p^0=0
& {\rm in} ~~\mathbb{R}_+^3,\\
\nabla \cdot u^0 = 0 & {\rm in} ~~\mathbb{R}_+^3\\
	 		u_3^0 =0
\quad & {\rm on} ~~\mathbb{R}^2 \times \{ x_3=0 \} ,\\
u^0|_{t=0}=u_0 & {\rm in} ~~\mathbb{R}_+^3.
	 	\end{array}\right.
	 \end{equation}
The well-posedness and stability results for \eqref{eq1-4}
can be found in many references such as
\cite{Chemin2006,Chemin2007,Iftimie2002,Paicu2005}.
Our second result establishes the global-in-time
uniform estimates and the convergence rate for
anisotropic incompressible Navier-Stokes equations
\eqref{eq1-3}.
	\begin{theo}\label{th2}
		For every integer $m \ge 3$, we consider \eqref{eq1-3} with the initial data
$u_0 \in H^m_{co}(\mathbb{R}_+^3)$ satisfying $\nabla \cdot u_0 = 0$ and
$\partial_3 u_0 \in H^{m-1}_{co}(\mathbb{R}_+^3)$.
Then, there exists a small constant $\delta_1>0$ such that, if
\begin{equation}\label{small-01}
			\Vert u_0 \Vert_{H^m_{co}}^2
			+\Vert \partial_3 u_0\Vert_{H^{m-1}_{co}}^2 \le \delta_1,
\end{equation}
the	Eq.\eqref{eq1-3} has a unique global classical
        solution $u^\ep $ satisfying
        \begin{equation}\label{ns-estimate}
		\begin{aligned}
			&\|u^\ep (t) \|_{H^m_{co}}^2
            +\| \partial_3 u^\ep(t) \|_{H^{m-1}_{co}}^2
			+\int_{0}^{t}
			\left(\|\partial_h u^\ep\|_{H^m_{co}}^2
			+\|\partial_h \partial_3 u^\ep\|_{H^{m-1}_{co}}^2\right)d \tau\\
           &+\varepsilon \int_{0}^{t}
           \left( \Vert \partial_3 u^\varepsilon \Vert_{H^m_{co}}^2
           +\Vert \partial_3^2 u^\varepsilon \Vert_{H^{m-1}_{co}}^2 \right) d \tau
			\le C_1 \delta_1
		\end{aligned}
       \end{equation}
    for all $t>0$.
    Furthermore, for any $s \in \big(\frac{13}{14}, 1\big)$, if there exists a small constant $\delta_2>0$ such that
    \begin{equation}\label{small-02}
    \|\Lambda^{-s}_h u_0\|_{L^2}^2
			+\|\Lambda_h^{-s} \partial_3 u_0\|_{L^2}^2
    \le \delta_2,
    \end{equation}
    then the global solution $u^\ep$ satisfies the time decay estimate
    \begin{equation}\label{decay-1}
	\|u^\ep(t) \|_{H_{tan}^m}^2+\|\partial_3 u^\ep(t)\|_{H_{tan}^{m-1}}^2
	\le C_2 (1+t)^{-s}
	\end{equation}
	for all $t>0$.
   If $u^0$ is the global solution of system \eqref{eq1-4}
    supplemented with initial data $u_0$ satisfying
   \eqref{small-01} and \eqref{small-02}, then
    \begin{equation}\label{decay-2}
	\|u^0(t) \|_{H_{tan}^m}^2+\|\partial_3 u^0 (t)\|_{H_{tan}^{m-1}}^2
	\le C_3 (1+t)^{-s}
	\end{equation}
	for all $t>0$.
    Furthermore, for any $t>0$, the difference $u^\ep-u^0$ will obey the estimate
		\begin{equation}\label{difference-estimate}
			\|(u^\varepsilon - u^0)(t)\|_{L^2}^2
			\le C_4
			\varepsilon^{\frac12},
		\end{equation}
with the positive constants $C_i(i=1,2,3,4)$
above independent of time and parameter $\ep$.
	\end{theo}

\begin{rema}
Due to the lack of the vertical dissipation in
anisotropic incompressible Navier-Stokes \eqref{eq1-4}, the classical
Fourier splitting method can not help us establish time decay rate directly.
Thus, due to analysis of spectrum of linear system and energy estimate
of nonlinear system, the optimal decay rate of anisotropic incompressible Navier-Stokes
with only horizontal dissipation in three-dimensional whole space
can be established \cite{Ji2022CVPDE}.
As a byproduct of Theorem \ref{th2}, we establish optimal time decay
rate for the anisotropic incompressible Navier-Stokes \eqref{eq1-4}
in three-dimensional half space. The advantage of this time decay rate
is that the higher order tangential derivatives of velocity
obey the same time decay rate as the lower one.
On the other hand, due to the estimates
\eqref{decay-1}, \eqref{decay-2}, \eqref{difference-estimate}
and anisotropic Sobolev inequality, it holds
\begin{equation*}
\|(u^\varepsilon - u^0)(t)\|_{L^\infty}
\le C \ep^{\frac{1}{16}},
\end{equation*}
where $C$ is a positive constant independent of time and parameter $\ep$.
\end{rema}


Throughout this paper, we use symbol $A \lesssim B$ for $ A\le C B$ where $C>0$ is a constant which may change from line to line and independent of time and  $\varepsilon$
for $0<\varepsilon <1$.
The symbol $A \thicksim B$ implies that $A \lesssim B$ and $B \lesssim A$.
	The rest of the paper is organized as follows.
	In Section \ref{difficulty}, we explain the difficulties and our approach to establish
the global uniform estimate and time decay rate estimate
for the system  \eqref{eq1} and \eqref{eq1-3} respectively.
In Section  \ref{global-estimate}, we establish the global uniform estimate
for system  \eqref{eq1} under the condition of small initial data.
In Section  \ref{asymptotic-behavior}, we will establish the time decay
rate for the system \eqref{eq1-3} under the addition condition \eqref{small-02}.
This decay rate provides us a specific
convergence rate for the solution of 3D incompressible Navier-Stokes
equations as the vertical viscosity tends to zero.
Finally, we introduce some useful inequalities in Appendix \ref{usefull-inequality}
and provide the detailed proofs for the claimed technical estimates in Appendix
\ref{claim-estimates}.

\section{Difficulties and outline of our approach}\label{difficulty}
In this section, we will explain the main difficulties of proving Theorems \ref{th1} and  \ref{th2} as well as our strategies for overcoming them.

\textbf{Step 1: Analysis of global well-posedness independent of $\ep$.}
Compared with the global-in-time estimate for system \eqref{eq47} in \cite{Wu2021Advance},
the uniform estimate \eqref{uniform_estimate} can only include one order
vertical derivative since our estimates should be independent of parameter $\ep$.
First of all, we apply the anisotropic Sobolev inequality to
establish the estimate of tangential derivative of velocity and magnetic field.
Due to the lack of the vertical dissipation and the horizontal magnetic diffusion,
we can not apply the anisotropic Sobolev inequality to deal with the two difficult terms
\begin{align}\label{eq79}
	\int_{\mathbb{R}_+^3} \partial_1 u^\varepsilon_1 |\partial_2^m B^\varepsilon|^2 dx
	\quad {\rm and} \quad \int_{\mathbb{R}_+^3} \partial_1 u^\varepsilon_3 |\partial_2^m B^\varepsilon|^2 dx.
\end{align}
Thus, we replace $\partial_1 u^\varepsilon_1$ and $\partial_1 u^\varepsilon_3$ via the equation of $B^\varepsilon$ in \eqref{eq1}$_3$, i.e.,
\begin{align}\label{eq80}
	\partial_1 u^\varepsilon
	=
	\partial_t B^\varepsilon + u^\varepsilon \cdot \nabla B^\varepsilon -\varepsilon \Delta_h B^\varepsilon - \partial_3^2 B^\varepsilon - B^\varepsilon\cdot \nabla u^\varepsilon .
\end{align}
Due to the special structure
(i.e., the appearance of term $\partial_1 u^\ep$)
in $\eqref{eq1}_3$ or \eqref{eq80},
we apply  this equation to give the bound of difficult term \eqref{eq79}.
Secondly, we will establish estimate for the vertical derivative of velocity.
With the help of the divergence-free condition, we only need to establish the
estimate for $(\partial_3 u_h^\ep, \partial_3 B_h^\ep)$.
On the one hand, as we deal with the difficult term
$$\sum\limits_{\substack{\beta+\gamma=\alpha\\\beta \neq 0}} C_{\beta,\gamma} \int_{\mathbb{R}^2 \times [0,1)} Z^\beta u_3^\ep
\cdot Z^\gamma \partial_3 \omega_h^{u^\ep} \cdot Z^\alpha \omega_h^{u^\ep} dx$$
in \eqref{eq01}, one should apply the boundary condition of $u_3^\ep$,
divergence-free condition and anisotropic Sobolev inequality
to control this term by the quantity
$\Vert \omega_h^{u^\ep} \Vert_{H_{co}^{m-1}}
	\Vert \partial_h {u^\ep} \Vert_{H_{co}^{m}}
	\Vert \partial_h \omega_h^{u^\ep} \Vert_{H_{co}^{m-1}}$
(see \eqref{3206} and \eqref{3237} in detail).
On the other hand, similar to \eqref{eq79}, we will apply the
Eq.\eqref{eq80} to control the term
$\int_{\mathbb{R}_+^3} \partial_1 u^\ep_1 |Z^\alpha \omega_1^{B^\ep}|^2 dx$.
Finally, in order to close the estimate, we need to establish
the dissipative estimate $\Vert \partial_1 B^\var(\tau) \Vert_{H^{m-1}_{co}}^2$.
Thus, we apply the dissipative structure of magnetic field $\partial_1 B^\ep$
in Eq.$\eqref{eq1}_1$.
Then, we need to use the Eq.$\eqref{eq1}_3$ again to give the control
of time derivative term
$$
\sum_{|\alpha| \le m-1}
	\int_{\mathbb{R}_+^3}
	Z^\alpha u^\ep \cdot Z^\alpha \partial_1 \partial_t B^\ep dx
$$
in \eqref{3238}.
%
%
%
%
%
%
%
%
Therefore, we can use bootstrapping argument and these inequalities to prove Theorem \ref{th1}.

\textbf{Step 2: Analysis of asymptotic behavior with respect to time.}
Let $u^\ep$ and $u^0$ be the global solution of systems \eqref{eq1-3}
and \eqref{eq1-4} with the same initial data assumptions
\eqref{small-01} and \eqref{small-02} respectively.
Then, we are devoted to establishing the specific convergence rate
with respect to parameter $\ep$.
By energy method, one can establish the following estimate
\begin{equation}\label{grow-1}
\begin{aligned}
\|(u^\ep-u^0)(t)\|_{L^2}^2
&\lesssim \varepsilon \exp \bigg\{\int_0^t \|\partial_3 \partial_h u^0\|_{L^2}
                     \|\nabla u^0\|_{L^2} d \tau \bigg\}\\
         &~~~~\times \int_0^t \|\partial_3 u^\varepsilon\|_{L^2}
           \|\partial_3 (u^\ep-u^0)\|_{L^2}
           \exp\bigg\{-\int_0^{\tau} \|\partial_3 \partial_h u^0\|_{L^2}
                     \|\nabla u^0\|_{L^2} d\varsigma\bigg\}d\tau.
\end{aligned}
\end{equation}
If one hopes to obtain uniform  bound (with respect to time) for
the quantity $u^\ep-u^0$, we should not only establish suitable uniform (with respect to time) bound for the quantity
$\int_0^t \|\partial_3 \partial_h u^0\|_{L^2}\|\nabla u^0\|_{L^2} d \tau$
but also provide some time decay estimates for the global solution $u^\ep$ and $u^0$ respectively.
Due to the degeneracy of dissipative term $\partial_3^2 u^\ep$, the classical
Fourier splitting method can not help us establish time decay rate directly.
Thus, from the differential inequality
\begin{equation}\label{energy-2}
\frac{d}{dt}\left(\|u^\ep\|_{H_{tan}^m}^2
+\|\omega_h^{u^\ep}\|_{H_{tan}^{m-1}}^2\right)
+\left(\|\partial_h u^\ep\|_{H_{tan}^m}^2
+\|\partial_h  \omega_h^{u^\ep}\|_{H_{tan}^{m-1}}^2\right)
+\varepsilon\left(\|\partial_3 u^\ep\|_{H_{tan}^m}^2+\|\partial_3 \omega_h^{u^\ep}\|_{H_{tan}^{m-1}}^2\right)
\le 0,
\end{equation}
one can apply the interpolation inequality
$(e.g.~ \|f\|_{H_{tan}^m}^2
\lesssim \|\Lambda^{-s }f\|_{H_{tan}^m}^{\frac{2}{1+s}}
         \|\partial_h f\|_{H_{tan}^m}^{\frac{2s}{1+s}})$ to obtain
\begin{equation}\label{energy-1}
\frac{d}{dt}\left(\|u^\ep\|_{H_{tan}^m}^2+\|\omega_h^{u^\ep}\|_{H_{tan}^{m-1}}^2\right)
+C_0^{-\frac{1}{s}}\left(\|u^\ep\|_{H_{tan}^m}^2+\|\omega_h^{u^\ep}\|_{H_{tan}^{m-1}}^2\right)^{1+\frac{1}{s}}\le 0,
\end{equation}
if one has the uniform bound
$
\|\Lambda^{-s }u^\ep\|_{L^2}^2+\|\Lambda^{-s }\omega_h^{u^\ep}\|_{L^2}^2 \le C_0.
$
The differential inequality \eqref{energy-1} yields the time decay rate
\begin{equation}\label{decay-3}
\|u^\ep(t)\|_{H_{tan}^m}^2+\|\omega_h^{u^\ep}(t)\|_{H_{tan}^{m-1}}^2
\lesssim C_0(1+t)^{-s},
\end{equation}
where the index $s \in \big(\frac{13}{14}, 1\big)$.
Similarly, one can apply this method again to establish
the time decay estimate for $u^0$ as follows:
\begin{equation}\label{decay-4}
\|u^0(t)\|_{H_{tan}^m}^2+\|\omega_h^{u^0}(t)\|_{H_{tan}^{m-1}}^2
\lesssim C_0(1+t)^{-s}.
\end{equation}
However, the decay rates \eqref{decay-3}-\eqref{decay-4}
are still not enough to obtain uniform  bound
for the quantity
$\ep \int_0^t \|\partial_3 u^\varepsilon\|_{L^2}
\|\partial_3 (u^\ep-u^0)\|_{L^2} d\tau$.
Fortunately, we can apply the decay rate \eqref{decay-3}
to obtain suitable time integral for the degeneracy term
$\varepsilon \int_0^t(1+\tau)^{\sigma}
   \|\partial_3 u^\ep \|_{H_{tan}^m}^2 d\tau \lesssim 1$,
where the parameter $\sigma \overset{def}{=}s-\frac{14s-13}{6(2-s)}$.
Finally, we will establish the important uniform bound of
quantity $\|(\Lambda^{-s }u^\ep, \Lambda^{-s }\omega_h^{u^\ep})\|_{L^2}^2 $.
The most tricky problem is to deal with the difficult term
$\int_{\mathbb{R}_+^3} \Lambda_h^{-s}(u^\ep \cdot \nabla \omega^{u^\ep}_h)\cdot \Lambda_h^{-s} \omega^{u^\ep}_h dx$ as we establish the uniform estimate for
$\|\Lambda^{-s }\omega_h^{u^\ep}\|_{L^2}$.
By using the uniform estimate \eqref{ns-estimate},
conormal version interpolation inequality \eqref{a5}
and some time integral of dissipative term,
one can establish the uniform estimate of $\int_{\mathbb{R}_+^3} \Lambda_h^{-s}(u^\ep \cdot \nabla \omega^{u^\ep}_h)\cdot \Lambda_h^{-s} \omega^{u^\ep}_h dx$
(see estimates \eqref{524} and \eqref{528} in detail).

\section{Global-in-time uniform estimate}\label{global-estimate}
	This section provides the proof of Theorem \ref{th1}.
	We use the notations $\omega^{u^\varepsilon}\overset{def}{=}\nabla \times u^\varepsilon$ and $\omega^{B^\varepsilon}\overset{def}{=}\nabla \times B	^\varepsilon$. Furthermore, for some vector field $v=(v_1,v_2,v_3)$, we denote $v_h=(v_1,v_2)$.
For every $m\ge 3$, we introduce the following energy functionals
	\begin{equation*}
	\begin{aligned}
		E_1(t)&\overset{def}{=}\sup\limits_{0\le \tau \le t} \Vert (u^\var,B^\varepsilon)(\tau) \Vert_{H^m_{tan}}^2
		+ 2 \int_{0}^{t} \Vert (\partial_h u^\var,\partial_3 B^\var)(\tau) \Vert_{H^m_{tan}}^2 d \tau
		+2\varepsilon \int_{0}^{t} \Vert (\partial_3 u^\var,\partial_h B^\var)(\tau) \Vert_{H^m_{tan}}^2 d \tau,\\
 	E_2(t)&\overset{def}{=}\sup\limits_{0\le \tau \le t}
 	\Vert (\omega^{u^\var}_h,\omega^{B^\var}_h)(\tau) \Vert_{H^{m-1}_{co}}^2
 + 2 \int_{0}^{t} \Vert (\partial_h \omega^{u^\varepsilon}_h,\partial_3 \omega^{B^\varepsilon}_h)(\tau) \Vert_{H^{m-1}_{co}}^2 d \tau\\
	&~~~~~~~~~~~~+2\varepsilon \int_{0}^{t} \Vert (\partial_3 \omega^{u^\varepsilon}_h,\partial_h \omega^{B^\varepsilon}_h)(\tau) \Vert_{H^{m-1}_{co}}^2 d \tau,
	\\
	E(t)&\overset{def}{=}E_1(t)+E_2(t),\\
	G(t)&\overset{def}{=} \int_{0}^{t} \Vert \partial_1 B^\var(\tau) \Vert_{H^{m-1}_{co}}^2 d \tau.
\end{aligned}
\end{equation*}	
\textbf{Proof of Theorem \ref{th1}.}
First of all, similar to the result in \cite{MR3472518},
one can establish the uniform (with respect to $\ep$)
local-in-time existence and uniqueness for system \eqref{eq1}.
Thus, we will extend the local-in-time solution to be global one
under the small condition of initial data.
Since we will use bootstrapping argument to give the proof
if the initial data satisfies some small initial condition,
it is convenient to assume the energy $E(t)\le 1$.
Thus, we claim that we can establish the following energy inequalities:
	\begin{align}\label{eq1-1}
		E_{1}(t) &\lesssim E(0) + E(0)^\frac{3}{2} + E(t)^\frac{3}{2}
		+ G(t)^\frac{3}{2}
+
E(t)^2
+ G(t)^2,\\
\label{eq1-0}
	E_{2}(t) &\lesssim E(0) + E(0)^\frac{3}{2} + E(t)^\frac{3}{2}
	+ G(t)^\frac{3}{2}+
	E(t)^2
	+ G(t)^2,
\end{align}
and
\begin{align}\label{eq1-2}
	G(t) \lesssim E(0) + E(t) + E(t)^\frac{3}{2} + G(t)^\frac{3}{2}.
\end{align}
By \eqref{eq1-1}--\eqref{eq1-2}, we obtain
	\begin{equation*}
	\begin{aligned}
		E(t) + G(t) \lesssim E(0) + E(0)^\frac{3}{2} + E(t)^\frac{3}{2} + G(t)^\frac{3}{2}+
		E(t)^2
		+ G(t)^2,
	\end{aligned}
\end{equation*}
	or, for some pure constants $C_1$, $C_2$ and $C_3$, we obtain
	\begin{align}\label{eq99}
		E(t)+G(t) \le
		C_1 \big( E(0) + E(0)^\frac{3}{2} \big)
		+
		C_2 \big( E(t)^\frac{3}{2} + G(t)^\frac{3}{2} \big)
		+
		C_3 \big( E(t)^2 + G(t)^2 \big).
	\end{align}
	To use the bootstrapping argument(see, e.g., \cite{Tao2006}), we assume
	\begin{align}\label{eq100}
		E(t) + G(t) \le \min \bigg\{  \frac{1}{16 C_2^2}, \frac{1}{4C_3} \bigg\}.
	\end{align}
	We then show that \eqref{eq99} allows us to conclude that $E(t)+G(t)$ actually admits a smaller bound by taking the initial data $E(0)$ small enough. Actually, if \eqref{eq100} holds, \eqref{eq99} implies
	\begin{equation*}
	\begin{aligned}
		E(t)+G(t)
		&\le
		C_1 \big( E(0) + E(0)^\frac{3}{2} \big)
		+
		C_2
		\sqrt{E(t)+G(t)}
		\big( E(t) + G(t) \big)
		+C_3
		\big(
		E(t)+G(t)\big)
		\big(
		E(t)+G(t)\big)
		\\
		&\le C_1 \big( E(0) + E(0)^\frac{3}{2} \big)
		+
		\frac{1}{2}
		\big(
		E(t) + G(t)
		\big),
	\end{aligned}
\end{equation*}
	or
	\begin{align}\label{eq101}
		E(t) + G(t) \le 2C_1 \big(
		E(0)+E(0)^\frac{3}{2}
		\big).
	\end{align}
Hence, if $E(0)$ is small enough such that
\begin{align}\label{eq102}
	2C_1 \big(E(0) + E(0)^\frac{3}{2} \big) < \min \bigg\{  \frac{1}{16 C_2^2}, \frac{1}{4C_3} \bigg\},
\end{align}
then $E(t)+G(t)$ actually admits a smaller bound in \eqref{eq101} than the one in the ansatz \eqref{eq100}. The bootstrapping argument then assesses that \eqref{eq101} holds for all time when $E(0)$ satisfies \eqref{eq102}.
Let us define the norm
\begin{equation*}
\begin{aligned}
X(t)\overset{def}{=}
&\|(u^\ep , B^\ep) (t) \|_{H^m_{co}}^2
            +\| (\partial_3 u^\ep, \partial_3 B^\ep) (t) \|_{H^{m-1}_{co}}^2
			+\int_{0}^{t}
			\left(\|(\partial_h u^\ep, \partial_3 B^\ep) \|_{H^m_{co}}^2
			+\|(\partial_h \partial_3 u^\ep, \partial_3^2 B^\ep)\|_{H^{m-1}_{co}}^2\right)
            d \tau\\
           &+\varepsilon \int_{0}^{t}
           \left( \Vert (\partial_3 u^\varepsilon,\partial_h B^\varepsilon) \Vert_{H^m_{co}}^2
           +\Vert (\partial_3^2 u^\varepsilon,\partial_h \partial_3 B_h^\varepsilon) \Vert_{H^{m-1}_{co}}^2 \right) d \tau,
\end{aligned}
\end{equation*}
then it is easy to check that $X(t)$ is equivalent to $E(t)$.
In other words, there exist two constant $C_*$ and $C^*$ such that
\begin{equation}\label{equivalent-norm}
C_* E(t) \le X(t) \le C^* E(t),
\end{equation}
which can be proved in \eqref{b7} in Appendix \ref{claim-estimates}.
Therefore, the combination of estimate \eqref{eq101},
initial data condition \eqref{eq102} and equivalent norm
\eqref{equivalent-norm} complete the proof of estimate \eqref{uniform_estimate}
in Theorem \ref{th1}.
Finally, due to the Eq.$\eqref{eq1}_1$, Eq.$\eqref{eq1}_3$
and estimate \eqref{uniform_estimate}, it is easy to check that
\begin{equation}\label{31001}
\int_0^t (\|\partial_\tau u^\ep \|_{H^{m-1}_{co}}^2
+\|\partial_\tau B^\ep \|_{H^{m-1}_{co}}^2)d\tau \le C_1,
\end{equation}
with the positive constant $C_1$ independent of
time $t$ and parameter $\ep$.
On the other hand, the estimate \eqref{uniform_estimate} gives directly
\begin{equation}\label{31002}
\underset{0 \le \tau \le t}{\sup}
\|(u^\ep, B^\ep)(\tau)\|_{H^{m}_{co}}^2
+\underset{0 \le \tau \le t}{\sup}
\|\nabla(u^\ep, B^\ep)(\tau)\|_{H^{m-1}_{co}}^2 \le  C \delta_0.
\end{equation}
Therefore, the estimates \eqref{31001}, \eqref{31002}
and strong compactness argument (see \cite[Lemma 4]{MR1062395SIAM1990})
yield directly
\begin{equation*}
(u^\ep, B^\ep) \rightarrow (u^0, B^0) \text{~strongly~in~}
L^\infty(0, t; H^{m-1,co}_{loc}(\mathbb{R}_+^3)),
\end{equation*}
for any $t>0$. Therefore, we complete the proof of Theorem \ref{th1}.

\subsection{Estimate of tangential derivative}
In this section, we will give the proof for  estimate \eqref{eq1-1}.
For notational convenience, we drop the superscript $\ep$ throughout this section. Due to the equivalence of $\Vert (u,B)\Vert_{H^m_{tan}}$ with $\Vert (u,B)\Vert_{L^2} + \Vert (\partial_h^m u,\partial_h^m B)\Vert_{L^2}$, it suffices to bound the $L^2$-norm of $(u,B)$ and $(\partial_h^m u, \partial_h ^m B )$.
%
%
First of all, the equations $\eqref{eq1}$ and $\nabla \cdot u =\nabla \cdot B=0$
imply the $L^2$-energy estimate as follows:
\begin{equation}\label{3124}
	\begin{aligned}
\frac{d}{dt}\Vert (u,B)(t) \Vert_{L^2}^2
+ 2 \Vert (\partial_h u,\partial_3 B) \Vert_{L^2}^2
+ 2 \varepsilon \Vert (\partial_3 u,\partial_h B)\Vert_{L^2}^2 = 0.
	\end{aligned}
\end{equation}	
Now we turn to the $L^2$-energy estimate of $(\partial_h^m u,\partial_h^m B)$.
	For every $m \ge 3$, applying $\partial_h^m$ to $\eqref{eq1}$ and dotting by $(\partial_h^m u,\partial_h^m B)$, we obtain
	\begin{equation}\label{eqE1}
		\begin{aligned}
			&~~~~~\frac{1}{2}\frac{d}{dt}  \Vert (\partial_h^m u,\partial_h^m B) \Vert_{L^2}^2
			+
			\Vert (\partial_h \partial_h^m u,\partial_3 \partial_h^m B) \Vert_{L^2}^2
			+ \varepsilon \Vert (\partial_3 \partial_h^m u,\partial_h \partial_h^m B) \Vert_{L^2}^2\\
			&~=-\sum\limits_{i=1}^2\int_{\mathbb{R}_+^3}
 [\partial_i^m,u\cdot\nabla]u\cdot \partial_i^mu dx
			+\sum\limits_{i=1}^2\int_{\mathbb{R}_+^3} [\partial_i^m,B\cdot\nabla]B\cdot \partial_i^mu dx\\
			&~~~~-\sum\limits_{i=1}^2\int_{\mathbb{R}_+^3} [\partial_i^m,u\cdot\nabla]B\cdot \partial_i^mB dx
			+\sum\limits_{i=1}^2\int_{\mathbb{R}_+^3} [\partial_i^m,B\cdot\nabla]u\cdot \partial_i^mB dx
			\\
			&\overset{def}{=}I_1+I_2+I_3+I_4.
		\end{aligned}
	\end{equation}
Due to the basic fact
\begin{equation*}
	\begin{aligned}
\left[\partial_i^m, u\cdot \nabla \right]
u = \partial_i^m(u \cdot \nabla u) - u \cdot \nabla \partial_i^m u
 = \sum\limits_{\substack{k+\ell=m\\k\neq0}} C_{k,\ell} \partial_i^k u \cdot \partial_i^\ell \nabla u ,~~\forall i=1,2,
\end{aligned}
\end{equation*}
we use Lemma \ref{lemma1} and $\nabla \cdot u=0$ to obtain
    \begin{equation}\label{3117}
	\begin{aligned}
    	I_1 &= -\sum\limits_{i=1}^2 \sum\limits_{\substack{k+\ell=m\\k\neq0}} C_{k,\ell}
    	\int_{\mathbb{R}_+^3} \partial_i^k u \cdot \partial_i^\ell \nabla u \cdot \partial_i^mu dx\\
    	&=-\sum\limits_{i=1}^2 \sum\limits_{\substack{k+\ell=m\\k\neq0}} C_{k,\ell}
    	\int_{\mathbb{R}_+^3} \partial_i^k u_h \cdot \partial_i^\ell \partial_h u \cdot \partial_i^mu dx
    	-\sum\limits_{i=1}^2 \sum\limits_{\substack{k+\ell=m\\k\neq0}} C_{k,\ell}
    	\int_{\mathbb{R}_+^3} \partial_i^k u_3  \partial_i^\ell \partial_3 u \cdot \partial_i^mu dx \\
    	&\lesssim \sum\limits_{\substack{k+\ell=m\\k\neq0}} \Vert \partial_h^k u_h \Vert_{L^2}^{\frac{1}{2}} \Vert \partial_1 \partial_h^k u_h \Vert_{L^2}^{\frac{1}{2}}
	\Vert \partial_h^\ell \partial_h u \Vert_{L^2}^{\frac{1}{2}} \Vert \partial_3 \partial_h^\ell \partial_h u \Vert_{L^2}^{\frac{1}{2}} \Vert \partial_h^m u \Vert_{L^2}^{\frac{1}{2}} \Vert \partial_2 \partial_h^m u \Vert_{L^2}^{\frac{1}{2}} \\
	&~~~~+ \sum\limits_{\substack{k+\ell=m\\k\neq0}} \Vert \partial_h^k u_3 \Vert_{L^2}^{\frac{1}{2}} \Vert \partial_3 \partial_h^k u_3 \Vert_{L^2}^{\frac{1}{2}}
	\Vert \partial_h^\ell \partial_3 u \Vert_{L^2}^{\frac{1}{2}} \Vert \partial_1 \partial_h^\ell \partial_3 u \Vert_{L^2}^{\frac{1}{2}} \Vert \partial_h^m u \Vert_{L^2}^{\frac{1}{2}} \Vert \partial_2 \partial_h^m u \Vert_{L^2}^{\frac{1}{2}} \\
&\lesssim
\Vert u \Vert_{H_{tan}^m}
\Vert \partial_h u \Vert_{H_{tan}^m}^\frac32
\Vert \partial_h \partial_3 u \Vert_{H_{tan}^{m-1}}^\frac12
+
\Vert u \Vert_{H_{tan}^m}^\frac12
\Vert \partial_3 u \Vert_{H_{tan}^{m-1}}^\frac12
\Vert \partial_h u \Vert_{H_{tan}^m}^\frac32
\Vert \partial_h \partial_3 u \Vert_{H_{tan}^{m-1}}^\frac12
\\
	&\lesssim \Vert u \Vert_{H_{tan}^m}
	\Vert \partial_h u \Vert_{H_{tan}^m}^\frac32
	(\Vert \partial_h \partial_h u \Vert_{H_{tan}^{m-1}}
	+
	\Vert \partial_h \omega_h^u \Vert_{H_{tan}^{m-1}}
	)^\frac12\\
	&~~~~+ \Vert u \Vert_{H_{tan}^m}^\frac12
	(\Vert \partial_h u \Vert_{H_{tan}^{m-1}}
	+
	\Vert \omega_h^u \Vert_{H_{tan}^{m-1}}
	)^\frac12
	\Vert \partial_h u \Vert_{H_{tan}^m}^\frac32
	(\Vert \partial_h \partial_h u \Vert_{H_{tan}^{m-1}}
	+
	\Vert \partial_h \omega_h^u \Vert_{H_{tan}^{m-1}}
	)^\frac12\\
		&\lesssim
\Vert u \Vert_{H_{tan}^m}^\frac12
	(\Vert u \Vert_{H_{tan}^{m}}
	+
	\Vert \omega_h^u \Vert_{H_{tan}^{m-1}}
	)^\frac12
	\Vert \partial_h u \Vert_{H_{tan}^m}^\frac32
	(\Vert \partial_h u \Vert_{H_{tan}^{m}}
	+
	\Vert \partial_h \omega_h^u \Vert_{H_{tan}^{m-1}}
	)^\frac12\\
&\lesssim
 ( \Vert u \Vert_{H_{tan}^m}+\Vert \omega_h^u \Vert_{H_{tan}^{m-1}} )
 (\Vert \partial_h u \Vert_{H_{tan}^m}^2
  +\Vert \partial_h \omega_h^u \Vert^{2}_{H_{tan}^{m-1}}).
\end{aligned}
\end{equation}
Similarly, it is easy to check that
\begin{equation}\label{3118}
\begin{aligned}
I_2 
&\lesssim
	\Vert u \Vert_{H_{tan}^m} ^{\frac{1}{2}} \Vert B \Vert_{H_{tan}^m} ^{\frac{1}{2}} \Vert \partial_h u \Vert_{H_{tan}^m} ^{\frac{1}{2}} \Vert \partial_3 B \Vert_{H_{tan}^m} ^{\frac{3}{2}}
	+
	\Vert B \Vert_{H^m_{tan}} \Vert \partial_h u \Vert_{H^m_{tan}}
	\Vert \partial_3 B \Vert_{H^m_{tan}}^{\frac{1}{2}} \Vert \partial_1 B \Vert_{H^{m-1}_{tan}}^{\frac{1}{2}}\\
&\lesssim
\Vert (u, B)\Vert_{H^m_{tan}}
(\Vert (\partial_h u, \partial_3 B)\Vert_{H^m_{tan}}^2
   +\Vert \partial_1 B \Vert_{H^{m-1}_{tan}}^{2}).
\end{aligned}
\end{equation}
Now, let us deal with the difficult term $I_3$
due to the disappearance of dissipative structure of magnetic field
in the $x_2$ direction.
Thus, one splits the term $I_3$ into three terms as follows:
\begin{equation*}
	\begin{aligned}
	&I_3= -\sum\limits_{i=1}^2 \sum\limits_{\substack{k+\ell=m\\k\neq0}} C_{k,\ell}
	\int_{\mathbb{R}_+^3} \partial_i^k u \cdot \partial_i^\ell \nabla B \cdot \partial_i^m B dx\\
	&~~~=-\sum\limits_{i=1}^2 \sum\limits_{\substack{k+\ell=m\\k\neq0}} C_{k,\ell}
	\int_{\mathbb{R}_+^3} \partial_i^k u_1  \partial_i^\ell \partial_1 B \cdot \partial_i^m B dx -\sum\limits_{i=1}^2 \sum\limits_{\substack{k+\ell=m\\k\neq0}} C_{k,\ell}
	\int_{\mathbb{R}_+^3} \partial_i^k u_2 \partial_i^\ell \partial_2 B \cdot \partial_i^m B dx \\
	&~~~~~~-\sum\limits_{i=1}^2 \sum\limits_{\substack{k+\ell=m\\k\neq0}} C_{k,\ell}
	\int_{\mathbb{R}_+^3} \partial_i^k u_3  \partial_i^\ell \partial_3 B \cdot \partial_i^m B dx \\
	&~~\overset{def}{=}I_{3,1}+I_{3,2}+I_{3,3}.
\end{aligned}
\end{equation*}
By Lemma \ref{lemma1}, we obtain
\begin{equation}\label{3101}
	\begin{aligned}
	I_{3,1}&=-\sum\limits_{i=1}^2 \sum\limits_{\substack{k+\ell=m\\k\neq0,\ell \neq 0}} C_{k,\ell}
	\int_{\mathbb{R}_+^3} \partial_i^k u_1  \partial_i^\ell \partial_1 B \cdot \partial_i^m B dx
	-
	\sum\limits_{i=1}^2 \int_{\mathbb{R}_+^3} \partial_i^m u_1  \partial_1 B \cdot \partial_i^m B dx \\
	&\lesssim
\sum\limits_{\substack{k+\ell=m\\k\neq0,\ell \neq 0}} \Vert \partial_h^k u_1 \Vert_{L^2}^{\frac{1}{4}}
	\Vert \partial_1 \partial_h^k u_1 \Vert_{L^2}^{\frac{1}{4}}
	\Vert \partial_2 \partial_h^k u_1 \Vert_{L^2}^{\frac{1}{4}}
	\Vert \partial_1 \partial_2 \partial_h^k u_1 \Vert_{L^2}^{\frac{1}{4}}
	\Vert \partial_h^\ell \partial_1 B \Vert_{L^2}
	\Vert \partial_h^m B \Vert_{L^2}^{\frac{1}{2}}
	\Vert \partial_3 \partial_h^m B \Vert_{L^2}^{\frac{1}{2}}
	\\
	&~~~~+ \Vert \partial_h^m u_1 \Vert_{L^2}^{\frac{1}{2}}
	\Vert \partial_1 \partial_h^m u_1 \Vert_{L^2}^{\frac{1}{2}}
	\Vert \partial_1 B \Vert_{L^2}^{\frac{1}{2}}
	\Vert \partial_2 \partial_1 B  \Vert_{L^2}^{\frac{1}{2}}
	\Vert \partial_h^m B \Vert_{L^2}^{\frac{1}{2}}
	\Vert \partial_3 \partial_h^m B \Vert_{L^2}^{\frac{1}{2}} \\
	&\lesssim
	\Vert (u,B) \Vert_{H^m_{tan}}
	(
	\Vert (\partial_h u,\partial_3 B) \Vert_{H^m_{tan}}^2
	+
	\Vert \partial_1 B \Vert_{H^{m-1}_{tan}}^2
	).
\end{aligned}
\end{equation}
The estimate of $I_{3,3}$ is easy since $I_{3,3}$ involves dissipation term $\partial_3 B$. Thus, by Lemma \ref{lemma1}, we obtain
\begin{equation}\label{3102}
	\begin{aligned}
	I_{3,3} & \lesssim \sum\limits_{\substack{k+\ell=m\\k\neq0}} \Vert \partial_h^k u_3 \Vert_{L^2}^{\frac{1}{2}}
	\Vert \partial_1 \partial_h^k u_3 \Vert_{L^2}^{\frac{1}{2}}
	\Vert \partial_h^\ell \partial_3 B \Vert_{L^2}^{\frac{1}{2}}
	\Vert \partial_2 \partial_h^\ell \partial_3 B \Vert_{L^2}^{\frac{1}{2}}
	\Vert \partial_h^m B \Vert_{L^2}^{\frac{1}{2}}
	\Vert \partial_3 \partial_h^m B \Vert_{L^2}^{\frac{1}{2}}\\
	&\lesssim \Vert u \Vert_{H^m_{tan}}^{\frac{1}{2}}
	\Vert B \Vert_{H^m_{tan}}^{\frac{1}{2}}
	\Vert \partial_h u \Vert_{H^m_{tan}}^{\frac{1}{2}}
	\Vert \partial_3 B \Vert_{H^m_{tan}}^{\frac{3}{2}}
	\\
	&\lesssim
	\Vert (u,B) \Vert_{H^m_{tan}}
	\Vert (\partial_h u,\partial_3 B) \Vert_{H^m_{tan}}^2.
\end{aligned}
\end{equation}
Due to the disappearance of dissipative structure of magnetic field
in the $x_2$ direction, we should deal with $I_{3,2}$ term carefully.
\begin{equation*}
	\begin{aligned}
	I_{3,2}&=-\sum\limits_{i=1}^2 \sum\limits_{\substack{k+\ell=m\\k\neq0,1,m}} C_{k,\ell}
	\int_{\mathbb{R}_+^3} \partial_i^k u_2 \partial_i^\ell \partial_2 B \cdot \partial_i^m B dx \\
	&~~~~
	-m\sum\limits_{i=1}^2\int_{\mathbb{R}_+^3} \partial_i u_2 \partial_i^{m-1} \partial_2 B \cdot \partial_i^m B dx
	- \sum\limits_{i=1}^2 \int_{\mathbb{R}_+^3} \partial_i^m u_2 \partial_2 B \cdot \partial_i^m B dx \\
	&\overset{def}{=}I_{3,2,1} + I_{3,2,2} + I_{3,2,3}.
\end{aligned}
\end{equation*}
By Lemma \ref{lemma1}, it holds
\begin{equation}\label{3105}
	\begin{aligned}
	 I_{3,2,1} & \lesssim  \sum\limits_{\substack{k+\ell=m\\k\neq0,1,m}}
	 \Vert \partial_h^k u_2 \Vert_{L^2}^{\frac{1}{2}}
	 \Vert \partial_2 \partial_h^k u_2 \Vert_{L^2}^{\frac{1}{2}}
	 \Vert \partial_h^\ell \partial_2 B \Vert_{L^2}^{\frac{1}{2}}
	 \Vert \partial_1 \partial_h^\ell \partial_2 B \Vert_{L^2}^{\frac{1}{2}}
	 \Vert \partial_h^m B \Vert_{L^2}^{\frac{1}{2}}
	 \Vert \partial_3 \partial_h^m B \Vert_{L^2}^{\frac{1}{2}}\\
	 &\lesssim \Vert B \Vert_{H^m_{tan}}
	 \Vert \partial_h u \Vert_{H^m_{tan}}
	 \Vert \partial_3 B \Vert_{H^m_{tan}}^{\frac{1}{2}}
	 \Vert \partial_1 B \Vert_{H^{m-1}_{tan}}^{\frac{1}{2}}\\
	 &
	 \lesssim
	 \Vert B \Vert_{H^m_{tan}}
	 (\Vert (\partial_h u, \partial_3 B)\Vert_{H^m_{tan}}^2
	 +
	 \Vert \partial_1 B \Vert_{H^{m-1}_{tan}}^2
	 ),
\end{aligned}
\end{equation}
and
\begin{equation}\label{3106}
	\begin{aligned}
	I_{3,2,3} & \lesssim
	\Vert \partial_h^m u_2 \Vert_{L^2}^{\frac{1}{2}}
	\Vert \partial_2 \partial_h^m u_2 \Vert_{L^2}^{\frac{1}{2}}
	\Vert \partial_2 B \Vert_{L^2}^{\frac{1}{2}}
	\Vert \partial_1 \partial_2 B \Vert_{L^2}^{\frac{1}{2}}
	\Vert \partial_h^m B \Vert_{L^2}^{\frac{1}{2}}
	\Vert \partial_3 \partial_h^m B \Vert_{L^2}^{\frac{1}{2}}\\
	&\lesssim
	\Vert B \Vert_{H^m_{tan}}
	\Vert \partial_h u \Vert_{H^m_{tan}}
	\Vert \partial_3 B \Vert_{H^m_{tan}}
	 ^{\frac{1}{2}}
	 \Vert \partial_1 B \Vert_{H^{m-1}_{tan}}
	 ^{\frac{1}{2}}
	 \\
	 &\lesssim
	 \Vert B \Vert_{H^m_{tan}}
	 (\Vert (\partial_h u, \partial_3 B)\Vert_{H^m_{tan}}^2
	 +
	 \Vert \partial_1 B \Vert_{H^{m-1}_{tan}}^2
	 ).
\end{aligned}
\end{equation}
Due to the dissipation structure of magnetic field in the $x_1$
and $x_3$ direction, we apply the divergence-free condition $\nabla \cdot u=0$
to control the difficult term $\int_{\mathbb{R}_+^3} \partial_2 u_2 |\partial_2^m B|^2 dx$.
Thus, we have
\begin{equation*}
	\begin{aligned}
	I_{3,2,2}
	&=-m \int_{\mathbb{R}_+^3} \partial_1 u_2 \partial_1^{m-1} \partial_2 B \cdot \partial_1 ^m B dx
	-m\int_{\mathbb{R}_+^3} \partial_2 u_2 |\partial_2^m B|^2 dx
	\\
	&=-m \int_{\mathbb{R}_+^3} \partial_1 u_2 \partial_1^{m-1} \partial_2 B \cdot \partial_1 ^m B dx
	+m\int_{\mathbb{R}_+^3} \partial_1 u_1 |\partial_2^m B|^2 dx
	+m\int_{\mathbb{R}_+^3} \partial_3 u_3 |\partial_2^m B|^2 dx
	\\
	&\overset{def}{=}I_{3,2,2,1}+I_{3,2,2,2}+I_{3,2,2,3}.
\end{aligned}
\end{equation*}
By Lemma \ref{lemma1}, we obtain
\begin{equation}\label{3107}
	\begin{aligned}
	I_{3,2,2,1} & \lesssim
	\Vert \partial_1 u_2\Vert_{L^2}^{\frac{1}{4}}
	\Vert \partial_1 \partial_1 u_2\Vert_{L^2}^{\frac{1}{4}}
	\Vert \partial_2 \partial_1 u_2\Vert_{L^2}^{\frac{1}{4}}
	\Vert \partial_1 \partial_2 \partial_1 u_2\Vert_{L^2}^{\frac{1}{4}}
	\Vert \partial_1^{m-1} \partial_2 B \Vert_{L^2}
	\Vert \partial_1^m B \Vert_{L^2}^{\frac{1}{2}}
	\Vert \partial_3 \partial_1^m B \Vert_{L^2}^{\frac{1}{2}}\\
	&\lesssim
	\Vert B \Vert_{H^m_{tan}}
	(\Vert (\partial_h u, \partial_3 B)\Vert_{H^m_{tan}}^2
	+\Vert \partial_1 B\Vert_{H^{m-1}_{tan}}^2
	).
\end{aligned}
\end{equation}
Similarly, by integration by parts and Lemma \ref{lemma1}, it holds
\begin{equation}\label{3108}
	\begin{aligned}
	I_{3,2,2,3} & = -2m \int_{\mathbb{R}_+^3} u_3 \partial_2^m B \cdot \partial_3 \partial_2^m B dx \\
	&\lesssim \Vert u_3 \Vert_{L^2}^{\frac{1}{4}}
	\Vert \partial_1 u_3 \Vert_{L^2}^{\frac{1}{4}}
	\Vert \partial_2 u_3 \Vert_{L^2}^{\frac{1}{4}}
	\Vert \partial_1 \partial_2 u_3 \Vert_{L^2}^{\frac{1}{4}}
	\Vert \partial_2^m B \Vert_{L^2}^{\frac{1}{2}}
	\Vert \partial_3 \partial_2^m B \Vert_{L^2}^{\frac{1}{2}}
	\Vert \partial_3 \partial_2^m B \Vert_{L^2}\\
	&\lesssim \Vert(u, B)\Vert_{H^m_{tan}}
	\Vert (\partial_h u, \partial_3 B)\Vert_{H^m_{tan}}^2.
\end{aligned}
\end{equation}
By using the appearance of special term $\partial_1 u$ in $\eqref{eq1}_3$,
we can deduce the relation
\begin{equation*}\label{a01}
	\begin{aligned}
		I_{3,2,2,2}&=m\int_{\mathbb{R}_+^3} \partial_1 u_1 |\partial_2^m B|^2 dx
		=m \int_{\mathbb{R}_+^3} (\partial_t B_1 + u \cdot \nabla B_1 - \partial_3^2 B_1 - \varepsilon \Delta_h B_1 - B\cdot \nabla u_1) |\partial_2^m B|^2 dx \\
		&=
		m
		\frac{d}{dt} \int_{\mathbb{R}_+^3} B_1 |\partial_2^m B|^2 dx
		- 2m \int_{\mathbb{R}_+^3} B_1 \partial_2^m (-u\cdot \nabla B+ \partial_3^2 B + \varepsilon \Delta_h B + B \cdot \nabla u +\partial_1 u) \cdot \partial_2^m B dx
		\\
		&~~~~+m \int_{\mathbb{R}_+^3}
		(u \cdot \nabla B_1 - \partial_3^2 B_1 - \varepsilon \Delta_h B_1 - B\cdot \nabla u_1) |\partial_2^m B|^2 dx.
	\end{aligned}
\end{equation*}
Then, using Lemma \ref{lemma1} repeatedly, it is easy to check that
\begin{equation}\label{3103}
\begin{aligned}
	I_{3,2,2,2}
&
\lesssim
 \frac{d}{dt} \int_{\mathbb{R}_+^3} B_1 |\partial_2^m B|^2 dx
+
\varepsilon (\Vert B \Vert_{H^m_{tan}}
+
\Vert \omega_h^B \Vert_{H^{m-1}_{tan}}
)
\Vert \partial_h B \Vert_{H^m_{tan}}^2
\\
&~~~~
+
(
\Vert B \Vert_{H^m_{tan}}
+
\Vert \omega_h^B \Vert_{H^{m-1}_{tan}}
)
(
\Vert (\partial_h u,\partial_3 B) \Vert_{H^m_{tan}}^2
+
\Vert \partial_1 B \Vert_{H^{m-1}_{tan}}^2
)\\
&~~~~
+
(
\Vert (u,B) \Vert^2_{H^m_{tan}}
+
\Vert \omega_h^u \Vert^2_{H^{m-1}_{tan}}
)
(
\Vert (\partial_h u,\partial_3 B) \Vert^2_{H^m_{tan}}
+
\Vert (\partial_h \omega_h^u,\partial_1 B) \Vert^2_{H^{m-1}_{tan}}
).
	\end{aligned}
\end{equation}
The combination of estimates \eqref{3107}, \eqref{3108}
and \eqref{3103} yields directly
\begin{equation*}
\begin{aligned}
I_{3,2,2}
&
\lesssim
\frac{d}{dt} \int_{\mathbb{R}_+^3} B_1 |\partial_2^m B|^2 dx
+
\varepsilon
(\Vert B \Vert_{H^m_{tan}}
+
\Vert \omega_h^B \Vert_{H^{m-1}_{tan}})
\Vert \partial_h B \Vert_{H^m_{tan}}^2
\\
&~~~~+
(\Vert (u,B) \Vert_{H^m_{tan}}+\Vert \omega_h^B \Vert_{H^{m-1}_{tan}})
(\Vert (\partial_h u, \partial_3 B) \Vert_{H^m_{tan}}^2
+
\Vert \partial_1 B \Vert_{H^{m-1}_{tan}}^2
)\\
&~~~~
+
(
\Vert (u,B) \Vert_{H^m_{tan}}^2
+
\Vert \omega_h^u \Vert_{H^{m-1}_{tan}}^2
)
(
\Vert (\partial_h u,\partial_3 B) \Vert_{H^m_{tan}}^2
+
\Vert (\partial_h \omega_h^u, \partial_1 B) \Vert_{H^{m-1}_{tan}}^2
),
	\end{aligned}
\end{equation*}
which, together with estimates \eqref{3105} and \eqref{3106}, yields directly
\begin{equation}\label{3104}
\begin{aligned}
I_{3,2}
\lesssim
& \frac{d}{dt} \int_{\mathbb{R}_+^3} B_1 |\partial_2^m B|^2 dx
		+
		\varepsilon (\Vert B \Vert_{H^m_{tan}}+\Vert \omega_h^B \Vert_{H^{m-1}_{tan}})
		\Vert \partial_h B \Vert_{H^m_{tan}}^2\\
&+
(\Vert(u,B)\Vert_{H^m_{tan}}+\Vert \omega_h^B \Vert_{H^{m-1}_{tan}})
(\Vert(\partial_h u, \partial_3 B)\Vert_{H^m_{tan}}^2
 +\Vert \partial_1 B\Vert_{H^{m-1}_{tan}}^2)\\
 &
 +
 (
 \Vert (u,B) \Vert_{H^m_{tan}}^2
 +
 \Vert \omega_h^u \Vert_{H^{m-1}_{tan}}^2
 )
 (
 \Vert (\partial_h u,\partial_3 B) \Vert_{H^m_{tan}}^2
 +
 \Vert (\partial_h \omega_h^u, \partial_1 B) \Vert_{H^{m-1}_{tan}}^2
 ).
	\end{aligned}
\end{equation}
Summarizing the estimates  of $I_{3,1}$ through $I_{3,3}$
in \eqref{3101}, \eqref{3102} and \eqref{3104}, we have
\begin{equation}\label{3109}
\begin{aligned}
I_3\lesssim
& \frac{d}{dt} \int_{\mathbb{R}_+^3} B_1 |\partial_2^m B|^2 dx
		+
		\varepsilon (\Vert B \Vert_{H^m_{tan}}+\Vert \omega_h^B \Vert_{H^{m-1}_{tan}})
		\Vert \partial_h B \Vert_{H^m_{tan}}^2\\
&+
(\Vert(u,B)\Vert_{H^m_{tan}}+\Vert \omega_h^B \Vert_{H^{m-1}_{tan}})
(\Vert(\partial_h u, \partial_3 B)\Vert_{H^m_{tan}}^2
 +\Vert \partial_1 B\Vert_{H^{m-1}_{tan}}^2)\\
 &
 +
 (
 \Vert (u,B) \Vert_{H^m_{tan}}^2
 +
 \Vert \omega_h^u \Vert_{H^{m-1}_{tan}}^2
 )
 (
 \Vert (\partial_h u,\partial_3 B) \Vert_{H^m_{tan}}^2
 +
 \Vert (\partial_h \omega_h^u, \partial_1 B) \Vert_{H^{m-1}_{tan}}^2
 ).
	\end{aligned}
\end{equation}
Similarly, we use the Eq.\eqref{eq1} repeatedly to obtain
\begin{equation}\label{3119}
\begin{aligned}
	I_4
		\lesssim&
        \frac{d}{dt} \int_{\mathbb{R}_+^3} B_1 |\partial_2^m B_1|^2 dx
        +
        \frac{d}{dt}\int_{\mathbb{R}_+^3} \partial_2^m B_1 B_3 \partial_2^m B_3 dx
        +\varepsilon
        (\Vert B \Vert_{H^m_{tan}}
        +
        \Vert \omega_h^B \Vert_{H^{m-1}_{tan}})
        \Vert \partial_h B \Vert_{H^m_{tan}}^2\\
&+
(\Vert (u,B)\Vert_{H^m_{tan}}+\Vert (\omega_h^u,\omega_h^B) \Vert_{H^{m-1}_{tan}})
(\Vert (\partial_h u,\partial_3 B )\Vert_{H^m_{tan}}^2
+\Vert (\partial_h \omega_h^u, \partial_1 B) \Vert_{H^{m-1}_{tan}}^2)
\\
&
+
(
\Vert (u,B) \Vert_{H^m_{tan}}^2
+
\Vert \omega_h^u \Vert_{H^{m-1}_{tan}}^2
)
(
\Vert (\partial_h u,\partial_3 B) \Vert_{H^m_{tan}}^2
+
\Vert (\partial_h \omega_h^u, \partial_1 B) \Vert_{H^{m-1}_{tan}}^2
),
\end{aligned}
\end{equation}
Integrating \eqref{3124} and \eqref{eqE1} in time, we obtain
\begin{align}\label{e1}
	E_{1}(t) \lesssim E(0)+\int_{0}^t \left(
	I_1(\tau)+I_2(\tau)+I_3(\tau)+I_4(\tau)
	\right) d \tau.
\end{align}
Applying the estimates \eqref{3117} and \eqref{3118}, then it holds
\begin{equation}\label{3120}
\begin{aligned}
\!\!\int_{0}^t \!\!I_1(\tau) d \tau
		&\lesssim
		\!\!\sup_{0 \le \tau \le t}
		(\Vert u(\tau) \Vert_{H_{tan}^m}
		\!+\!\Vert \omega_h^u(\tau) \Vert_{H_{tan}^{m-1}})
		\!\! \int_{0}^{t}\!\!
		(\Vert \partial_h u(\tau) \Vert^2_{H_{tan}^m}
          \!+\!\Vert \partial_h \omega_h^u(\tau) \Vert_{H_{tan}^{m-1}}^2)d \tau
		\lesssim E(t)^\frac{3}{2},
	\end{aligned}
\end{equation}
and
\begin{equation}\label{3121}
\begin{aligned}
\int_{0}^{t}  I_2(\tau) d\tau
\lesssim
		\sup_{0 \le \tau \le t}
		\Vert (u, B)(\tau)\Vert_{H^m_{tan}}
		\int_{0}^{t} (\Vert (\partial_h u, \partial_3 B)(\tau)\Vert_{H^m_{tan}}^2  +
        \Vert \partial_1 B(\tau) \Vert_{H^{m-1}_{tan}}^{2}) d\tau
\lesssim  E(t)^\frac{3}{2}+G(t)^\frac{3}{2}.
\end{aligned}
\end{equation}
Similarly, due to the estimates \eqref{3109} \eqref{3119}, we have
\begin{equation}\label{3122}
\int_{0}^{t} I_3 (\tau) d\tau + \int_{0}^{t} I_4 (\tau) d\tau
\lesssim E(0)^\frac{3}{2} + E(t)^\frac{3}{2}+ G(t)^\frac{3}{2}
+E(t)^2+G(t)^2,
\end{equation}
where, by Lemma \ref{lemma1},  we have used the following fact
\begin{equation*}
	\begin{aligned}
		&
		\left|
		\int_{0}^{t} \frac{d}{d\tau}
		\int_{\mathbb{R}_+^3} B_1 |\partial_2^m B|^2 dx d\tau \right|
		+
		\left|
		\int_{0}^{t} \frac{d}{d\tau}
		\int_{\mathbb{R}_+^3} B_1 |\partial_2^m B_1|^2 dx d\tau \right|
		+
         \left| \int_{0}^{t} \frac{d}{d\tau}
         \int_{\mathbb{R}_+^3} \partial_2^m B_1 B_3 \partial_2^m B_3 dx d\tau \right|
         \\
		\lesssim
		&\Vert B(0) \Vert_{L^\infty}
		\Vert \partial_2^m B(0) \Vert^2_{L^2}
		+
		\Vert B(t) \Vert_{L^\infty}
		\Vert \partial_2^m B(t) \Vert^2_{L^2}
		\\
		\lesssim&
(
		\Vert B(0) \Vert_{H^m_{tan}}
		+
		\Vert \omega_h^B(0) \Vert_{H^{m-1}_{tan}}
)
		\Vert B(0) \Vert_{H^m_{tan}}^2
+
(
		\Vert B(t) \Vert_{H^m_{tan}}
		+
		\Vert \omega_h^B(t) \Vert_{H^{m-1}_{tan}}
)
		\Vert B(t) \Vert_{H^m_{tan}}^2\\
		\lesssim &
		E(0)^\frac{3}{2}
		+E(t)^\frac{3}{2}.
	\end{aligned}
\end{equation*}
Substituting the estimates \eqref{3120}-\eqref{3122} into
\eqref{e1}, then we have
\begin{equation*}
	E_{1}(t)
   \lesssim E(0)+E(0)^\frac{3}{2}+E(t)^\frac{3}{2}+G(t)^\frac{3}{2}+E(t)^2+G(t)^2.
\end{equation*}
Therefore, we complete the proof of estimate \eqref{eq1-1}.
		
\subsection{Estimate of normal derivative}\label{Estimateofnormalderivative}
	
In this subsection, we will give the proof for estimate \eqref{eq1-0}.
For notational convenience, we drop the superscript $\ep$ throughout this section.
Applying $\nabla \times $ differential operator to $\eqref{eq1}_1$
and $\eqref{eq1}_3$ respectively, we obtain
		\begin{equation}\label{eq2}
			\left\{ \begin{array}{*{2}{ll}}
				\partial_t \omega_h^u + u \cdot \nabla \omega_h^u -\Delta_h \omega_h^u -\varepsilon \partial_3^2 \omega_h^u = B\cdot \nabla \omega_h^B + \omega^u \cdot \nabla u_h - \omega^B \cdot \nabla B_h + \partial_1 \omega_h^B \quad  & {\rm in}~~\mathbb{R}_+^3,\\
				\partial_t \omega_h^B + u \cdot \nabla \omega_h^B -\varepsilon \Delta_h \omega_h^B -\partial_3^2 \omega_h^B =  \mathcal{A}(u,B)
				+
				B\cdot \nabla \omega_h^u
				+ \partial_1 \omega_h^u \quad & {\rm in}~~\mathbb{R}_+^3,
			\end{array}
		\right.
		\end{equation}
with the boundary condition
$$
\omega_h^u=0, \quad \omega_h^B=0 ~~ {\rm on}~~\mathbb{R}^2 \times \{x_3=0\}.
$$
Here the function $\mathcal{A}(u,B)$ is defined by
	\begin{equation}\label{3216}
		\mathcal{A}(u,B)\overset{def}{=}\begin{pmatrix}
			\partial_2 B \cdot \nabla u_3 -\partial_3 B \cdot \nabla u_2  -\partial_2 u \cdot \nabla B_3
			+\partial_3 u \cdot \nabla B_2 \\
			\partial_3 B \cdot \nabla u_1 -\partial_1 B \cdot \nabla u_3 -\partial_3 u \cdot \nabla B_1
			+\partial_1 u \cdot \nabla B_3
		\end{pmatrix}.
	\end{equation}
The $L^2$-estimate of $(\omega_h^u,\omega_h^B)$ implies that
\begin{equation}\label{eqL2}
	\begin{aligned}
		&~~~~\frac{1}{2}\frac{d}{dt} \| (\omega_h^u,\omega_h^B)\|^2_{L^2}
		+ \| (\partial_h w_h^u,\partial_3 w_h^B) \|^2_{L^2}
		+ \varepsilon \| (\partial_3 w_h^u,\partial_h w_h^B) \|^2_{L^2}
		\\
		&=
		\int_{\mathbb{R}_+^3}
		B\cdot \nabla \omega_h^B
		\cdot w_h^u dx
		 + \int_{\mathbb{R}_+^3} \omega^u \cdot \nabla u_h \cdot w_h^u dx
		  - \int_{\mathbb{R}_+^3} \omega^B \cdot \nabla B_h \cdot w_h^u dx
		   \\
		   &~~~~+
		   \int_{\mathbb{R}_+^3} \mathcal{A}(u,B) \cdot \omega_h^B dx
		   +
		   \int_{\mathbb{R}_+^3} B \cdot \nabla \omega_h^u \cdot \omega_h^B dx
		   \\
		   &\overset{def}{=}II_1+II_2+II_3+II_4+II_5,
	\end{aligned}
\end{equation}
where we have used the basic fact
\begin{equation*}
	\int_{\mathbb{R}_+^3} \omega_h^u \cdot \partial_1  \omega_h^B dx
	+
	\int_{\mathbb{R}_+^3} \partial_1 \omega_h^u \cdot \omega_h^B dx
	=0.
\end{equation*}
We split $II_1$ into the following terms:
\begin{equation*}
	\begin{aligned}
		II_1=
		\int_{\mathbb{R}_+^3} B_1 \partial_1 \omega_h^B \cdot \omega_h^u dx
		+
		\int_{\mathbb{R}_+^3} B_2 \partial_2 \omega_h^B \cdot \omega_h^u dx
		+
		\int_{\mathbb{R}_+^3} B_3 \partial_3 \omega_h^B \cdot \omega_h^u dx
		\overset{def}{=}
		II_{1,1}+II_{1,2}+II_{1,3}.
	\end{aligned}
\end{equation*}
By Lemma \ref{lemma1}, we have
\begin{equation*}
	\begin{aligned}
		II_{1,1}
		&=
		\int_{\mathbb{R}_+^3} B_1 \partial_1 \omega_1^B  \omega_1^u dx
		+
		\int_{\mathbb{R}_+^3} B_1 \partial_1 \omega_2^B  \omega_2^u dx
		\\
		&=\int_{\mathbb{R}_+^3}
		B_1 \partial_{12}^2 B_3 \omega^u_1 dx
		- \int_{\mathbb{R}_+^3}
		B_1 \partial_{13}^2 B_2 \omega^u_1 dx
		+
		\int_{\mathbb{R}_+^3}
		B_1 \partial_{13}^2 B_1 \omega^u_2 dx
		- \int_{\mathbb{R}_+^3}
		B_1 \partial_1^2 B_3 \omega^u_2 dx\\
		&\lesssim
		\| B_1 \|^\frac{1}{4}_{L^2}
		\| \partial_1 B_1 \|^\frac{1}{4}_{L^2}
		\| \partial_3 B_1 \|^\frac{1}{4}_{L^2}
		\| \partial_1 \partial_3 B_1 \|^\frac{1}{4}_{L^2}
		\| \partial_{12}^2 B_3 \|_{L^2}
		\| \omega_1^u \|^\frac{1}{2}_{L^2}
		\| \partial_2 \omega_1^u \|^\frac{1}{2}_{L^2}
		\\
		&~~~~
		+\| B_1 \|^\frac{1}{4}_{L^2}
		\| \partial_1 B_1 \|^\frac{1}{4}_{L^2}
		\| \partial_3 B_1 \|^\frac{1}{4}_{L^2}
		\| \partial_1 \partial_3 B_1 \|^\frac{1}{4}_{L^2}
		\| \partial_{13}^2 B_2 \| _{L^2}
		\| \omega_1^u \|^\frac{1}{2}_{L^2}
		\| \partial_2 \omega_1^u \|^\frac{1}{2}_{L^2}
		\\
		&~~~~+
		\Vert B_1 \Vert^\frac{1}{4}_{L^2}
		\Vert \partial_1 B_1 \Vert^\frac{1}{4}_{L^2}
		\Vert \partial_3 B_1 \Vert^\frac{1}{4}_{L^2}
		\Vert \partial_1 \partial_3 B_1 \Vert^\frac{1}{4}_{L^2}
		\| \partial_{13}^2 B_1 \|_{L^2}
		\| \omega_2^u \|^\frac{1}{2}_{L^2}
		\| \partial_2 \omega_2^u \|^\frac{1}{2}_{L^2}
		\\
		&~~~~
		+
		\Vert B_1 \Vert^\frac{1}{4}_{L^2}
		\Vert \partial_1 B_1 \Vert^\frac{1}{4}_{L^2}
		\Vert \partial_3 B_1 \Vert^\frac{1}{4}_{L^2}
		\Vert \partial_1 \partial_3 B_1 \Vert^\frac{1}{4}_{L^2}
		\| \partial_1^2 B_3 \| _{L^2}
		\| \omega_2^u \|^\frac{1}{2}_{L^2}
		\| \partial_2 \omega_2^u \|^\frac{1}{2}_{L^2}
		\\
		&
		\lesssim
		(\Vert B \Vert_{H^m_{co}}+\Vert \omega_h^u \Vert_{H^{m-1}_{co}})
        (\Vert \partial_3 B \Vert^2_{H^m_{co}}
		 +\Vert (\partial_h \omega_h^u, \partial_1 B)\Vert^2_{H^{m-1}_{co}}).
	\end{aligned}
\end{equation*}
Similarly, using integration by parts, $\nabla \cdot B=0$ and Lemma \ref{lemma1}, we have
\begin{equation*}
	\begin{aligned}
		II_{1,2}
		=&
		\int_{\mathbb{R}_+^3} B_2 \partial_2 \omega_1^B  \omega_1^u dx
		+
		\int_{\mathbb{R}_+^3} B_2 \partial_2 \omega_2^B  \omega_2^u dx
		\\
		=&\int_{\mathbb{R}_+^3}
		B_2 \partial_2^2 B_3 \omega^u_1 dx
		-
		\int_{\mathbb{R}_+^3}
		B_2 \partial_{23}^2 B_2 \omega^u_1 dx
		+ \int_{\mathbb{R}_+^3}
		B_2 \partial_{23}^2 B_1 \omega^u_2 dx
        - \int_{\mathbb{R}_+^3}
		B_2 \partial_{12}^2 B_3 \omega^u_2 dx\\
		=&
\int_{\mathbb{R}_+^3} \partial_1 B_1  \partial_2 B_3 \omega^u_1 dx
+
\int_{\mathbb{R}_+^3} \partial_3 B_3  \partial_2 B_3 \omega^u_1 dx
-\int_{\mathbb{R}_+^3} B_2  \partial_2 B_3 \partial_2 \omega^u_1 dx\\
		&-
		\int_{\mathbb{R}_+^3}
		B_2 \partial_{23}^2 B_2 \omega^u_1 dx
		+ \int_{\mathbb{R}_+^3}
		B_2 \partial_{23}^2 B_1 \omega^u_2 dx
        - \int_{\mathbb{R}_+^3}
		B_2 \partial_{12}^2 B_3 \omega^u_2 dx\\
		\lesssim&
		(\Vert B \Vert_{H^m_{co}}+\Vert \omega_h^u \Vert_{H^{m-1}_{co}})
        (\Vert \partial_3 B \Vert^2_{H^m_{co}}
		 +\Vert (\partial_h \omega_h^u, \partial_1 B)\Vert^2_{H^{m-1}_{co}}),
	\end{aligned}
\end{equation*}
and
\begin{equation*}
	\begin{aligned}
	II_{1,3}
	&\lesssim
	\Vert B_3 \Vert^\frac{1}{4}_{L^2}
	\Vert \partial_1 B_3 \Vert^\frac{1}{4}_{L^2}
	\Vert \partial_3 B_3 \Vert^\frac{1}{4}_{L^2}
	\Vert \partial_1 \partial_3 B_3 \Vert^\frac{1}{4}_{L^2}
	\Vert \partial_3 \omega_h^B \Vert _{L^2}
	\Vert \omega_h^u \Vert^\frac{1}{2}_{L^2}
	\Vert \partial_2 \omega_h^u \Vert^\frac{1}{2}_{L^2}
	\\
	&\lesssim
	(\Vert B \Vert_{H_{co}^m}+\Vert \omega_h^u \Vert_{H_{co}^{m-1}})
	(\Vert \partial_3 B \Vert_{H_{co}^m}^2
	 +\Vert (\partial_h \omega_h^u,\partial_3 \omega_h^B )\Vert^2_{H_{co}^{m-1}}).
\end{aligned}
\end{equation*}
Summarizing the bounds for $II_{1,1}$, $II_{1,2}$ and $II_{1,3}$, we obtain
\begin{equation}\label{3201}
II_1 \lesssim
(\Vert B \Vert_{H_{co}^m}+\Vert \omega_h^u \Vert_{H_{co}^{m-1}})
	(\Vert \partial_3 B \Vert_{H_{co}^m}^2
+\Vert (\partial_h \omega_h^u,\partial_3 \omega_h^B, \partial_1 B )\Vert^2_{H_{co}^{m-1}}).
\end{equation}
Similarly, it is easy to check that
\begin{equation}\label{3202}
\begin{aligned}
II_2
&\lesssim (\Vert u \Vert_{H^m_{co}}
     +\Vert \omega_h^u \Vert_{H^{m-1}_{co}})
	(\Vert \partial_h u \Vert_{H^m_{co}}^2
	 +\Vert \partial_h \omega_h^u \Vert_{H^{m-1}_{co}}^2),\\
II_3
&\lesssim
(\Vert B \Vert_{H_{co}^m}+\Vert \omega_h^u \Vert_{H_{co}^{m-1}})
	(\Vert \partial_3 B \Vert_{H_{co}^m}^2
+\Vert (\partial_h \omega_h^u, \partial_1 B )\Vert^2_{H_{co}^{m-1}}),\\
II_4
&\lesssim
(\Vert (u, B)\Vert_{H_{co}^m}+\Vert (\omega_h^u, \omega_h^B) \Vert_{H_{co}^{m-1}})
	(\Vert (\partial_h u, \partial_3 B )\Vert_{H_{co}^m}^2
+\Vert (\partial_h \omega_h^u, \partial_3 \omega_h^B,\partial_1 B)\Vert^2_{H_{co}^{m-1}}).
\end{aligned}
\end{equation}
Using integration by parts and Lemma \ref{lemma1}, we have
\begin{equation}\label{3203}
	\begin{aligned}
	II_5	
	=&\int_{\mathbb{R}_+^3} B_h \cdot \partial_h \omega_h^u \cdot \omega_h^B dx
	+
	\int_{\mathbb{R}_+^3} B_3 \partial_3 \omega_h^u \cdot \omega_h^B dx
	\\
	=&
	\int_{\mathbb{R}_+^3} B_h \cdot \partial_h \omega_h^u \cdot \omega_h^B dx
	-
	\int_{\mathbb{R}_+^3} \partial_3 B_3  \omega_h^u \cdot \omega_h^B dx
	-
	\int_{\mathbb{R}_+^3} B_3  \omega_h^u \cdot \partial_3 \omega_h^B dx
	\\
	\lesssim&
	\Vert B_h \Vert^\frac{1}{4}_{L^2}
	\Vert \partial_1 B_h \Vert^\frac{1}{4}_{L^2}
	\Vert \partial_2 B_h \Vert^\frac{1}{4}_{L^2}
	\Vert \partial_1 \partial_2 B_h \Vert^\frac{1}{4}_{L^2}
	\Vert \partial_h \omega_h^u \Vert_{L^2}
	\Vert \omega_h^B \Vert^\frac{1}{2}_{L^2}
	\Vert \partial_3 \omega_h^B \Vert^\frac{1}{2}_{L^2}
	\\
	&
	+
	\Vert \partial_3 B_3 \Vert^\frac{1}{2}_{L^2}
	\Vert \partial_1 \partial_3 B_3 \Vert^\frac{1}{2}_{L^2}
	\Vert \omega_h^u \Vert^\frac{1}{2}_{L^2}
	\Vert \partial_2 \omega_h^u \Vert^\frac{1}{2}_{L^2}
	\Vert \omega_h^B \Vert^\frac{1}{2}_{L^2}
	\Vert \partial_3 \omega_h^B \Vert^\frac{1}{2}_{L^2}
	\\
	&
	+
	\Vert B_3 \Vert^\frac{1}{4}_{L^2}
	\Vert \partial_1 B_3 \Vert^\frac{1}{4}_{L^2}
	\Vert \partial_3 B_3 \Vert^\frac{1}{4}_{L^2}
	\Vert \partial_1 \partial_3 B_3 \Vert^\frac{1}{4}_{L^2}
	\Vert \omega_h^u \Vert^\frac{1}{2}_{L^2}
	\Vert \partial_2 \omega_h^u \Vert^\frac{1}{2}_{L^2}
	\Vert \partial_3 \omega_h^B \Vert_{L^2}
	\\
\lesssim&
(\Vert B\Vert_{H_{co}^m}+\Vert (\omega_h^u, \omega_h^B) \Vert_{H_{co}^{m-1}})
	(\Vert \partial_3 B \Vert_{H_{co}^m}^2
+\Vert (\partial_h \omega_h^u, \partial_3 \omega_h^B,\partial_1 B)\Vert^2_{H_{co}^{m-1}}).
\end{aligned}
\end{equation}
Substituting the estimates \eqref{3201}, \eqref{3202}
and \eqref{3203} into \eqref{eqL2}, we have
\begin{equation}\label{3204}
\begin{aligned}
&\frac{1}{2}\frac{d}{dt} \| (\omega_h^u,\omega_h^B)\|^2_{L^2}
		+ \| (\partial_h w_h^u,\partial_3 w_h^B) \|^2_{L^2}
		+ \varepsilon \| (\partial_3 w_h^u,\partial_h w_h^B) \|^2_{L^2}\\
\lesssim
&
(\Vert (u, B)\Vert_{H_{co}^m}+\Vert (\omega_h^u, \omega_h^B) \Vert_{H_{co}^{m-1}})
	(\Vert (\partial_h u, \partial_3 B )\Vert_{H_{co}^m}^2
+\Vert (\partial_h \omega_h^u, \partial_3 \omega_h^B,\partial_1 B)\Vert^2_{H_{co}^{m-1}}).
\end{aligned}
\end{equation}
Next, we turn to the higher order of energy estimates for $(\omega_h^u,\omega_h^B)$.
    For every $m\ge 3$, applying $Z^\alpha$ ($1 \le |\alpha| \le m-1$) to \eqref{eq2}$_1$ and \eqref{eq2}$_2$, we obtain
    \begin{equation*}
	\begin{aligned}
    	&~~~~\partial_t Z^\alpha \omega_h^u + u \cdot \nabla Z^\alpha \omega_h^u - \Delta_h Z^\alpha \omega_h^u - \varepsilon \partial_3^2 Z^\alpha \omega_h^u \\
    	&= - [Z^\alpha, u \cdot \nabla]\omega_h^u + \varepsilon [Z^\alpha,  \partial_3^2]\omega_h^u + Z^\alpha (B\cdot \nabla \omega_h^B + \omega^u \cdot \nabla u_h - \omega^B \cdot \nabla B_h + \partial_1 \omega_h^B),
    \end{aligned}
\end{equation*}
and
\begin{equation*}
	\begin{aligned}
	&~~~~\partial_t Z^\alpha \omega_h^B + u \cdot \nabla Z^\alpha \omega_h^B
- \varepsilon  \Delta_h Z^\alpha \omega_h^B - \partial_3^2 Z^\alpha \omega_h^B \\
	&= - [Z^\alpha, u \cdot \nabla]\omega_h^B + [Z^\alpha,  \partial_3^2]\omega_h^B + Z^\alpha (\mathcal{A} (u,B)+B \cdot \nabla \omega_h^u+\partial_1 \omega_h^u).
\end{aligned}
\end{equation*}
Multiplying the above equation by $Z^\alpha \omega_h^u$
and $Z^\alpha \omega_h^B$ respectively and integrating over
$\mathbb{R}_+^3$, we obtain
	\begin{equation}\label{eqE2}
		\begin{aligned}
			&~~~~\frac{1}{2} \frac{d}{dt} \Vert (Z^\alpha \omega_h^u,Z^\alpha \omega_h^B) \Vert_{L^2}^2 + \Vert (\partial_h Z^\alpha \omega_h^u,\partial_3 Z^\alpha \omega_h^B) \Vert_{L^2}^2 + \varepsilon \Vert (\partial_3 Z^\alpha \omega_h^u,\partial_h Z^\alpha \omega_h^B) \Vert_{L^2}^2 \\
			&= - \int_{\mathbb{R}_+^3} [Z^\alpha, u \cdot \nabla]\omega_h^u \cdot Z^\alpha \omega_h^u dx
			+ \varepsilon \int_{\mathbb{R}_+^3} [Z^\alpha, \partial_3^2]\omega_h^u \cdot Z^\alpha \omega_h^u dx +\int_{\mathbb{R}_+^3} Z^\alpha(B \cdot \nabla\omega_h^B) \cdot Z^\alpha \omega_h^u dx
			\\
			& ~~~~
			+ \int_{\mathbb{R}_+^3} Z^\alpha(\omega^u \cdot \nabla u_h) \cdot Z^\alpha \omega_h^u dx
			-\int_{\mathbb{R}_+^3} Z^\alpha(\omega^B \cdot \nabla B_h) \cdot Z^\alpha \omega_h^u dx
			\\
			& ~~~~
			- \int_{\mathbb{R}_+^3} [Z^\alpha, u \cdot \nabla]\omega_h^B \cdot Z^\alpha \omega_h^B dx
			+ \int_{\mathbb{R}_+^3} [Z^\alpha, \partial_3^2]\omega_h^B \cdot Z^\alpha \omega_h^B dx
			\\
			& ~~~~
			+ \int_{\mathbb{R}_+^3} Z^\alpha \mathcal{A}(u,B) \cdot Z^\alpha \omega_h^B dx
			+\int_{\mathbb{R}_+^3} Z^\alpha(B \cdot \nabla\omega_h^u) \cdot Z^\alpha \omega_h^B dx
			\\
			&\overset{def}{=}
III_1+III_2+III_3+III_4+III_5+III_6+III_7+III_8+III_9,
		\end{aligned}
	\end{equation}
where we have used the basic fact
\begin{equation*}
\int_{\mathbb{R}_+^3} Z^\alpha \omega_h^u \cdot \partial_1 Z^\alpha \omega_h^B dx
+
\int_{\mathbb{R}_+^3} \partial_1 Z^\alpha \omega_h^u \cdot Z^\alpha \omega_h^B dx=0.
\end{equation*}
In the sequence, we will give the estimate from
$III_1$ to $III_9$.
\textbf{Deal with the term $III_1$.}
Obviously, it holds
\begin{equation*}
III_1=-
\int_{\mathbb{R}_+^3} u_3 [Z^\alpha , \partial_3] \omega_h^u \cdot Z^\alpha \omega_h^u dx
-\sum\limits_{\substack{\beta+\gamma=\alpha\\\beta \neq 0}} C_{\beta,\gamma} \int_{\mathbb{R}_+^3} Z^\beta u \cdot Z^\gamma \nabla \omega_h^u \cdot Z^\alpha \omega_h^u dx
	\overset{def}{=}III_{1,1}+III_{1,2}.
\end{equation*}
Obviously, the term $III_{1,1}$ vanishes if $\alpha_3 =0$.
Hence we only consider the case of $\alpha_3 \ge 1$.
It is clear that
\begin{align*}
	[Z_3,\partial_3]f
	&= Z_3 \partial_3 f - \partial_3 Z_3 f
	=
	\varphi \partial_3^2 f
	-
	 \partial_3 (\varphi \partial_3 f)
	=
	\varphi \partial_3^2 f
	- \varphi' \partial_3 f
	- \varphi \partial_3^2 f
	= -\varphi' \partial_3 f,
\end{align*}
and thus
\begin{align*}
	[Z_3^2,\partial_3]f
	=
	-2 \varphi' \partial_3 Z_3 f
	-
	(Z_3
	\varphi' \partial_3 f
	-
	(\varphi')^2) \partial_3 f.
\end{align*}
By repeating the above steps, it is easy to check that
\begin{align*}
	[Z_3^{\alpha_3},\partial_3]f = \sum_{k=0}^{\alpha_3-1}
	C_{k,\alpha_3}(\varphi') \partial_3 Z_3^k f, \quad \forall \alpha_3 \ge 1,
\end{align*}
where $C_{k,\alpha_3}(\varphi')$ are some harmless functions composed by $Z_3^i (\varphi')^j$ with $i+j\le k$, $j \ge 1$. Furthermore, one can easily obtain
\begin{align*}
	|Z_3^i (\varphi')^j| \lesssim  \frac{1}{x_3+1},\quad
	\forall x_3 \in \mathbb{R_+}, \quad \forall i \in \mathbb{N}, \quad \forall j \in \mathbb{N}^*,
\end{align*}
which also implies that $Z_3^i (\varphi')^j$ is bounded for every $x_3 \in \mathbb{R_+}$, $i \in \mathbb{N}$ and $j \in \mathbb{N}^*$. Thus, $C_{k,\alpha_3}(\varphi')$ are also bounded and we have
\begin{align*}
	|C_{k,\alpha_3}(\varphi')| \lesssim \frac{1}{x_3 + 1}= \frac{1}{x_3}\varphi,\quad
	\forall x_3 \in \mathbb{R_+}, \quad \forall k \in \mathbb{N}.
\end{align*}
This fact implies directly
\begin{equation}\label{3207}
	\begin{aligned}
	III_{1,1}
	&\lesssim  \sum\limits_{k=0}^{\alpha_3-1} \int_{\mathbb{R}_+^3}
	\frac{1}{x_3} \varphi \cdot x_3 \Vert \partial_3 u_3 \Vert_{L^\infty_{x_3}}
	|\partial^{\alpha_1}_1
	\partial^{\alpha_2}_2 \partial_3
	Z^k_3 \omega_h^u | | Z^\alpha \omega_h^u| dx\\
	&\lesssim  \sum\limits_{k=0}^{\alpha_3-1} \int_{\mathbb{R}_+^3}
	\Vert \partial_3 u_3 \Vert_{L^\infty_{x_3}}
	|\partial^{\alpha_1}_1
	\partial^{\alpha_2}_2
	Z^{k+1}_3 \omega_h^u | | Z^\alpha \omega_h^u| dx\\
	&\lesssim  \sum\limits_{k=0}^{\alpha_3-1}
	\Vert \partial_3 u_3 \Vert_{L^2_{x_1} L^2_{x_2} L^\infty_{x_3}}
	\Vert \partial^{\alpha_1}_1
	\partial^{\alpha_2}_2
	Z^{k+1}_3 \omega_h^u \Vert_{L^2_{x_1} L^\infty_{x_2} L^2_{x_3}}
	 \Vert Z^\alpha \omega_h^u \Vert_{L^\infty_{x_1} L^2_{x_2} L^2_{x_3}}\\
	 &\lesssim
	 \sum\limits_{k=0}^{\alpha_3-1}
	 \Vert \partial_3 u_3 \Vert_{L^2}^{\frac{1}{2}}
	 \Vert \partial_3 \partial_3 u_3 \Vert_{L^2}^{\frac{1}{2}}
	 \Vert \partial^{\alpha_1}_1
	 \partial^{\alpha_2}_2
	 Z^{k+1}_3 \omega_h^u \Vert_{L^2}^{\frac{1}{2}}
	 \Vert \partial_2 \partial^{\alpha_1}_1
	 \partial^{\alpha_2}_2
	 Z^{k+1}_3 \omega_h^u \Vert_{L^2}^{\frac{1}{2}}
	 \Vert Z^\alpha \omega_h^u \Vert_{L^2}^{\frac{1}{2}}
	  \Vert \partial_1 Z^\alpha \omega_h^u \Vert_{L^2}^{\frac{1}{2}}\\
	 &\lesssim
	 \Vert \omega_h^u \Vert_{H^{m-1}_{co}}
	 (\Vert \partial_h u \Vert_{H^m_{co}}^2
	  +\Vert \partial_h \omega_h^u \Vert_{H^{m-1}_{co}}^2).
\end{aligned}
\end{equation}
Here, the notation $\Vert f \Vert_{L^2_{x_3} L^2_{x_2} L^\infty_{x_1}}$ represents the $L^\infty$-norm in the $x_1$ variable, followed by the $L^2$-norm in $x_2$ and the $L^2$-norm in $x_3$.
On the other hand, let us deal with $III_{1,2}$ term.
\begin{equation*}
	\begin{aligned}
	III_{1,2}&=-
	\sum\limits_{\substack{\beta+\gamma=\alpha\\\beta \neq 0}} C_{\beta,\gamma} \int_{\mathbb{R}_+^3} Z^\beta u_h \cdot Z^\gamma \partial_h \omega_h^u \cdot Z^\alpha \omega_h^u dx
	-
	\sum\limits_{\substack{\beta+\gamma=\alpha\\\beta \neq 0}} C_{\beta,\gamma} \int_{\mathbb{R}_+^3} Z^\beta u_3  Z^\gamma \partial_3 \omega_h^u \cdot Z^\alpha \omega_h^u dx\\
	&\overset{def}{=}III_{1,2,1}+III_{1,2,2}.
\end{aligned}
\end{equation*}
Integrating by parts and applying Lemma \ref{lemma1}, it holds
\begin{equation}\label{3208}
	\begin{aligned}
	III_{1,2,1}&=
	\sum\limits_{\substack{\beta+\gamma=\alpha\\\beta \neq 0}} C_{\beta,\gamma} \int_{\mathbb{R}_+^3} Z^\beta (\partial_h \cdot u_h)  Z^\gamma \omega_h^u \cdot Z^\alpha \omega_h^u dx
	+
	\sum\limits_{\substack{\beta+\gamma=\alpha\\\beta \neq 0}} C_{\beta,\gamma} \int_{\mathbb{R}_+^3} Z^\beta u_h  \cdot \partial_h Z^\alpha \omega_h^u \cdot Z^\gamma \omega_h^u dx\\
	&\lesssim
	\sum\limits_{\substack{\beta+\gamma=\alpha\\\beta \neq 0}}
	\Vert Z^\beta \partial_h u_h \Vert_{L^2}^{\frac{1}{2}}
	\Vert \partial_3 Z^\beta \partial_h u_h \Vert_{L^2}^{\frac{1}{2}}
	\Vert Z^\gamma \omega_h^u \Vert_{L^2}^{\frac{1}{2}}
	\Vert \partial_2 Z^\gamma \omega_h^u \Vert_{L^2}^{\frac{1}{2}}
	\Vert Z^\alpha \omega_h^u \Vert_{L^2}^{\frac{1}{2}}
	\Vert \partial_1 Z^\alpha \omega_h^u \Vert_{L^2}^{\frac{1}{2}}\\
	&~~~~+\!\!\!\!\!\sum\limits_{\substack{\beta+\gamma=\alpha\\\beta \neq 0}}
	\Vert Z^\beta u_h \Vert_{L^2}^{\frac{1}{4}}
	\Vert \partial_1 Z^\beta u_h \Vert_{L^2}^{\frac{1}{4}}
	\Vert \partial_2 Z^\beta u_h \Vert_{L^2}^{\frac{1}{4}}
	\Vert \partial_1 \partial_2 Z^\beta u_h \Vert_{L^2}^{\frac{1}{4}}
	\Vert \partial_h Z^\alpha \omega_h^u \Vert_{L^2}
	\Vert Z^\gamma \omega_h^u \Vert_{L^2}^{\frac{1}{2}}
	\Vert \partial_2 Z^\gamma \omega_h^u
	\Vert_{L^2}^{\frac{1}{2}} \\
	&\lesssim
	(\Vert u \Vert_{H^m_{co}}+\Vert \omega_h^u \Vert_{H^{m-1}_{co}})
	(\Vert \partial_h u \Vert_{H^m_{co}}^2
	 +\Vert \partial_h \omega_h^u \Vert_{H^{m-1}_{co}}^2).
\end{aligned}
\end{equation}
It should be pointed out that the estimate of term $III_{1,2,2}$
is somewhat complicated. Thus, we decompose this term
into two component as follows:
\begin{equation}\label{eq01}
	\begin{aligned}
	III_{1,2,2}&=-
	\sum\limits_{\substack{\beta+\gamma=\alpha\\\beta \neq 0}} C_{\beta,\gamma} \int_{\mathbb{R}^2 \times [1,+\infty)} Z^\beta u_3  Z^\gamma \partial_3 \omega_h^u \cdot Z^\alpha \omega_h^u dx\\
	&~~~~-
	\sum\limits_{\substack{\beta+\gamma=\alpha\\\beta \neq 0}} C_{\beta,\gamma} \int_{\mathbb{R}^2 \times [0,1)} Z^\beta u_3 Z^\gamma \partial_3 \omega_h^u \cdot Z^\alpha \omega_h^u dx\\
	&\overset{def}{=}III_{1,2,2,1}+III_{1,2,2,2}.
\end{aligned}
\end{equation}
Obviously, $\varphi \ge 1$ in $[1,+\infty)$ holds, and
then we have
\begin{equation}\label{3205}
	\begin{aligned}
	III_{1,2,2,1}&\lesssim
	\sum\limits_{\substack{\beta+\gamma=\alpha\\\beta \neq 0}}  \int_{\mathbb{R}^2 \times [1,+\infty)} |Z^\beta u_3|  | \varphi Z^\gamma \partial_3 \omega_h^u|  |Z^\alpha \omega_h^u| dx
	\\
	&\lesssim
	\sum\limits_{\substack{\beta+\gamma=\alpha\\\beta \neq 0}}  \sum\limits_{k=0}^{\gamma_3}
	\Vert Z^\beta u_3 \Vert_{L^2}^{\frac{1}{2}}
	\Vert \partial_3 Z^\beta u_3 \Vert_{L^2}^{\frac{1}{2}}
	\Vert \partial_1^{\gamma_1}
	\partial_2^{\gamma_2}
	Z_3^{k+1} \omega_h^u \Vert_{L^2}^{\frac{1}{2}}
	\Vert \partial_1 \partial_1^{\gamma_1}
	\partial_2^{\gamma_2}
	Z_3^{k+1} \omega_h^u \Vert_{L^2}^{\frac{1}{2}}
	\\
	&~~~~~~~~~~~~~~~~~~~~~~~~~~\times
	\Vert Z^\alpha \omega_h^u \Vert_{L^2}^{\frac{1}{2}}
	\Vert \partial_2 Z^\alpha \omega_h^u \Vert_{L^2}^{\frac{1}{2}}\\
	&\lesssim
	\Vert \omega_h^u \Vert_{H_{co}^{m-1}}
	(\Vert \partial_h u \Vert_{H_{co}^{m}}^2
	+
	\Vert \partial_h \omega_h^u \Vert_{H_{co}^{m-1}}^2).
\end{aligned}
\end{equation}
On the other hand, due to the basic fact
$x_3 \le 2 \varphi(x_3)$ in $[0,1)$ and $u_3$ vanishes on the boundary, we have
\begin{equation}\label{3206}
	\begin{aligned}
III_{1,2,2,2} &\lesssim
	\sum\limits_{\substack{\beta+\gamma=\alpha\\\beta \neq 0}}  \int_{\mathbb{R}^2 \times [0,1)}
x_3  \Vert \partial_3 Z^\beta u_3 \Vert_{L^\infty_{x_3}}  | Z^\gamma \partial_3 \omega_h^u|  |Z^\alpha \omega_h^u| dx\\
	&\lesssim \sum\limits_{\substack{\beta+\gamma=\alpha\\\beta \neq 0}}
	\sum\limits_{k=0}^{\gamma_3} \int_{\mathbb{R}^3_+}
 \Vert \partial_3 Z^\beta u_3 \Vert_{L^\infty_{x_3}} | \partial_1^{\gamma_1}
 \partial_2^{\gamma_2}
 Z_3^{k+1} \omega_h^u|  |Z^\alpha \omega_h^u| dx\\
	&\lesssim
	\sum\limits_{\substack{\beta+\gamma=\alpha\\\beta \neq 0}}
	\sum\limits_{k=0}^{\gamma_3}
	\Vert \partial_3 Z^\beta u_3 \Vert_{L^2}^{\frac{1}{2}}
	\Vert \partial_3 \partial_3 Z^\beta u_3 \Vert_{L^2}^{\frac{1}{2}}
	\Vert \partial_1^{\gamma_1}
	\partial_2^{\gamma_2}
	Z_3^{k+1} \omega_h^u \Vert_{L^2}^{\frac{1}{2}}
	\Vert \partial_1 \partial_1^{\gamma_1}
	\partial_2^{\gamma_2}
	Z_3^{k+1} \omega_h^u \Vert_{L^2}^{\frac{1}{2}}
	\\
	&~~~~~~~~~~~~~~~~~~~\times
	\Vert Z^\alpha \omega_h^u \Vert_{L^2}^{\frac{1}{2}}
	\Vert \partial_2 Z^\alpha \omega_h^u \Vert_{L^2}^{\frac{1}{2}}\\
	&\lesssim
	\Vert \omega_h^u \Vert_{H_{co}^{m-1}}
	(\Vert \partial_h u \Vert_{H_{co}^{m}}^2
	+
	\Vert \partial_h \omega_h^u \Vert_{H_{co}^{m-1}}^2),
\end{aligned}
\end{equation}
where we have used the fact that
\begin{equation}\label{3237}
\begin{aligned}
&\int_{\mathbb{R}^3_+}
 \Vert \partial_3 Z^\beta u_3 \Vert_{L^\infty_{x_3}} | \partial_1^{\gamma_1}
 \partial_2^{\gamma_2}
 Z_3^{k+1} \omega_h^u|  |Z^\alpha \omega_h^u| dx\\
\lesssim
&\int_{\mathbb{R}^2}
 \Vert \partial_3 Z^\beta u_3 \Vert_{L^\infty_{x_3}}
 \|\partial_1^{\gamma_1}
 \partial_2^{\gamma_2}
 Z_3^{k+1} \omega_h^u\|_{L^2_{x_3}}
 \|Z^\alpha \omega_h^u\|_{L^2_{x_3}} dx_h \\
\lesssim
&\int_{\mathbb{R}}
 \left\|\|Z^\alpha \omega_h^u\|_{L^2_{x_3}}\right\|_{L^\infty_{x_2}}
 \left\|\Vert \partial_3 Z^\beta u_3 \Vert_{L^\infty_{x_3}}\right\|_{L^2_{x_2}}
 \|\partial_1^{\gamma_1}
 \partial_2^{\gamma_2}
 Z_3^{k+1} \omega_h^u\|_{L^2_{x_3, x_2}}d x_1\\
\lesssim
& \left\|\Vert \partial_3 Z^\beta u_3 \Vert_{L^\infty_{x_3}}\right\|_{L^2_{x_2, x_1}}
  \left\|\|\partial_1^{\gamma_1}
  \partial_2^{\gamma_2}
  Z_3^{k+1} \omega_h^u\|_{L^\infty_{x_1}} \right\|_{L^2_{x_3, x_2}}
   \left\|\|Z^\alpha \omega_h^u\|_{L^\infty_{x_2}}\right\|_{L^2_{x_3, x_1}}   \\
\lesssim
&\Vert \partial_3 Z^\beta u_3 \Vert_{L^2}^{\frac{1}{2}}
	\Vert \partial_3 \partial_3 Z^\beta u_3 \Vert_{L^2}^{\frac{1}{2}}
	\Vert \partial_1^{\gamma_1}
	\partial_2^{\gamma_2}
	Z_3^{k+1} \omega_h^u \Vert_{L^2}^{\frac{1}{2}}
	\Vert \partial_1 \partial_1^{\gamma_1}
	\partial_2^{\gamma_2}
	Z_3^{k+1} \omega_h^u \Vert_{L^2}^{\frac{1}{2}}
	\Vert Z^\alpha \omega_h^u \Vert_{L^2}^{\frac{1}{2}}
	\Vert \partial_2 Z^\alpha \omega_h^u \Vert_{L^2}^{\frac{1}{2}}.
\end{aligned}
\end{equation}
Thus, the combination of estimates \eqref{3205} and \eqref{3206} gives directly
\begin{equation*}
III_{1,2,2}
\lesssim
\Vert \omega_h^u \Vert_{H_{co}^{m-1}}
(\Vert \partial_h u \Vert_{H_{co}^{m}}^2
+
\Vert \partial_h \omega_h^u \Vert_{H_{co}^{m-1}}^2),
\end{equation*}
which, together with estimate \eqref{3208}, yields directly
\begin{equation}\label{3209}
III_{1,2}
\lesssim
(\Vert u \Vert_{H^m_{co}}+\Vert \omega_h^u \Vert_{H^{m-1}_{co}})
	(\Vert \partial_h u \Vert_{H^m_{co}}^2
	 +\Vert \partial_h \omega_h^u \Vert_{H^{m-1}_{co}}^2).
\end{equation}
Combining the estimates \eqref{3207} and \eqref{3209}, we obtain
\begin{equation}\label{3210}
III_1 \lesssim (\Vert u \Vert_{H^m_{co}}+\Vert \omega_h^u \Vert_{H^{m-1}_{co}})
	(\Vert \partial_h u \Vert_{H^m_{co}}^2
	 +\Vert \partial_h \omega_h^u \Vert_{H^{m-1}_{co}}^2).
\end{equation}
\textbf{Deal with the term $III_2$.}
The term $III_{2}$ vanishes if $\alpha_3 =0$.
Hence we only consider the case of $\alpha_3 \ge 1$.
Obviously, it holds
\begin{equation*}
III_2 = \varepsilon \sum\limits_{k=0}^{\alpha_3-1} \int_{\mathbb{R}_+^3} \partial_1^{\alpha_1}
	\partial_2^{\alpha_2} \left( \hat{C}_{k,\alpha_3}(\varphi') Z^k_3 \partial_3 \omega_h^u
	+ \widetilde{C}_{k,\alpha_3}(\varphi') \partial_3^2 Z^k_3 \omega_h^u \right)\cdot Z^\alpha \omega_h^u dx\overset{def}{=}III_{2,1}+III_{2,2},
\end{equation*}
where $\hat{C}_{k,\alpha_3}(\varphi')$ and $\widetilde{C}_{k,\alpha_3}(\varphi')$ denote the harmless functions depending on derivatives of $\varphi$.
The Cauchy's inequality implies
\begin{equation*}
	III_{2,1} \lesssim
	\varepsilon \sum\limits_{k=0}^{\alpha_3-1}
	\Vert \partial_1^{\alpha_1}
	\partial_2^{\alpha_2} Z^k_3 \partial_3 \omega_h^u \Vert_{L^2}
	\Vert Z^\alpha \omega_h^u \Vert_{L^2}
	\lesssim \varepsilon \Vert \partial_3 \omega_h^u \Vert_{H^{m-2}_{co}}^2,
\end{equation*}
where we have used the fact $\alpha_3 \ge 1$.
Integration by parts and Lemma \ref{lemma1} yield that
\begin{equation*}
	\begin{aligned}
	III_{2,2} &=
	- \varepsilon \sum\limits_{k=0}^{\alpha_3-1} \int_{\mathbb{R}_+^3} \partial_3 \widetilde{C}_{k,\alpha_3}(\varphi') \partial_3 \partial_1^{\alpha_1} \partial_2^{\alpha_2} Z^k_3 \omega_h^u \cdot Z^\alpha \omega_h^u dx\\
	&~~~~- \varepsilon \sum\limits_{k=0}^{\alpha_3-1} \int_{\mathbb{R}_+^3} \widetilde{C}_{k,\alpha_3}(\varphi') \partial_3 \partial_1^{\alpha_1} \partial_2^{\alpha_2} Z^k_3 \omega_h^u \cdot \partial_3 Z^\alpha \omega_h^u dx\\
	&\lesssim
	\varepsilon \sum\limits_{k=0}^{\alpha_3-1}
	\Vert \partial_3 \partial_1^{\alpha_1} \partial_2^{\alpha_2} Z^k_3 \omega_h^u \Vert_{L^2}
	\Vert Z^\alpha \omega_h^u \Vert_{L^2}
	+
	\varepsilon \sum\limits_{k=0}^{\alpha_3-1}
	\Vert \partial_3 \partial_1^{\alpha_1} \partial_2^{\alpha_2} Z^k_3 \omega_h^u \Vert_{L^2}
	\Vert \partial_3 Z^\alpha \omega_h^u \Vert_{L^2}\\
	&\lesssim \varepsilon
	\Vert \partial_3 \omega_h^u \Vert^2_{H^{m-2}_{co}}
	+
	\varepsilon
	\Vert \partial_3 \omega_h^u \Vert_{H^{m-2}_{co}}
	\Vert \partial_3 \omega_h^u \Vert_{H^{m-1}_{co}}.
\end{aligned}
\end{equation*}
Thus, we have
\begin{equation}\label{3211}
\begin{aligned}
III_2 \lesssim \varepsilon \Vert \partial_3 \omega_h^u \Vert_{H^{m-2}_{co}}^2
+
\varepsilon
\Vert \partial_3 \omega_h^u \Vert_{H^{m-2}_{co}}
\Vert \partial_3 \omega_h^u \Vert_{H^{m-1}_{co}}.
\end{aligned}
\end{equation}
\textbf{Deal with the term $III_3-III_5$.}
With the help of Lemma \ref{lemma1}, it is easy to see that
\begin{equation*}
\begin{aligned}
III_3
&\lesssim (\Vert B \Vert_{H^m_{co}}+\Vert (\omega_h^u, \omega_h^B) \Vert_{H^{m-1}_{co}})
(\Vert \partial_3 B \Vert_{H^m_{co}}^2
+\Vert ( \partial_h \omega^u_h, \partial_3 \omega_h^B ,\partial_1 B) \Vert_{H^{m-1}_{co}}^2),\\
III_4
&\lesssim (\Vert u \Vert_{H^{m}_{co}}+\Vert \omega_h^u\Vert_{H^{m-1}_{co}})
(\Vert \partial_h u \Vert_{H^{m}_{co}}^2
 +\Vert  \partial_h \omega_h^u \Vert_{H^{m-1}_{co}}^2),\\
III_5
&\lesssim (\Vert B \Vert_{H^{m}_{co}}+\Vert (\omega_h^u, \omega_h^B) \Vert_{H^{m-1}_{co}})
(\Vert \partial_3 B \Vert_{H^{m}_{co}}^2
 +\Vert ( \partial_h \omega_h^u, \partial_3 \omega_h^B,\partial_1 B)\Vert_{H^{m-1}_{co}}^2).
\end{aligned}
\end{equation*}
\textbf{Deal with the term $III_6$.}
We have the following splitting:
\begin{equation*}
	\begin{aligned}
III_6 =&
	-
	\sum\limits_{\substack{\beta + \gamma = \alpha \\ \beta \neq 0}} C_{\beta,\gamma} \int_{\mathbb{R}_+^3}
	Z^\beta u_3 Z^\gamma \partial_3 \omega_h^B \cdot Z^\alpha \omega_h^B dx
	-\sum\limits_{\substack{\beta + \gamma = \alpha \\ \beta  \neq 0}} C_{\beta,\gamma} \int_{\mathbb{R}_+^3}
	Z^\beta u_h \cdot Z^\gamma \partial_h \omega_h^B \cdot Z^\alpha \omega_h^B dx
	\\
	&
	-\int_{\mathbb{R}^3_+}u_3 [Z^\alpha,\partial_3]\omega_h^B \cdot Z^\alpha \omega_h^B dx\\
	\overset{def}{=}&III_{6,1}+III_{6,2}
	+III_{6,3}
	.
\end{aligned}
\end{equation*}
The application of Lemma \ref{lemma1} yields directly
\begin{equation}\label{3215}
	\begin{aligned}
	III_{6,1} &\lesssim
	\sum\limits_{\substack{\beta + \gamma = \alpha \\ \beta  \neq 0}}
	\Vert
	Z^\beta u_3 \Vert_{L^2} ^{\frac{1}{2}}
	\Vert \partial_1
	Z^\beta u_3 \Vert_{L^2} ^{\frac{1}{2}}
	\Vert Z^\gamma \partial_3 \omega_h^B \Vert_{L^2} ^{\frac{1}{2}}
	\Vert \partial_2 Z^\gamma \partial_3 \omega_h^B \Vert_{L^2} ^{\frac{1}{2}}
	\Vert Z^\alpha \omega_h^B \Vert_{L^2}^{\frac{1}{2}}
	\Vert \partial_3 Z^\alpha \omega_h^B \Vert_{L^2}^{\frac{1}{2}}\\
	&\lesssim
	(\Vert u \Vert_{H^m_{co}}+\Vert\omega_h^B \Vert_{H^{m-1}_{co}})
	(\Vert \partial_h u \Vert_{H^m_{co}}^2+
	\Vert \partial_3 \omega_h^B \Vert_{H^{m-1}_{co}}^2).
\end{aligned}
\end{equation}
Obviously, the term $III_{6,3}$ vanishes if $\alpha_3 \neq 0$. Hence, we only consider the case of $\alpha_3 \ge 1$. It is clear that
\begin{equation}\label{32151}
	\begin{aligned}
III_{6,3}&=\sum_{k=0}^{\alpha_3-1} \int_{\mathbb{R}^3_+} C_{k,\alpha_3}(\varphi') u_3  \partial_1^{\alpha_1} \partial_2^{\alpha_2} \partial_3 Z^k_3 \omega_h^B \cdot Z^\alpha \omega_h^B dx\\
&\lesssim
\sum_{k=0}^{\alpha_3-1}
\Vert u_3 \Vert_{L^2}^\frac12
\Vert \partial_1 u_3 \Vert_{L^2}^\frac12
\Vert \partial_1^{\alpha_1} \partial_2^{\alpha_2} \partial_3 Z^k_3 \omega_h^B \Vert_{L^2}^\frac12
\Vert \partial_2 \partial_1^{\alpha_1} \partial_2^{\alpha_2} \partial_3 Z^k_3 \omega_h^B \Vert_{L^2}^\frac12
\Vert Z^\alpha \omega_h^B \Vert_{L^2}^\frac12
\Vert \partial_3 Z^\alpha \omega_h^B \Vert_{L^2}^\frac12
\\
&\lesssim
(\Vert u \Vert_{H^m_{co}}
+
\Vert \omega_h^B \Vert_{H^{m-1}_{co}}
)
(\Vert \partial_h u \Vert_{H^m_{co}}^2
+
\Vert \partial_3 \omega_h^B \Vert_{H^{m-1}_{co}}^2
).
	\end{aligned}
\end{equation}
Next,  we turn to estimate $III_{6,2}$.
Since  magnetic field only has the dissipation structure in the $x_1$
and $x_3$ direction,  we  classify the value of $\beta$ and discuss below.

If $\beta_3 \neq 0$, then $\alpha_3 \neq 0$. Hence, we obtain
\begin{equation}\label{32152}
	\begin{aligned}
	III_{6,2} &\lesssim
	\sum\limits_{\substack{\beta + \gamma = \alpha \\ \beta  \neq 0}}
	\Vert
	Z^\beta u_h \Vert_{L^2} ^{\frac{1}{2}}
	\Vert \partial_1
	Z^\beta u_h \Vert_{L^2} ^{\frac{1}{2}}
	\Vert Z^\gamma \partial_h \omega_h^B \Vert_{L^2} ^{\frac{1}{2}}
	\Vert \partial_3 Z^\gamma \partial_h \omega_h^B \Vert_{L^2} ^{\frac{1}{2}}
	\Vert Z^\alpha \omega_h^B \Vert_{L^2}^{\frac{1}{2}}
	\Vert \partial_2 Z^\alpha \omega_h^B \Vert_{L^2}^{\frac{1}{2}}\\
	&\lesssim
	(\Vert u \Vert_{H^m_{co}}+\Vert\omega_h^B \Vert_{H^{m-1}_{co}})
	(\Vert \partial_h u \Vert_{H^m_{co}}^2+
	\Vert \partial_3 \omega_h^B \Vert_{H^{m-1}_{co}}^2),
\end{aligned}
\end{equation}
where we have used the fact that $\Vert Z^\alpha \omega_h^B \Vert_{L^2}
+
\Vert \partial_2 Z^\alpha \omega_h^B \Vert_{L^2}
\lesssim \Vert \partial_3 \omega_h^B \Vert_{H^{m-1}_{co}}$ since $\alpha_3\neq 0$ in this case.

If $\beta_1 \neq 0$, then $\alpha_1 \neq 0$. Hence, we obtain
\begin{equation}\label{32153}
	\begin{aligned}
	III_{6,2}
&= -\sum\limits_{\substack{\beta + \gamma = \alpha \\ \beta \neq 0,\gamma \neq 0}} C_{\beta,\gamma} \int_{\mathbb{R}_+^3} Z^\beta u_h \cdot Z^\gamma \partial_h \omega_h^B \cdot Z^\alpha \omega_h^B dx
    -
    \int_{\mathbb{R}_+^3} Z^\alpha u_h \cdot \partial_h \omega_h^B \cdot Z^\alpha \omega_h^B dx\\
	&\lesssim\sum\limits_{\substack{\beta + \gamma = \alpha \\ \beta \neq 0,\gamma \neq 0}}
	 \Vert Z^\beta u_h \Vert_{L^2} ^{\frac{1}{4}}
	\Vert \partial_1
	Z^\beta u_h \Vert_{L^2} ^{\frac{1}{4}}
	\Vert \partial_2
	Z^\beta u_h \Vert_{L^2} ^{\frac{1}{4}}
	\Vert \partial_1 \partial_2
	Z^\beta u_h \Vert_{L^2} ^{\frac{1}{4}}
	\Vert Z^\gamma \partial_h \omega_h^B \Vert_{L^2}^{\frac{1}{2}}
	\Vert \partial_3 Z^\gamma \partial_h \omega_h^B \Vert_{L^2} ^{\frac{1}{2}}
	\Vert Z^\alpha \omega_h^B \Vert_{L^2}\\
	&~~~~
	+
	\Vert Z^\alpha u_h \Vert_{L^2} ^{\frac{1}{2}}
	\Vert \partial_2 Z^\alpha u_h \Vert _{L^2}^{\frac{1}{2}}
	\Vert \partial_h \omega_h^B \Vert _{L^2}^{\frac{1}{2}}
	\Vert \partial_1 \partial_h \omega_h^B \Vert_{L^2} ^{\frac{1}{2}}
	\Vert Z^\alpha \omega_h^B \Vert _{L^2}^{\frac{1}{2}}
	\Vert \partial_3 Z^\alpha \omega_h^B \Vert _{L^2}^{\frac{1}{2}}
	\\
	&\lesssim
	\Vert u \Vert_{H^m_{co}}^{\frac{1}{2}}
	\Vert \omega_h^B \Vert_{H^{m-1}_{co}}^{\frac{1}{2}}
	\Vert \partial_h u \Vert_{H^m_{co}}^{\frac{1}{2}}
	\Vert \partial_3 \omega_h^B \Vert_{H^{m-1}_{co}}^{\frac{1}{2}}
	\Vert \partial_1 \omega_h^B \Vert_{H^{m-2}_{co}}\\
	&\lesssim
	(\Vert u \Vert_{H^m_{co}}
	 +\Vert \omega_h^B \Vert_{H^{m-1}_{co}})
	(\Vert (\partial_h u, \partial_3 B)\Vert_{H^m_{co}}^2 + \Vert (\partial_3  \omega_h^B, \partial_1 B)\Vert_{H^{m-1}_{co}}^2),
\end{aligned}
\end{equation}
where we have used the fact that $\Vert Z^\alpha \omega_h^B \Vert_{L^2} \lesssim \Vert \partial_1 \omega_h^B \Vert_{H^{m-2}_{co}} \lesssim \Vert \partial_3 B \Vert_{H^{m}_{co}} + \Vert \partial_1 B \Vert_{H^{m-1}_{co}}
$ since $\alpha_1 \neq 0$ in this case.

If $\beta_2 \neq 0$, then $\alpha_2 \neq 0$. Then, we have the decomposition
\begin{equation*}
	\begin{aligned}
	III_{6,2}
=&-
\int_{\mathbb{R}_+^3} Z^\alpha u_h \cdot \partial_h \omega_h^B \cdot Z^\alpha \omega_h^B dx
-\sum\limits_{\substack{\beta + \gamma = \alpha \\ \beta \neq 0,\gamma \neq 0}} C_{\beta,\gamma} \int_{\mathbb{R}_+^3} Z^\beta u_h \cdot Z^\gamma \partial_h \omega_h^B \cdot Z^\alpha \omega_h^B dx\\
=&-
\int_{\mathbb{R}_+^3} Z^\alpha u_h \cdot \partial_h \omega_h^B \cdot Z^\alpha \omega_h^B dx
-\sum\limits_{\substack{\beta + \gamma = \alpha \\ \beta \neq 0,\gamma \neq 0}} C_{\beta,\gamma} \int_{\mathbb{R}_+^3} Z^\beta u_1 Z^\gamma \partial_1 \omega_h^B \cdot Z^\alpha \omega_h^B dx \\
&-
	\sum\limits_{\substack{\beta + \gamma = \alpha \\ \beta \neq 0,\gamma \neq 0}} C_{\beta,\gamma} \int_{\mathbb{R}_+^3} Z^\beta u_2 Z^\gamma \partial_2 \omega_2^B Z^\alpha \omega_2^B dx
-\alpha_2\int_{\mathbb{R}_+^3} \partial_2 u_2 |Z^\alpha \omega_1^B|^2 dx\\
&-\sum\limits_{\substack{\beta + \gamma = \alpha \\ \beta \neq 0,\gamma \neq 0,|\gamma| \neq m-2}} C_{\beta,\gamma} \int_{\mathbb{R}_+^3} Z^\beta u_2 Z^\gamma \partial_2 \omega_1^B Z^\alpha \omega_1^B dx\\
\overset{def}{=}&III_{6,2,1}+III_{6,2,2}+III_{6,2,3}+III_{6,2,4}+III_{6,2,5}.
\end{aligned}
\end{equation*}
With the help of Lemma \ref{lemma1}, it is easy to check that
\begin{equation}\label{3214}
\begin{aligned}
III_{6,2,1}
\lesssim & \Vert \omega_h^B \Vert_{H^{m-1}_{co}}
	(\Vert (\partial_h u, \partial_3 B )\Vert_{H^m_{co}}^2
     +
	\Vert ( \partial_3 \omega_h^B,\partial_1 B)\Vert_{H^{m-1}_{co}}^2),\\
III_{6,2,2}
\lesssim &(\Vert u \Vert_{H^m_{co}}+\Vert \omega_h^B \Vert_{H^{m-1}_{co}})
	(\Vert (\partial_h u, \partial_3 B )\Vert_{H^m_{co}}^2
     +\Vert ( \partial_3 \omega_h^B,\partial_1 B)\Vert_{H^{m-1}_{co}}^2),\\
III_{6,2,3}
\lesssim &(\Vert u \Vert_{H^m_{co}}+\Vert \omega_h^B \Vert_{H^{m-1}_{co}})
	(\Vert (\partial_h u, \partial_3 B )\Vert_{H^m_{co}}^2
     +\Vert ( \partial_3 \omega_h^B,\partial_1 B)\Vert_{H^{m-1}_{co}}^2),\\
III_{6,2,5}
\lesssim & 	\Vert \omega_h^B \Vert_{H^{m-1}_{co}}
	(\Vert (\partial_h u, \partial_3 B )\Vert_{H^m_{co}}^2
     +
	\Vert ( \partial_3 \omega_h^B,\partial_1 B)\Vert_{H^{m-1}_{co}}^2).
\end{aligned}
\end{equation}
Since  magnetic field has the dissipation structure in the $x_1$
and $x_3$ direction, we apply the divergence-free condition
to obtain
\begin{equation*}
	\begin{aligned}
	III_{6,2,4}=\alpha_2
    \int_{\mathbb{R}_+^3} \partial_3 u_3 |Z^\alpha \omega_1^B|^2 dx
    +\alpha_2\int_{\mathbb{R}_+^3} \partial_1 u_1 |Z^\alpha \omega_1^B|^2 dx
	\overset{def}{=}III_{6,2,4,1}+III_{6,2,4,2}.
\end{aligned}
\end{equation*}
Integrating by parts and applying Lemma \ref{lemma1}, it holds
\begin{equation}\label{3212}
	\begin{aligned}
	III_{6,2,4,1}&= -2\alpha_2 \int_{\mathbb{R}_+^3} u_3 Z^\alpha \omega_1^B \partial_3 Z^\alpha \omega_1^B dx\\
	&\lesssim
	\Vert u \Vert_{L^2}^{\frac{1}{4}}
	\Vert \partial_1 u \Vert_{L^2}^{\frac{1}{4}}
	\Vert \partial_2 u \Vert_{L^2}^{\frac{1}{4}}
	\Vert \partial_1 \partial_2 u \Vert_{L^2}^{\frac{1}{4}}
	\Vert Z^\alpha \omega_h^B \Vert_{L^2}^{\frac{1}{2}}
	\Vert \partial_3 Z^\alpha \omega_h^B \Vert_{L^2}^{\frac{1}{2}}
	\Vert \partial_3 Z^\alpha \omega_h^B \Vert_{L^2}
	\\
	&\lesssim
	(\Vert u \Vert_{H^m_{co}}+\Vert \omega_h^B \Vert_{H^{m-1}_{co}})
	(\Vert \partial_h u \Vert_{H^m_{co}}^2
      +\Vert \partial_3 \omega_h^B \Vert_{H^{m-1}_{co}}^2).
\end{aligned}
\end{equation}
Applying the Eqs.\eqref{eq1}$_2$ and \eqref{eq2}$_2$, we can rewrite the
term $III_{6,2,4,2}$ as follows
\begin{equation*}\label{a03}
	\begin{aligned}
		III_{6,2,4,2}&=\alpha_2\int_{\mathbb{R}_+^3} (\partial_t B_1 + u \cdot \nabla B_1 - \partial_3^2 B_1 - \varepsilon \Delta_h B_1 - B\cdot\nabla u_1) |Z^\alpha \omega_1^B|^2 dx\\
		&=
		\alpha_2\int_{\mathbb{R}_+^3}
		(u \cdot \nabla B_1 - \partial_3^2 B_1 - \varepsilon \Delta_h B_1
		- B\cdot\nabla u_1) |Z^\alpha \omega_1^B|^2 dx
		+\alpha_2
		\frac{d}{dt}\int_{\mathbb{R}_+^3} B_1 |Z^\alpha \omega_1^B|^2 dx\\
		&~~~~-2\alpha_2\int_{\mathbb{R}_+^3} B_1 Z^\alpha \omega_1^B Z^\alpha (\partial_3^2 \omega_1^B + \varepsilon \Delta_h \omega_1^B - u \cdot \nabla \omega_1^B +\mathcal{A}(u,B)_1 + B\cdot \nabla \omega_1^B + \partial_1 \omega_1^u) dx.
	\end{aligned}
\end{equation*}
Then, using Lemma \ref{lemma1} repeatedly, it is easy to check that
\begin{equation}\label{3213}
\begin{aligned}
III_{6,2,4,2}
&\lesssim
\frac{d}{dt}\int_{\mathbb{R}_+^3} B_1 |Z^\alpha \omega_1^B|^2 dx
+
\varepsilon
(\Vert B \Vert_{H^m_{co}}
+
\Vert \omega_h^B \Vert_{H^{m-1}_{co}}
)
(\Vert \partial_h B \Vert_{H^m_{co}}^2
+\Vert \partial_h \omega_h^B \Vert_{H^{m-1}_{co}}^2)
\\
&~~~~
+
(\Vert B \Vert_{H^m_{co}}
+
\Vert \omega_h^B \Vert_{H^{m-1}_{co}}
)
(
\Vert \partial_3 B \Vert_{H^m_{co}}
^2
+
\Vert (\partial_h \omega_h^u ,\partial_3 \omega_h^B ,\partial_1 B) \Vert_{H^{m-1}_{co}}^2)
\\
&
~~~~
+
(\Vert (u,B) \Vert_{H^m_{co}}^2
+
\Vert (\omega_h^u,\omega_h^B) \Vert_{H^{m-1}_{co}}^2
)
(
\Vert (\partial_h u, \partial_3 B) \Vert_{H^m_{co}}^2	
+
\Vert (\partial_h \omega_h^u, \partial_3 \omega_h^B,\partial_1 B) \Vert_{H^{m-1}_{co}}^2
).
\end{aligned}
\end{equation}
The combination of estimates \eqref{3212} and \eqref{3213} yields directly
\begin{equation*}
\begin{aligned}
III_{6,2,4}
\lesssim
&\frac{d}{dt}\int_{\mathbb{R}_+^3} B_1 |Z^\alpha \omega_1^B|^2 dx
		+
		\varepsilon
		(\Vert B \Vert_{H^m_{co}}+\Vert \omega_h^B \Vert_{H^{m-1}_{co}})
		(\Vert \partial_h B \Vert_{H^m_{co}}^2
         +\Vert \partial_h \omega_h^B \Vert_{H^{m-1}_{co}}^2)\\
&
+(\Vert (u,B) \Vert_{H^m_{co}}+\Vert \omega_h^B \Vert_{H^{m-1}_{co}})
(\Vert (\partial_h u, \partial_3 B) \Vert_{H^m_{co}}^2
 +\Vert ( \partial_h \omega_h^u, \partial_3 \omega_h^B,\partial_1 B)\Vert_{H^{m-1}_{co}}^2)
 \\
 &
 +
 (\Vert (u,B) \Vert_{H^m_{co}}^2
 +
 \Vert (\omega_h^u,\omega_h^B) \Vert_{H^{m-1}_{co}}^2
 )
 (
 \Vert (\partial_h u, \partial_3 B) \Vert_{H^m_{co}}^2	
 +
 \Vert (\partial_h \omega_h^u, \partial_3 \omega_h^B,\partial_1 B) \Vert_{H^{m-1}_{co}}^2
 ),
\end{aligned}
\end{equation*}
which, together with estimates \eqref{3214}, gives directly
\begin{equation}\label{32154}
	\begin{aligned}
III_{6,2}\lesssim
&
\frac{d}{dt}\int_{\mathbb{R}_+^3} B_1 |Z^\alpha \omega_1^B|^2 dx
+
\varepsilon
(\Vert B \Vert_{H^m_{co}}+\Vert \omega_h^B \Vert_{H^{m-1}_{co}})
(\Vert \partial_h B \Vert_{H^m_{co}}^2
+\Vert \partial_h \omega_h^B \Vert_{H^{m-1}_{co}}^2)\\
&
+(\Vert (u,B) \Vert_{H^m_{co}}+\Vert \omega_h^B \Vert_{H^{m-1}_{co}})
(\Vert (\partial_h u, \partial_3 B) \Vert_{H^m_{co}}^2
+\Vert ( \partial_h \omega_h^u, \partial_3 \omega_h^B,\partial_1 B)\Vert_{H^{m-1}_{co}}^2)
\\
&
+
(\Vert (u,B) \Vert_{H^m_{co}}^2
+
\Vert (\omega_h^u,\omega_h^B) \Vert_{H^{m-1}_{co}}^2
)
(
\Vert (\partial_h u, \partial_3 B) \Vert_{H^m_{co}}^2	
+
\Vert (\partial_h \omega_h^u, \partial_3 \omega_h^B,\partial_1 B) \Vert_{H^{m-1}_{co}}^2
).
	\end{aligned}
\end{equation}
Therefore, combining \eqref{3215}, \eqref{32151},\eqref{32152}, \eqref{32153} and \eqref{32154}, we obtain
\begin{equation}
\begin{aligned}
III_{6}
\lesssim
&\frac{d}{dt}\int_{\mathbb{R}_+^3} B_1 |Z^\alpha \omega_1^B|^2 dx
		+
		\varepsilon
		(\Vert B \Vert_{H^m_{co}}+\Vert \omega_h^B \Vert_{H^{m-1}_{co}})
		(\Vert \partial_h B \Vert_{H^m_{co}}^2
         +\Vert \partial_h \omega_h^B \Vert_{H^{m-1}_{co}}^2)\\
&
+(\Vert (u,B) \Vert_{H^m_{co}}+\Vert \omega_h^B \Vert_{H^{m-1}_{co}})
(\Vert (\partial_h u, \partial_3 B) \Vert_{H^m_{co}}^2
 +\Vert ( \partial_h \omega_h^u, \partial_3 \omega_h^B,\partial_1 B)\Vert_{H^{m-1}_{co}}^2)
 \\
 &
 +
 (\Vert (u,B) \Vert_{H^m_{co}}^2
 +
 \Vert (\omega_h^u,\omega_h^B) \Vert_{H^{m-1}_{co}}^2
 )
 (
 \Vert (\partial_h u, \partial_3 B) \Vert_{H^m_{co}}^2	
 +
 \Vert (\partial_h \omega_h^u, \partial_3 \omega_h^B,\partial_1 B) \Vert_{H^{m-1}_{co}}^2
 ).
\end{aligned}
\end{equation}
\textbf{Deal with the term $III_7$.}
Indeed, the term $III_7$ vanishes if $\alpha_3 =0$.
Thus, we only consider the case of $\alpha_3 \ge 1$.
Similar to the term $III_2$, it is easy to check that
\begin{equation}
	\begin{aligned}
III_7 \lesssim
\Vert \partial_3 \omega_h^B \Vert^2_{H^{m-2}_{co}}
+
\Vert \partial_3 \omega_h^B \Vert_{H^{m-2}_{co}}
\Vert \partial_3 \omega_h^B \Vert_{H^{m-1}_{co}}.
	\end{aligned}
\end{equation}
\textbf{Deal with the term $III_8$.}
Due to the definition of $\mathcal{A}(u,B)$ in \eqref{3216}, we have
\begin{equation*}
	\begin{aligned}
	III_8&\overset{def}{=}\int_{\mathbb{R}_+^3} Z^\alpha \left(
	\partial_2 B \cdot \nabla u_3
	-\partial_3 B \cdot \nabla u_2
	-\partial_2 u \cdot \nabla B_3
	+\partial_3 u \cdot \nabla B_2
	\right) Z^\alpha \omega_1^B dx\\
	&~~~~+\int_{\mathbb{R}_+^3} Z^\alpha \left(
	\partial_3 B \cdot \nabla u_1
	-\partial_1 B \cdot \nabla u_3
	-\partial_3 u \cdot \nabla B_1
	+\partial_1 u \cdot \nabla B_3
	\right) Z^\alpha \omega_2^B dx,
\end{aligned}
\end{equation*}
which, together with Lemma \ref{lemma1} repeatedly, yields directly
\begin{equation}\label{3226}
\begin{aligned}
III_{8}
\lesssim
&\frac{d}{dt}\int_{\mathbb{R}_+^3} B_3 Z^\alpha \partial_2 B_1  Z^\alpha \partial_2 B_3 dx
+
\frac{d}{dt}\int_{\mathbb{R}_+^3} B_1 | Z^\alpha \partial_2 B_3|^2 dx
+
		\varepsilon
		(\Vert B \Vert_{H^m_{co}}
		+
		\Vert  \omega_h^B \Vert_{H^{m-1}_{co}})
		\Vert \partial_h B \Vert_{H^m_{co}}^2\\
&+(\Vert (u,B)\Vert_{H^m_{co}}+\Vert (\omega_h^u,  \omega_h^B)\Vert_{H^{m-1}_{co}})
		(\Vert (\partial_h u,\partial_3 B )\Vert_{H^m_{co}}^2
         +\Vert ( \partial_h \omega_h^u,
         \partial_3 \omega_h^B,\partial_1 B )\Vert_{H^{m-1}_{co}}^2)
         \\
         &
         +
         (
         \Vert (u,B) \Vert_{H^m_{co}}^2
         +
         \Vert \omega_h^u \Vert_{H^{m-1}_{co}}^2
         )
         (
         \Vert (\partial_h u,\partial_3 B) \Vert_{H^m_{co}}^2
         +
         \Vert (\partial_h \omega_h^u, \partial_1 B) \Vert_{H^{m-1}_{co}}^2.
\end{aligned}
\end{equation}
\textbf{Deal with the term $III_9$.}
We have the following splitting:
\begin{equation*}
	\begin{aligned}
	III_9 &= \sum\limits_{\beta + \gamma = \alpha} C_{\beta,\gamma}
	\int_{\mathbb{R}_+^3} Z^\beta B_h \cdot Z^\gamma \partial_h \omega_h^u \cdot Z^\alpha \omega_h^B dx
	+
	\sum\limits_{\beta + \gamma = \alpha} C_{\beta,\gamma}
	\int_{\mathbb{R}_+^3} Z^\beta B_3 Z^\gamma \partial_3 \omega_h^u \cdot Z^\alpha \omega_h^B dx \\
	&\overset{def}{=}III_{9,1}+III_{9,2}.
\end{aligned}
\end{equation*}
Using Lemma \ref{lemma1}, it is easy to obtain
\begin{equation*}
	\begin{aligned}
	III_{9,1}&=\sum\limits_{\substack{\beta + \gamma = \alpha\\ \gamma \neq 0}} C_{\beta,\gamma}
	\int_{\mathbb{R}_+^3} Z^\beta B_h \cdot Z^\gamma \partial_h \omega_h^u \cdot Z^\alpha \omega_h^B dx
	+\int_{\mathbb{R}_+^3} Z^\alpha B_h \cdot \partial_h \omega_h^u \cdot Z^\alpha \omega_h^B dx
	\\
	&\lesssim\sum\limits_{\substack{\beta + \gamma = \alpha\\ \gamma \neq 0}}\!\!
	\Vert Z^\beta B_h \Vert^{\frac{1}{4}}_{L^2}
	\Vert \partial_1 Z^\beta B_h \Vert^{\frac{1}{4}}_{L^2}
	\Vert \partial_2 Z^\beta B_h \Vert^{\frac{1}{4}}_{L^2}
	\Vert \partial_1 \partial_2 Z^\beta B_h \Vert^{\frac{1}{4}}_{L^2}
	\Vert Z^\gamma \partial_h \omega_h^u \Vert
	\Vert Z^\alpha \omega_h^B \Vert_{L^2}^{\frac{1}{2}}
	\Vert \partial_3 Z^\alpha \omega_h^B \Vert_{L^2}^{\frac{1}{2}}
	\\
	&~~~~+
	\Vert Z^\alpha B_h \Vert^{\frac{1}{2}}_{L^2}
	\Vert \partial_1 Z^\alpha B_h \Vert^{\frac{1}{2}}_{L^2}
	\Vert \partial_h \omega_h^u \Vert^{\frac{1}{2}}_{L^2}
	\Vert \partial_2 \partial_h \omega_h^u \Vert^{\frac{1}{2}}_{L^2}
	\Vert Z^\alpha \omega_h^B \Vert^{\frac{1}{2}}_{L^2}
	\Vert \partial_3 Z^\alpha \omega_h^B \Vert^{\frac{1}{2}}_{L^2}\\
	&
	\lesssim
	(\Vert B \Vert_{H_{co}^m}+\Vert \omega_h^B \Vert_{H_{co}^{m-1}})
	\Vert ( \partial_h \omega_h^u,
            \partial_3 \omega_h^B,\partial_1 B )\Vert_{H_{co}^{m-1}}^2.
\end{aligned}
\end{equation*}
We now focus on term $III_{9,2}$.
The case $\gamma_3 =0$ is easy, since $Z^\gamma  \partial_3 \omega_h^u=\partial_3 Z^\gamma \omega_h^u $, so we can use integration by parts and Lemma \ref{lemma1} to estimate it. Now we focus on the case $\gamma_3 \neq 0$.
We use integration by parts and Lemma \ref{lemma1} to obtain
\begin{equation*}
	\begin{aligned}
	III_{9,2}&=\sum\limits_{\beta + \gamma = \alpha}  C_{\beta,\gamma}
	\int_{\mathbb{R}_+^3} Z^\beta B_3 \bigg(
	\partial_3 Z^\gamma \omega_h^u
	+ \sum\limits_{k=0}^{\gamma_3-1} C_{k,\gamma_3}(\varphi') \partial_3 \partial_1 ^{\gamma_1}
	\partial_2 ^{\gamma_2}
	 Z_3^k  \omega_h^u \bigg) \cdot Z^\alpha \omega_h^B dx\\
	 &=-\sum\limits_{\beta + \gamma = \alpha}  C_{\beta,\gamma}
	 \int_{\mathbb{R}_+^3} \partial_3  Z^\beta B_3 Z^\gamma \omega_h^u \cdot Z^\alpha \omega_h^B dx
	 -\sum\limits_{\beta + \gamma = \alpha}  C_{\beta,\gamma}
	 \int_{\mathbb{R}_+^3} Z^\beta B_3 Z^\gamma \omega_h^u \cdot \partial_3 Z^\alpha \omega_h^B dx\\
	 &~~~~-\sum\limits_{\beta + \gamma = \alpha} \sum\limits_{k=0}^{\gamma_3-1} C_{\beta,\gamma}
	 \int_{\mathbb{R}_+^3} C_{k,\gamma_3}(\varphi')
	 \partial_3  Z^\beta B_3
	 \partial_1 ^{\gamma_1}
	 \partial_2 ^{\gamma_2}
	 Z_3^k \omega_h^u \cdot Z^\alpha \omega_h^B dx\\
	 &~~~~-\sum\limits_{\beta + \gamma = \alpha} \sum\limits_{k=0}^{\gamma_3-1} C_{\beta,\gamma}
	 \int_{\mathbb{R}_+^3} \partial_3 C_{k,\gamma_3}(\varphi')
	  Z^\beta B_3
	 \partial_1 ^{\gamma_1}
	 \partial_2 ^{\gamma_2}
	 Z_3^k \omega_h^u \cdot Z^\alpha \omega_h^B dx\\
	 &~~~~-\sum\limits_{\beta + \gamma = \alpha} \sum\limits_{k=0}^{\gamma_3-1} C_{\beta,\gamma}
	 \int_{\mathbb{R}_+^3} C_{k,\gamma_3}(\varphi')
	 Z^\beta B_3
	 \partial_1 ^{\gamma_1}
	 \partial_2 ^{\gamma_2}
	 Z_3^k \omega_h^u \cdot \partial_3 Z^\alpha \omega_h^B dx
\end{aligned}
\end{equation*}
\begin{equation*}
	 \begin{aligned}
	 &\lesssim
	 \sum\limits_{\beta + \gamma = \alpha}
	 \Vert \partial_3  Z^\beta B_3 \Vert^{\frac{1}{2}}_{L^2}
	 \Vert \partial_1 \partial_3  Z^\beta B_3 \Vert^{\frac{1}{2}}_{L^2}
	 \Vert Z^\gamma \omega_h^u \Vert^{\frac{1}{2}}_{L^2}
	 \Vert \partial_2 Z^\gamma \omega_h^u \Vert^{\frac{1}{2}}_{L^2}
	 \Vert Z^\alpha \omega_h^B \Vert^{\frac{1}{2}}_{L^2}
	 \Vert \partial_3 Z^\alpha \omega_h^B \Vert^{\frac{1}{2}}_{L^2}\\
	 &~~~~+\sum\limits_{\beta + \gamma = \alpha}
	 \Vert Z^\beta B_3 \Vert^{\frac{1}{4}}_{L^2}
	 \Vert \partial_1 Z^\beta B_3 \Vert^{\frac{1}{4}}_{L^2}
	 \Vert \partial_3 Z^\beta B_3 \Vert^{\frac{1}{4}}_{L^2}
	 \Vert \partial_1 \partial_3 Z^\beta B_3 \Vert^{\frac{1}{4}}_{L^2}\\
	 &~~~~~~~~~~~~~~~~~~~~\times
	 \Vert Z^\gamma \omega_h^u \Vert^{\frac{1}{2}}_{L^2}
	 \Vert \partial_2 Z^\gamma \omega_h^u \Vert^{\frac{1}{2}}_{L^2}
	 \Vert \partial_3 Z^\alpha \omega_h^B \Vert_{L^2}\\
	 &~~~~+\sum\limits_{\beta + \gamma = \alpha} \sum\limits_{k=0}^{\gamma_3-1}
	 \Vert \partial_3
	 Z^\beta B_3 \Vert^{\frac{1}{2}}_{L^2}
	 \Vert \partial_1 \partial_3
	 Z^\beta B_3 \Vert^{\frac{1}{2}}_{L^2}
	 \Vert \partial_1 ^{\gamma_1}
	 \partial_2 ^{\gamma_2}
	 Z_3^k \omega_h^u \Vert^{\frac{1}{2}}_{L^2}
	 \Vert \partial_2 \partial_1 ^{\gamma_1}
	 \partial_2 ^{\gamma_2}
	 Z_3^k \omega_h^u \Vert^{\frac{1}{2}}_{L^2}\\
	 &~~~~~~~~~~~~~~~~~~~~~~~~\times
	 \Vert Z^\alpha \omega_h^B \Vert^{\frac{1}{2}}_{L^2}
	 \Vert \partial_3 Z^\alpha \omega_h^B \Vert^{\frac{1}{2}}_{L^2}\\
	 &~~~~+\sum\limits_{\beta + \gamma = \alpha} \sum\limits_{k=0}^{\gamma_3-1}
	 \Vert
	 Z^\beta B_3 \Vert^{\frac{1}{2}}_{L^2}
	 \Vert \partial_1
	 Z^\beta B_3 \Vert^{\frac{1}{2}}_{L^2}
	 \Vert \partial_1 ^{\gamma_1}
	 \partial_2 ^{\gamma_2}
	 Z_3^k \omega_h^u \Vert^{\frac{1}{2}}_{L^2}
	 \Vert \partial_2 \partial_1 ^{\gamma_1}
	 \partial_2 ^{\gamma_2}
	 Z_3^k \omega_h^u \Vert^{\frac{1}{2}}_{L^2}\\
	 &~~~~~~~~~~~~~~~~~~~~~~~~\times
	 \Vert Z^\alpha \omega_h^B \Vert^{\frac{1}{2}}_{L^2}
	 \Vert \partial_3 Z^\alpha \omega_h^B \Vert^{\frac{1}{2}}_{L^2}\\
	 &~~~~+\sum\limits_{\beta + \gamma = \alpha} \sum\limits_{k=0}^{\gamma_3-1}
	 \Vert
	 Z^\beta B_3 \Vert^{\frac{1}{4}}_{L^2}
	 \Vert \partial_1
	 Z^\beta B_3 \Vert^{\frac{1}{4}}_{L^2}
	 \Vert \partial_3
	 Z^\beta B_3 \Vert^{\frac{1}{4}}_{L^2}
	 \Vert \partial_1 \partial_3
	 Z^\beta B_3 \Vert^{\frac{1}{4}}_{L^2}\\
	 &~~~~~~~~~~~~~~~~~~~~~~~~\times
	 \Vert \partial_1 ^{\gamma_1}
	 \partial_2 ^{\gamma_2}
	 Z_3^k \omega_h^u \Vert^{\frac{1}{2}}_{L^2}
	 \Vert \partial_2 \partial_1 ^{\gamma_1}
	 \partial_2 ^{\gamma_2}
	 Z_3^k \omega_h^u \Vert^{\frac{1}{2}}_{L^2}
	 \Vert \partial_3 Z^\alpha \omega_h^B \Vert_{L^2}\\
	 &\lesssim
	 (\Vert B \Vert_{H^{m}_{co}}
        +\Vert (\omega_h^u, \omega_h^B)\Vert_{H^{m-1}_{co}})
	 (\Vert \partial_3 B \Vert_{H^m_{co}}^2
	  +\Vert (\partial_h \omega_h^u, \partial_3 \omega_h^B,
       \partial_1 B ) \Vert_{H^{m-1}_{co}}^2),
\end{aligned}
\end{equation*}
where we have used the fact $\Vert Z^\alpha \omega_h^B \Vert_{L^2} \lesssim  \Vert \partial_3  \omega_h^B \Vert_{H_{co}^{m-1}}$ since $\gamma_3 \neq 0$ implies $\alpha_3 \neq 0$ in this case. Combining the estimates of terms $III_{9,1}$ and $III_{9,2}$, it holds
\begin{equation}\label{3227}
III_9
\lesssim
	 (\Vert B \Vert_{H^{m}_{co}}
        +\Vert (\omega_h^u, \omega_h^B)\Vert_{H^{m-1}_{co}})
	 (\Vert \partial_3 B \Vert_{H^m_{co}}^2
	  +\Vert (\partial_h \omega_h^u, \partial_3 \omega_h^B,
       \partial_1 B ) \Vert_{H^{m-1}_{co}}^2).
\end{equation}
Integrating \eqref{3204} and \eqref{eqE2} in time
and inserting all the estimates presented above for $III_1$ through $III_9$, we obtain \eqref{eq1-1} after using
equivalent norms in \eqref{b5} and \eqref{b7}. This concludes the proof of \eqref{eq1-0}.

\subsection{Estimate of horizontal dissipation of magnetic field }
This section aims to prove \eqref{eq1-2}, that is
\begin{equation*}
	\begin{aligned}
	G(t) \lesssim E(0) + E(t) + E(t)^\frac{3}{2} + G(t)^\frac{3}{2}.
\end{aligned}
\end{equation*}
For notational convenience, we drop the superscript $\ep$ throughout this section.
We use the Eq.\eqref{eq1}$_1$ to write
\begin{equation*}
	\begin{aligned}
	\partial_1 B
	=
	\partial_t u
	+ u \cdot \nabla u-\Delta_h u+\nabla p-\varepsilon \partial_3^2 u -B\cdot \nabla B.
\end{aligned}
\end{equation*}
Hence, for every $m \ge 3$, we have
\begin{equation}\label{eqGm-1}
	\begin{aligned}
		\Vert \partial_1 B \Vert_{H^{m-1}_{co}}^2
		&=
		\sum_{|\alpha| \le m-1}
		\int_{\mathbb{R}_+^3}
		Z^\alpha
		\left(
		\partial_t u
		+ u \cdot \nabla u-\Delta_h u+\nabla p-\varepsilon \partial_3^2 u -B\cdot \nabla B
		\right)
		\cdot Z^\alpha \partial_1 B
		dx\\
		&\overset{def}{=}IV_1+IV_2+IV_3+IV_4+IV_5+IV_6.
	\end{aligned}
\end{equation}
\textbf{Deal with term $IV_1$}.
Using the Eq.\eqref{eq1}$_3$, it is easy to check that
\begin{equation}\label{3238}
	\begin{aligned}
	IV_1 &= \sum_{|\alpha| \le m-1}
	\frac{d}{dt}  \int_{\mathbb{R}_+^3}
	Z^\alpha u \cdot Z^\alpha \partial_1 B dx
	-\sum_{|\alpha| \le m-1}
	\int_{\mathbb{R}_+^3}
	Z^\alpha u \cdot Z^\alpha \partial_1 \partial_t B dx \\
	&=\sum_{|\alpha| \le m-1}\frac{d}{dt} \int_{\mathbb{R}_+^3}
	Z^\alpha u \cdot Z^\alpha \partial_1 B dx
	\\
	&~~~~~~~~-\sum_{|\alpha| \le m-1}
	\int_{\mathbb{R}_+^3}
	Z^\alpha u \cdot Z^\alpha \partial_1 (- u \cdot \nabla B + \partial_3^2 B +\varepsilon \Delta_h B +B\cdot \nabla u +\partial_1 u
	) dx,
\end{aligned}
\end{equation}
which, together with Lemma \ref{lemma1}, yields directly
\begin{equation*}
	\begin{aligned}
IV_1
\lesssim&
\sum_{|\alpha| \le m-1}\frac{d}{dt} \int_{\mathbb{R}_+^3}
	Z^\alpha u \cdot Z^\alpha \partial_1 B dx
+(\Vert (\partial_h u,\partial_3 B) \Vert_{H^m_{co}}^2
	+\Vert \partial_3 \omega_h^B \Vert_{H^{m-1}_{co}}^2
    +\ep \Vert \partial_h B \Vert_{H^m_{co}}^2)\\
&+(\Vert (u, B) \Vert_{H^m_{co}}+\Vert \omega_h^u \Vert_{H^{m-1}_{co}})
	 (\Vert (\partial_h u, \partial_3 B) \Vert_{H^m_{co}}^2
	 +\Vert \partial_h \omega_h^u \Vert_{H^{m-1}_{co}}^2).
	\end{aligned}
\end{equation*}
\textbf{Deal with terms $IV_2, IV_3, IV_5$ and $IV_6$}.
We also use Lemma \ref{lemma1} repeatedly to obtain
\begin{equation*}
	\begin{aligned}
	IV_2 
	&\lesssim
	(\Vert u \Vert_{H^m_{co}}+\Vert \omega_h^u \Vert_{H^{m-1}_{co}})
	(\Vert (\partial_h u, \partial_3 B) \Vert_{H^m_{co}}^2
	 +
	\Vert (\partial_h \omega_h^u, \partial_1 B) \Vert_{H^{m-1}_{co}}^2),\\
	IV_3
&\lesssim
	\Vert \partial_h u \Vert_{H^m_{co}}
	\Vert \partial_1 B \Vert_{H^{m-1}_{co}},\\
IV_5
     &\lesssim
     \varepsilon
     \Vert \partial_3 u \Vert_{H^{m-1}_{co}}
     \Vert \partial_3 B \Vert_{H^m_{co}}
     +
     \varepsilon
	 \Vert \partial_3 u \Vert_{H^{m-1}_{co}}
	 \Vert \partial_1 B \Vert_{H^{m-1}_{co}},\\
IV_6
&\lesssim
\Vert B \Vert_{H^m_{co}}
	(\Vert \partial_3 B \Vert_{H^m_{co}}^2
	 +\Vert \partial_1 B \Vert_{H^{m-1}_{co}}^2).
\end{aligned}
\end{equation*}
\textbf{Deal with term $IV_4$}.
Using H\"{o}lder's inequality, we have
\begin{equation}\label{3231}
	IV_4 \lesssim \Vert \nabla p \Vert_{H^{m-1}_{co}} \Vert \partial_1 B \Vert_{H^{m-1}_{co}}.
\end{equation}
In order to bound the pressure $p$, we take divergence operator
to the Eq.\eqref{eq1}$_1$ and obtain the equation
\begin{equation*}
	\Delta p = \nabla \cdot (B \cdot \nabla B - u \cdot \nabla u)\quad {\rm in} ~~\mathbb{R}_+^3,
\end{equation*}
which, together with estimate \eqref{a15} in Lemma \ref{p}, yields directly
\begin{equation*}
	\begin{aligned}
	\Vert \nabla p \Vert_{H_{co}^{m-1}}
	&\lesssim
	\Vert B \cdot \nabla B \Vert_{H_{co}^{m-1}}
	+
	\Vert u \cdot \nabla u \Vert_{H_{co}^{m-1}}
	+\Vert \nabla \cdot (B \cdot \nabla B) \Vert_{H_{co}^{m-2}}
	+
	\Vert \nabla \cdot (u \cdot \nabla u) \Vert_{H_{co}^{m-2}}\\
	&\overset{def}{=}IV_{4,1}+IV_{4,2}+IV_{4,3}+IV_{4,4}.
\end{aligned}
\end{equation*}
Using Lemma \ref{lemma1}, it is easy to check that
\begin{equation}\label{3232}
	\begin{aligned}
		IV_{4,1}
		&\lesssim
		\sum_{|\alpha| \le m-1}
		\bigg(
		\sum_{\substack{\beta+\gamma=\alpha\\ \beta \neq 0}}
		\Vert Z^\beta B_h \cdot Z^\gamma \partial_h B \Vert_{L^2}
		+
		\Vert B_h \cdot Z^\alpha \partial_h B \Vert_{L^2}
		+
		\sum_{\beta+\gamma=\alpha}
		\Vert Z^\beta B_3 Z^\gamma \partial_3 B \Vert_{L^2}
		\bigg)\\
		&\lesssim
		\sum_{|\alpha| \le m-1} \sum_{\substack{\beta+\gamma=\alpha\\ \beta \neq 0}}
		\Vert Z^\beta B_h \Vert^\frac{1}{4}_{L^2}
		\Vert \partial_2 Z^\beta B_h \Vert^\frac{1}{4}_{L^2}
		\Vert \partial_3 Z^\beta B_h \Vert^\frac{1}{4}_{L^2}
		\Vert \partial_2 \partial_3 Z^\beta B_h \Vert^\frac{1}{4}_{L^2}
		\Vert Z^\gamma \partial_h B \Vert^\frac{1}{2}_{L^2}
		\Vert \partial_1 Z^\gamma \partial_h B \Vert^\frac{1}{2}_{L^2}\\
		&~~~~
		+\sum_{|\alpha| \le m-1}
		\Vert B_h \Vert^\frac{1}{4}_{L^2}
		\Vert \partial_1 B_h \Vert^\frac{1}{4}_{L^2}
		\Vert \partial_2 B_h \Vert^\frac{1}{4}_{L^2}
		\Vert \partial_1 \partial_2 B_h \Vert^\frac{1}{4}_{L^2}
		\Vert Z^\alpha \partial_h B \Vert^\frac{1}{2}_{L^2}
		\Vert \partial_3 Z^\alpha \partial_h B \Vert^\frac{1}{2}_{L^2}\\
		&~~~~
		+
		\sum_{|\alpha| \le m-1} \sum_{\beta+\gamma=\alpha}
		\Vert Z^\beta B_3 \Vert^\frac{1}{4}_{L^2}
		\Vert \partial_1 Z^\beta B_3 \Vert^\frac{1}{4}_{L^2}
		\Vert \partial_3 Z^\beta B_3 \Vert^\frac{1}{4}_{L^2}
		\Vert \partial_1 \partial_3 Z^\beta B_3 \Vert^\frac{1}{4}_{L^2}
		\Vert Z^\gamma \partial_3 B \Vert^\frac{1}{2}_{L^2}
		\Vert \partial_2 Z^\gamma \partial_3 B \Vert^\frac{1}{2}_{L^2}\\
		&\lesssim
		(\Vert B \Vert_{H^m_{co}}
		+
		\Vert \omega_h^B \Vert_{H^{m-1}_{co}})
		(\Vert \partial_3 B \Vert_{H^m_{co}}
		 +\Vert \partial_1 B \Vert_{H^{m-1}_{co}})
	.
	\end{aligned}
\end{equation}
Similarly, it is easy to check that
\begin{equation}\label{3233}
IV_{4,3}\lesssim
(\Vert B \Vert_{H^m_{co}}
+
\Vert \omega_h^B \Vert_{H^{m-1}_{co}})
(\Vert \partial_3 B \Vert_{H^m_{co}}
+\Vert \partial_1 B \Vert_{H^{m-1}_{co}})
	.
\end{equation}
Applying the divergence free condition $\nabla \cdot u=0$
and Lemma \ref{lemma1}, it is easy to check that
\begin{equation}\label{3234}
	\begin{aligned}
	IV_{4,2}
	&\lesssim
	\sum_{|\alpha| \le m-1}
	\bigg(
	\sum_{\substack{\beta+\gamma=\alpha\\ \beta \neq 0}}
	\Vert Z^\beta u \cdot Z^\gamma \nabla u \Vert_{L^2}
	+
	\Vert u \cdot Z^\alpha \nabla u \Vert_{L^2}
	\bigg)\\
	&\lesssim
	\sum_{|\alpha| \le m-1}
	\sum_{\substack{\beta+\gamma=\alpha\\ \beta \neq 0}}
	\Vert Z^\beta u_h \cdot Z^\gamma \partial_h u \Vert_{L^2}
	+
	\sum_{|\alpha| \le m-1}
	\sum_{\substack{\beta+\gamma=\alpha\\ \beta \neq 0}}
	\Vert Z^\beta u_3 Z^\gamma \partial_3 u \Vert_{L^2}
	\\
	&~~~~
	+\sum_{|\alpha| \le m-1}
	\Vert u_h \cdot Z^\alpha \partial_h u \Vert_{L^2}
	+\sum_{|\alpha| \le m-1}
	\Vert u_3 Z^\alpha \partial_3 u \Vert_{L^2}\\
&\lesssim
(\Vert u \Vert_{H^m_{co}}
+\Vert \omega_h^u \Vert_{H^{m-1}_{co}})
(\Vert \partial_h u \Vert_{H^m_{co}}
+\Vert \partial_h \omega_h^u \Vert_{H^{m-1}_{co}}),
	\end{aligned}
\end{equation}
and
\begin{equation}\label{3235}
	\begin{aligned}
		IV_{4,4}
		\lesssim&
		\sum_{|\alpha| \le m-2}
		\sum_{\beta+\gamma=\alpha}
		\bigg(
		\sum_{i,j=1}^{2}
		\Vert Z^\beta \partial_i u_j Z^\gamma \partial_j u_i \Vert_{L^2}
		+
		\sum_{i=1}^{2}
		\Vert Z^\beta \partial_i u_3 Z^\gamma \partial_3 u_i \Vert_{L^2}\bigg)
		\\
		&+\sum_{|\alpha| \le m-2}
		\sum_{\beta+\gamma=\alpha}
		\bigg(
		\sum_{j=1}^{2}
		\Vert Z^\beta \partial_3 u_j Z^\gamma \partial_j u_3 \Vert_{L^2}
		+
		\Vert Z^\beta \partial_3 u_3 Z^\gamma \partial_3 u_3 \Vert_{L^2}
		\bigg)
		\\
\lesssim& (\Vert u \Vert_{H^m_{co}}
+\Vert \omega_h^u \Vert_{H^{m-1}_{co}})
(\Vert \partial_h u \Vert_{H^m_{co}}
+\Vert \partial_h \omega_h^u \Vert_{H^{m-1}_{co}}).
\end{aligned}
\end{equation}	
The combination of estimates \eqref{3232}, \eqref{3233},
\eqref{3234} and \eqref{3235}, it holds
\begin{equation*}
	\Vert \nabla p \Vert_{H_{co}^{m-1}}
	\lesssim
		(\Vert (u, B) \Vert_{H^m_{co}}+\Vert (\omega_h^u,\omega_h^B) \Vert_{H^{m-1}_{co}})
		(\Vert (\partial_h u,\partial_3 B) \Vert_{H^m_{co}}
		 +\Vert (\partial_h \omega_h^u, \partial_1 B)\Vert_{H^{m-1}_{co}}),
\end{equation*}
which, together with \eqref{3231}, yields directly
\begin{equation*}\label{3236}
IV_4  \lesssim
		(\Vert (u, B) \Vert_{H^m_{co}}+\Vert (\omega_h^u,\omega_h^B) \Vert_{H^{m-1}_{co}})
		(\Vert (\partial_h u,\partial_3 B) \Vert_{H^m_{co}}^2
		 +\Vert (\partial_h \omega_h^u,\partial_1 B)\Vert_{H^{m-1}_{co}}^2).
\end{equation*}
Integrating \eqref{eqGm-1} in time, invoking the bound for $IV_1$ through $IV_6$ and applying equivalent norms
 \eqref{b5} and \eqref{b6}
 to the time integrals, we obtain \eqref{eq1-2}. Here, we give some details.
Indeed, we have
\begin{align*}
	\int_{0}^{t} IV_1 (\tau) d\tau
	\lesssim
	E(0)
	+
	E(t)
	+
	E(t)^\frac{3}{2},
\end{align*}
where we have used the following fact
\begin{equation*}
\sum_{|\alpha| \le m-1}
		\bigg|
		\int_{0}^{t}
		\frac{d}{d\tau}  \int_{\mathbb{R}_+^3}
		Z^\alpha u(\tau,x) \cdot Z^\alpha \partial_1 B(\tau,x) dx d \tau
		\bigg|
\le E(0)+E(t).
\end{equation*}
The integrals of $IV_3$ and $IV_5$ are slightly different. Indeed, we have
\begin{equation*}
\begin{aligned}
&\int_{0}^{t} IV_3(\tau) d\tau
\le E(t)^\frac{1}{2} G(t)^\frac{1}{2}\le \frac{1}{4}G(t)+ CE(t),\\
&\int_{0}^{t} IV_5(\tau) d\tau \lesssim
E(t)^\frac{1}{2}E(t)^\frac{1}{2}
+ E(t)^\frac{1}{2}G(t)^\frac{1}{2}
\le\frac{1}{4}G(t)+ CE(t),
\end{aligned}
\end{equation*}
and
\begin{equation*}
\int_{0}^{t} (IV_2(\tau)+IV_4(\tau)+IV_6(\tau)) d\tau
\lesssim
E(t)^\frac{3}{2}+G(t)^\frac{3}{2}.
\end{equation*}
Combining all the bounds above gives
\begin{equation*}
	\begin{aligned}
		G(t) \le \frac{1}{2} G(t) + C E(0) + C E(t) + C E(t)^\frac{3}{2} + C G(t)^\frac{3}{2}.
	\end{aligned}
\end{equation*}
This completes the proof of \eqref{eq1-2}.

\section{Asymptotic behavior of global solution}\label{asymptotic-behavior}

In this section, we will provide some convergence rate for the
anisotropic incompressible Navier-Stokes equation.
First of all, based on the global-in-time uniform estimate \eqref{uniform_estimate}
in Theorem \ref{th1}, one can establish the estimate \eqref{ns-estimate}
under the small initial data condition \eqref{small-01}.
Thus, our target here is to establish the decay rate estimate
of global solution for the anisotropic Navier-Stokes equation.
This time decay rate will play an important role in the convergence rate of global
solution with respect to $\varepsilon$.
Secondly, we will establish the asymptotic behavior
with respect to parameter $\ep$ and time $t$.
In this section, we will provide some convergence rate for the following
anisotropic incompressible Navier-Stokes equation
\begin{equation}\label{eq51}
\left\{\begin{array}{*{4}{ll}}
\partial_t u^\var + u^\var \cdot \nabla u^\var - \Delta_h u^\var + \nabla p^\var
= \varepsilon \partial_3^2 u^\var
& {\rm in} ~~\mathbb{R}_+^3,\\
\nabla \cdot u^\var = 0 & {\rm in} ~~\mathbb{R}_+^3,\\
	 		u_3^\var = 0,\quad \partial_3 u_h^\var=0
\quad & {\rm on} ~~ \mathbb{R}^2 \times \{x_3=0\},\\
u^\var|_{t=0}=u_0 & {\rm in} ~~\mathbb{R}_+^3.
	 	\end{array}\right.
	 \end{equation}
As $\var \rightarrow 0^+$, then $(u^\var, \nabla p^\var)$ will
converge to the solution $(u^0, \nabla p^0)$ satisfying
\begin{equation}\label{eq52}
\left\{\begin{array}{*{4}{ll}}
\partial_t u^0 + u^0 \cdot \nabla u^0 - \Delta_h u^0 + \nabla p^0=0
& {\rm in} ~~\mathbb{R}_+^3,\\
\nabla \cdot u^0 = 0 & {\rm in} ~~\mathbb{R}_+^3,\\
	 		u_3^0 =0
\quad & {\rm on} ~~\mathbb{R}^2 \times \{ x_3=0 \},\\
u^0|_{t=0}=u_0 & {\rm in} ~~\mathbb{R}_+^3.
	 	\end{array}\right.
	 \end{equation}
Then, we will establish uniform (with respect to $\ep$) decay rate for the
global solution of system \eqref{eq51}.

\subsection{Decay rate of global solution}

For notational convenience, we drop the superscript $\ep$ throughout this section.
Let us define
\begin{equation}\label{index}
\sigma\overset{def}{=}s-\frac{14s-13}{6(2-s)}
~\text{and}~ s \in \bigg(\frac{13}{14}, 1\bigg).
\end{equation}
For any small positive constant $\delta>0$, let us assume
\begin{equation}\label{decay-assumption}
(1+t)^s\left(\|u\|_{H_{tan}^m}^2+\|\omega_h^u\|_{H_{tan}^{m-1}}^2\right)
 +\int_0^t (1+\tau)^{\sigma}
   \left(\|\partial_h u\|_{H_{tan}^m}^2+\|\partial_h  \omega_h^u\|_{H_{tan}^{m-1}}^2\right)d\tau
\le \delta,
\end{equation}
for all $t \in (0, T]$.
Then, we will establish some energy estimates under the assumption \eqref{decay-assumption}.
\begin{lemm}\label{lemma51}
Under the assumptions of \eqref{index} and \eqref{decay-assumption},
the global solution of system \eqref{eq51} will obey
\begin{equation}\label{511}
\|\Lambda^{-s}_h u(t)\|_{L^2}^2
+\int_0^t \|\partial_h \Lambda^{-s}_h u\|_{L^2}^2 d\tau
+\varepsilon \int_0^t \|\partial_3 \Lambda^{-s}_h u\|_{L^2}^2 d\tau
\le   2\|\Lambda^{-s}_h u_0\|_{L^2}^2+C_{s,\sigma} \delta^2,
\end{equation}
where $C_{s,\sigma}$ is a constant independent of time.
\end{lemm}
\begin{proof}
From the Eq.\eqref{eq51}, it holds
\begin{equation}\label{512}
\begin{aligned}
\frac12\frac{d}{dt} \|\Lambda^{-s}_h u\|_{L^2}^2
+\|\partial_h \Lambda^{-s}_h u\|_{L^2}^2
+\varepsilon \|\partial_3 \Lambda^{-s}_h u\|_{L^2}^2
=\int_{\mathbb{R}_+^3} \Lambda^{-s}_h(u\cdot \nabla u)\cdot \Lambda^{-s}_h u dx.
\end{aligned}
\end{equation}
Using H\"{o}lder's inequality, inequality \eqref{a16} in Lemma \ref{H-L} and
inequality \eqref{a9} in Lemma \ref{lemma1}, we have
\begin{equation}\label{513}
\begin{aligned}
&~~~~\int_{\mathbb{R}_+^3} \Lambda^{-s}_h(u\cdot \nabla u)\cdot \Lambda^{-s}_h u dx\\
&\lesssim \int_0^{+\infty}
\left(\|u_h \partial_h u\|_{L^{\frac{1}{\frac12+\frac{s}{2}}}(\mathbb{R}^2)}
+\|u_3 \partial_3 u\|_{L^{\frac{1}{\frac12+\frac{s}{2}}}(\mathbb{R}^2)}\right)
\|\Lambda_h^{-s} u\|_{L^2(\mathbb{R}^2)}dx_3\\
&\lesssim
\int_0^{+\infty}
\|u_h \|_{L^{\frac{2}{s}}(\mathbb{R}^2)}
\| \partial_h u\|_{L^2(\mathbb{R}^2)}
\|\Lambda_h^{-s} u\|_{L^2(\mathbb{R}^2)}dx_3
+
\int_0^{+\infty}
\|u_3 \|_{L^{\frac{2}{s}}(\mathbb{R}^2)}
\| \partial_3 u\|_{L^2(\mathbb{R}^2)}
\|\Lambda_h^{-s} u\|_{L^2(\mathbb{R}^2)}dx_3
\\
&\lesssim
\left\|\|u_h\|_{L^\infty(\mathbb{R}_+)}\right\|_{L^{\frac2s}(\mathbb{R}^2)}
\|\partial_h u\|_{L^2}
\|\Lambda^{-s}_h u\|_{L^2}
+
\left\|\|u_3\|_{L^\infty(\mathbb{R}_+)}\right\|_{L^{\frac2s}(\mathbb{R}^2)}
\|\partial_3 u\|_{L^2}
\|\Lambda^{-s}_h u\|_{L^2}
\\
&\lesssim
(\|u_h\|_{L^2}\|\partial_2 u_h\|_{L^2}
           +\|\partial_1 u_h\|_{L^2}\|\partial_{12} u_h\|_{L^2})^{\frac{1-s}{2}}
           \|u_h\|_{L^2}^{\frac{2s-1}{2}}
           \|\partial_3 u_h\|_{L^2}^{\frac12}\|\partial_h u\|_{L^2}
           \|\Lambda^{-s}_h u\|_{L^2}\\
&~~~~+(\|u_3\|_{L^2}\|\partial_2 u_3\|_{L^2}
           +\|\partial_1 u_3\|_{L^2}\|\partial_{12} u_3\|_{L^2})^{\frac{1-s}{2}}
           \|u_3\|_{L^2}^{\frac{2s-1}{2}}
           \|\partial_3 u_3\|_{L^2}^{\frac12}
\|\partial_3 u\|_{L^2}
\|\Lambda^{-s}_h u\|_{L^2}\\
&\lesssim \|(u_h, \partial_1 u_h, \partial_3 u_h)\|_{L^2}^{\frac{1+s}{2}}
          \|\partial_h u\|_{H^1_{tan}}^{\frac{3-s}{2}}
          \|\Lambda^{-s}_h u\|_{L^2}
          +\|(u_3, \partial_1 u_3, \partial_3 u)\|_{L^2}^{\frac{2+s}{2}}
          \|\partial_h u\|_{H^1_{tan}}^{\frac{2-s}{2}}
          \|\Lambda^{-s}_h u\|_{L^2}.
\end{aligned}
\end{equation}
Substituting the estimate \eqref{513} into \eqref{512}
and integrating  over $[0, t]$, we have
\begin{equation}\label{514}
\begin{aligned}
&~~~~\frac12 \|\Lambda^{-s}_h u(t)\|_{L^2}^2
+\int_0^t \|\partial_h \Lambda^{-s}_h u\|_{L^2}^2 d\tau
+\varepsilon \int_0^t \|\partial_3 \Lambda^{-s}_h u\|_{L^2}^2 d\tau\\
&\lesssim  \frac12 \|\Lambda^{-s}_h u_0\|_{L^2}^2
     +\underset{t}{\sup}\|\Lambda^{-s}_h u\|_{L^2}
     \int_0^t  \|(u_h, \partial_1 u_h, \partial_3 u_h)\|_{L^2}^{\frac{1+s}{2}}
          \|\partial_h u\|_{H^1_{tan}}^{\frac{3-s}{2}} d\tau\\
&\quad +\underset{t}{\sup}\|\Lambda^{-s}_h u\|_{L^2}
     \int_0^t  \|(u_3, \partial_1 u_3, \partial_3 u)\|_{L^2}^{\frac{2+s}{2}}
          \|\partial_h u\|_{H^1_{tan}}^{\frac{2-s}{2}} d\tau.
\end{aligned}
\end{equation}
Using the estimate \eqref{decay-assumption} and H\"{o}lder's inequality, we have
\begin{equation}\label{515}
\begin{aligned}
&~~~~\int_0^t  \|(u_h, \partial_1 u_h, \partial_3 u_h)\|_{L^2}^{\frac{1+s}{2}}
          \|\partial_h u\|_{H^1_{tan}}^{\frac{3-s}{2}} d\tau\\
&\le
\left\{\int_0^t  \|(u_h, \partial_1 u_h, \partial_3 u_h)\|_{L^2}^2
         (1+\tau)^{-\frac{3-s}{1+s}\sigma} d\tau\right\}^{\frac{1+s}{4}}
 \left\{\int_0^t  \|\partial_h u\|_{H^1_{tan}}^2 (1+\tau)^{\sigma}
   d\tau\right\}^{\frac{3-s}{4}}\\
&\lesssim
\delta \left\{\int_0^t (1+\tau)^{-s-\frac{3-s}{1+s}\sigma} d\tau\right\}^{\frac{1+s}{4}}
\lesssim
C_{s,\sigma}\delta,
\end{aligned}
\end{equation}
and
\begin{equation}\label{516}
\begin{aligned}
&~~~~\int_0^t  \|(u_3, \partial_1 u_3, \partial_3 u)\|_{L^2}^{\frac{2+s}{2}}
          \|\partial_h u\|_{H^1_{tan}}^{\frac{2-s}{2}} d\tau\\
&\lesssim
\left\{\int_0^t  \|(u_3, \partial_1 u_3, \partial_3 u)\|_{L^2}^2
         (1+\tau)^{-\frac{2-s}{2+s}\sigma} d\tau\right\}^{\frac{2+s}{4}}
 \left\{\int_0^t  \|\partial_h u\|_{H^1_{tan}}^2 (1+\tau)^{\sigma}
   d\tau\right\}^{\frac{2-s}{4}}\\
&\lesssim
\delta \left\{\int_0^t (1+\tau)^{-s-\frac{2-s}{2+s}\sigma} d\tau\right\}^{\frac{2+s}{4}}
\lesssim
C_{s,\sigma}\delta.
\end{aligned}
\end{equation}
Substituting the estimates \eqref{515} and \eqref{516}
into \eqref{514} and using the Cauchy's inequality, it holds
\begin{equation}
\|\Lambda^{-s}_h u\|_{L^2}^2
+\int_0^t \|\partial_h \Lambda^{-s}_h u\|_{L^2}^2 d\tau
+\varepsilon \int_0^t \|\partial_3 \Lambda^{-s}_h u\|_{L^2}^2 d\tau
\le   2\|\Lambda^{-s}_h u_0\|_{L^2}^2+C_{s,\sigma} \delta^2.
\end{equation}
Therefore, we complete the proof of this lemma.
\end{proof}

\begin{lemm}\label{lemma52}
Under the assumptions of \eqref{index} and \eqref{decay-assumption},
the global solution of system \eqref{eq51} will obey
\begin{equation}\label{521}
\|\Lambda_h^{-s} \omega^u_h\|_{L^2}^2
+\int_0^t \|\partial_h \Lambda_h^{-s} \omega^u_h\|_{L^2}^2 d\tau
+\varepsilon \int_0^t  \|\partial_3 \Lambda_h^{-s} \omega^u_h\|_{L^2}^2 d\tau
\le
2(\|\Lambda_h^{-s} (\omega^u_h)|_{t=0}\|_{L^2}^2+E(0)^2)+C_{s,\sigma} \delta^2,
\end{equation}
where $C_{s,\sigma}$ is a constant independent of time.
\end{lemm}
\begin{proof}
Taking vorticity operator to the Eq.\eqref{eq51}, then it holds
\begin{equation}\label{522}
\begin{aligned}
&~~~~\frac12\frac{d}{dt}\|\Lambda_h^{-s} \omega^u_h\|_{L^2}^2
+\|\partial_h \Lambda_h^{-s} \omega^u_h\|_{L^2}^2
+\varepsilon \|\partial_3 \Lambda_h^{-s} \omega^u_h\|_{L^2}^2\\
&=\int_{\mathbb{R}_+^3} \Lambda_h^{-s}(\omega^u \cdot \nabla u_h)\cdot \Lambda_h^{-s} \omega^u_h dx
-\int_{\mathbb{R}_+^3} \Lambda_h^{-s}(u\cdot \nabla \omega^u_h)\cdot \Lambda_h^{-s} \omega^u_h dx.
\end{aligned}
\end{equation}
The H\"{o}lder's inequality and  \eqref{a16} in Lemma \ref{H-L} yield directly
\begin{equation}\label{523}
\begin{aligned}
&~~~~\int_{\mathbb{R}_+^3} \Lambda_h^{-s}(u\cdot \nabla \omega^u_h)\cdot \Lambda_h^{-s} \omega^u_h dx\\
&\lesssim \int_0^{+\infty}
\left(\|u_h \partial_h \omega^u_h\|_{L^{\frac{1}{\frac12+\frac{s}{2}}}(\mathbb{R}^2)}
 +\|u_3 \partial_3 \omega^u_h\|_{L^{\frac{1}{\frac12+\frac{s}{2}}}(\mathbb{R}^2)}\right)
\|\Lambda_h^{-s} \omega^u_h\|_{L^2(\mathbb{R}^2)}dx_3.
\end{aligned}
\end{equation}
Obviously, it holds
\begin{equation}\label{525}
\begin{aligned}
&~~~~\int_0^{+\infty}
\|u_h \partial_h \omega^u_h\|_{L^{\frac{1}{\frac12+\frac{s}{2}}}(\mathbb{R}^2)}
\|\Lambda_h^{-s} \omega^u_h\|_{L^2(\mathbb{R}^2)}dx_3\\
&\lesssim
\int_0^{+\infty} \|u_h \|_{L^{\frac{2}{s}}(\mathbb{R}^2)}
\|\partial_h \omega^u_h\|_{L^2(\mathbb{R}^2)}
\|\Lambda_h^{-s} \omega^u_h\|_{L^2(\mathbb{R}^2)}dx_3\\
&\lesssim
\left\|\|u_h \|_{L^\infty(\mathbb{R}_+)}\right\|_{L^{\frac{2}{s}}}
\|\partial_h \omega^u_h\|_{L^2}
\|\Lambda_h^{-s} \omega^u_h\|_{L^2}\\
&\lesssim
 \left(\|u_h\|_{L^2}\|\partial_2 u_h\|_{L^2}
  +\|\partial_1 u_h\|_{L^2}\|\partial_{12}u_h\|_{L^2}\right)^{\frac{1-s}{2}}
 \|u_h\|_{L^2}^{\frac{2s-1}{2}}
 \|\partial_3 u_h\|_{L^2}^{\frac12}
 \|\partial_h \omega^u_h\|_{L^2}
 \|\Lambda_h^{-s} \omega^u_h\|_{L^2}\\
&\lesssim \|(u_h, \partial_1 u_h, \partial_3 u_h)\|_{L^2}^{\frac{1+s}{2}}
          \|(\partial_2 u_h, \partial_{12}u_h, \partial_h \omega^u_h)\|_{L^2}^{\frac{3-s}{2}}
          \|\Lambda_h^{-s} \omega^u_h\|_{L^2}.
\end{aligned}
\end{equation}
Now, let us deal with the difficult term $\int_0^{+\infty}
\|u_3 \partial_3 \omega^u_h\|_{L^{\frac{1}{\frac12+\frac{s}{2}}}(\mathbb{R}^2)}
\|\Lambda_h^{-s} \omega^u_h\|_{L^2(\mathbb{R}^2)}dx_3$.
Indeed, it is easy to check that
\begin{equation*}
\begin{aligned}
u_3 \partial_3 \omega^u_h
&=u_3 \partial_3 \omega^u_h \chi(x_3)+u_3 \partial_3 \omega^u_h (1-\chi(x_3))\\
&=\int_0^{x_3} \partial_3 u_3 d\xi \partial_3 \omega^u_h \chi(x_3)
 +\varphi^{-1}u_3 \varphi\partial_3 \omega^u_h (1-\chi(x_3)),
\end{aligned}
\end{equation*}
which yields directly
\begin{equation*}
|u_3 \partial_3 \omega^u_h|
\lesssim \|\partial_3 u_3\|_{L^\infty(\mathbb{R}_+)}|Z_3 \omega^u_h|
    +|u_3||Z_3 \omega^u_h|.
\end{equation*}
Here, $\chi(x_3)$ is a smooth cutoff function supported
near the boundary $x_3=0$.
Then, this decomposition yields directly
\begin{equation}\label{524}
\begin{aligned}
&~~~~\int_0^{+\infty}
\|u_3 \partial_3 \omega^u_h\|_{L^{\frac{1}{\frac12+\frac{s}{2}}}(\mathbb{R}^2)}
\|\Lambda_h^{-s} \omega^u_h\|_{L^2(\mathbb{R}^2)}dx_3\\
&\lesssim\int_0^{+\infty}
\left(\left\|\|\partial_3 u_3\|_{L^\infty(\mathbb{R}_+)}\right\|_{L^{\frac2s}(\mathbb{R}^2)}
+\left\|\|u_3\|_{L^\infty(\mathbb{R}_+)}\right\|_{L^{\frac2s}(\mathbb{R}^2)}\right)
\|Z_3 \omega^u_h\|_{L^2(\mathbb{R}^2)}
\|\Lambda_h^{-s} \omega^u_h\|_{L^2(\mathbb{R}^2)}dx_3\\
&\lesssim\left(\left\|\|\partial_3 u_3\|_{L^\infty(\mathbb{R}_+)}\right\|_{L^{\frac2s}(\mathbb{R}^2)}
+\left\|\|u_3\|_{L^\infty(\mathbb{R}_+)}\right\|_{L^{\frac2s}(\mathbb{R}^2)}\right)
\|Z_3 \omega^u_h\|_{L^2}\|\Lambda_h^{-s} \omega^u_h\|_{L^2}\\
&\lesssim\left(\left\|\|\partial_3 u_3\|_{L^\infty(\mathbb{R}_+)}\right\|_{L^{\frac2s}(\mathbb{R}^2)}
+\left\|\|u_3\|_{L^\infty(\mathbb{R}_+)}\right\|_{L^{\frac2s}(\mathbb{R}^2)}\right)
\|\Lambda_h^{-s} \omega^u_h\|_{L^2}\\
&~~~~\times \left(\|\omega^u_h\|_{L^2}
          +\|\omega^u_h\|_{L^2}^{\frac34}\|Z_3^3 \omega^u_h\|_{L^2}^{\frac14}
          +\|\omega^u_h\|_{L^2}^{\frac23}\|Z_3^3 \omega^u_h\|_{L^2}^{\frac13} \right)\\
&\lesssim \|(\omega^u_h, Z_3^3 \omega^u_h)\|_{L^2}^{\frac13}
         \left(\|\partial_h u_h\|_{H^2_{tan}}^{\frac12}
          \|\partial_h \partial_3 u_h\|_{L^2}^{\frac12}
          +\|u_3\|_{H^1_{tan}}^{\frac{s}{2}}
          \|\partial_h u\|_{H^1_{tan}}^{\frac{2-s}{2}}\right)
          \|\omega^u_h\|_{L^2}^{\frac23}
          \|\Lambda_h^{-s} \omega^u_h\|_{L^2},
\end{aligned}
\end{equation}
where we have used the estimate
\begin{equation*}
\begin{aligned}
&~~~~\left\|\|\partial_3 u_3\|_{L^\infty(\mathbb{R}_+)}\right\|_{L^{\frac2s}(\mathbb{R}^2)}
+\left\|\|u_3\|_{L^\infty(\mathbb{R}_+)}\right\|_{L^{\frac2s}(\mathbb{R}^2)}\\
&\lesssim \left(\|\partial_3 u_3\|_{L^2}\|\partial_2 \partial_3 u_3\|_{L^2}
           +\|\partial_1 \partial_3 u_3\|_{L^2}
           \|\partial_{12}\partial_3 u_3\|_{L^2}\right)^{\frac{1-s}{2}}
           \|\partial_3 u_3\|_{L^2}^{\frac{2s-1}{2}}
           \|\partial_3 \partial_3 u_3\|_{L^2}^{\frac12}\\
         &~~~~+(\|u_3\|_{L^2}\|\partial_2 u_3\|_{L^2}
           +\|\partial_1 u_3\|_{L^2}\|\partial_{12}u_3\|_{L^2})^{\frac{1-s}{2}}
           \|u_3\|_{L^2}^{\frac{2s-1}{2}}
           \|\partial_3 u_3\|_{L^2}^{\frac12}\\
&\lesssim \|\partial_h u_h\|_{H^2_{tan}}^{\frac12}
          \|\partial_h \partial_3 u_h\|_{L^2}^{\frac12}
          +\|u_3\|_{H^1_{tan}}^{\frac{s}{2}}
          \|\partial_h u\|_{H^1_{tan}}^{\frac{2-s}{2}}.
\end{aligned}
\end{equation*}
Substituting estimates \eqref{525} and \eqref{524}  into \eqref{523}, we have
\begin{equation}\label{526}
\begin{aligned}
&~~~~\int_{\mathbb{R}_+^3} \Lambda_h^{-s}(u\cdot \nabla \omega^u_h)\cdot \Lambda_h^{-s} \omega^u_h dx\\
&\lesssim
\|(u_h, \partial_1 u_h, \partial_3 u_h)\|_{L^2}^{\frac{1+s}{2}}
          \|(\partial_2 u_h, \partial_{12}u_h, \partial_h \omega^u_h)\|_{L^2}^{\frac{3-s}{2}}
          \|\Lambda_h^{-s} \omega^u_h\|_{L^2}\\
&~~~~+\|(\omega^u_h, Z_3^3 \omega^u_h)\|_{L^2}^{\frac13}
         (\|\partial_h u_h\|_{H^2_{tan}}^{\frac12}
          \|\partial_h \partial_3 u_h\|_{L^2}^{\frac12}
          +\|u_3\|_{H^1_{tan}}^{\frac{s}{2}}
          \|\partial_h u\|_{H^1_{tan}}^{\frac{2-s}{2}})
          \|\omega^u_h\|_{L^2}^{\frac23}
          \|\Lambda_h^{-s} \omega^u_h\|_{L^2}.
\end{aligned}
\end{equation}
Similarly, it holds
\begin{equation}\label{527}
\int_{\mathbb{R}_+^3} \Lambda_h^{-s}(\omega^u\cdot \nabla u_h)\cdot \Lambda_h^{-s} \omega^u_h dx
\lesssim  \|(\partial_h u_h,  \partial_h^2 u_h,
               \partial_h^3 u_h,\partial_3 \partial_h u_h)\|_{L^2}
           \|(\omega^u_h, \partial_3 u_h)\|_{L^2}\|\Lambda_h^{-s} \omega^u_h\|_{L^2}.
\end{equation}
Substituting the estimates  \eqref{526} and  \eqref{527} into \eqref{522}, we have
\begin{equation*}
\begin{aligned}
&~~~~\frac12 \frac{d}{dt}\|\Lambda_h^{-s} \omega^u_h\|_{L^2}^2
+\|\partial_h \Lambda_h^{-s} \omega^u_h\|_{L^2}^2
+\varepsilon \|\partial_3 \Lambda_h^{-s} \omega^u_h\|_{L^2}^2\\
&\lesssim
\|(\omega^u_h, Z_3^3 \omega^u_h)\|_{L^2}^{\frac13}
         \left(\|\partial_h u_h\|_{H^2_{tan}}^{\frac12}
          \|\partial_h \partial_3 u_h\|_{L^2}^{\frac12}
          +\|u_3\|_{H^1_{tan}}^{\frac{s}{2}}
          \|\partial_h u\|_{H^1_{tan}}^{\frac{2-s}{2}}\right)
          \|\omega^u_h\|_{L^2}^{\frac23}
          \|\Lambda_h^{-s} \omega^u_h\|_{L^2}\\
&~~~~+\|(u_h, \partial_1 u_h, \partial_3 u_h)\|_{L^2}^{\frac{1+s}{2}}
          \|(\partial_2 u_h, \partial_{12}u_h, \partial_h \omega^u_h)\|_{L^2}^{\frac{3-s}{2}}
          \|\Lambda_h^{-s} \omega^u_h\|_{L^2}\\
&~~~~+\|(\partial_h u_h,  \partial_h^2 u_h,
               \partial_h^3 u_h,\partial_3 \partial_h u_h)\|_{L^2}
           \|(\omega^u_h, \partial_3 u_h)\|_{L^2}\|\Lambda_h^{-s} \omega^u_h\|_{L^2},
\end{aligned}
\end{equation*}
which, integrating over $[0, t]$, yields directly
\begin{equation}\label{528}
\begin{aligned}
&~~~~\frac12\|\Lambda_h^{-s} \omega^u_h(t)\|_{L^2}^2
+\int_0^t \|\partial_h \Lambda_h^{-s} \omega^u_h\|_{L^2}^2 d\tau
+\varepsilon \int_0^t  \|\partial_3 \Lambda_h^{-s} \omega^u_h\|_{L^2}^2 d\tau\\
&\lesssim
\frac12\|\Lambda_h^{-s} (\omega^u_h)|_{t=0}\|_{L^2}^2
+\frac14 \underset{t}{\sup}\|\Lambda_h^{-s} \omega^u_h\|_{L^2}^2\\
&~~~~+
\bigg\{\int_0^t \|(u_h, \partial_1 u_h, \partial_3 u_h)\|_{L^2}^{\frac{1+s}{2}}
          \|(\partial_2 u_h, \partial_{12}u_h, \partial_h \omega^u_h)\|_{L^2}^{\frac{3-s}{2}}d\tau\bigg\}^2\\
&~~~~+\bigg\{\int_0^t \|(\partial_h u_h,  \partial_h^2 u_h,
               \partial_h^3 u_h,\partial_3 \partial_h u_h)\|_{L^2}
           \|(\omega^u_h, \partial_3 u_h)\|_{L^2} d\tau\bigg\}^2\\
&~~~~+\underset{t}{\sup}\|(\omega^u_h, Z_3^3 \omega^u_h)\|_{L^2}^{\frac23}
\bigg\{\int_0^t
         (\|\partial_h u_h\|_{H^2_{tan}}^{\frac12}
          \|\partial_h \partial_3 u_h\|_{L^2}^{\frac12}
          +\|u_3\|_{H^1_{tan}}^{\frac{s}{2}}
          \|\partial_h u\|_{H^1_{tan}}^{\frac{2-s}{2}})
          \|\omega^u_h\|_{L^2}^{\frac23}d\tau \bigg\}^2.
\end{aligned}
\end{equation}
Using the assumption \eqref{decay-assumption}, it is easy to check that
\begin{equation}\label{529}
\begin{aligned}
&~~~~\int_0^t \|(u_h, \partial_1 u_h, \partial_3 u_h)\|_{L^2}^{\frac{1+s}{2}}
          \|(\partial_2 u_h, \partial_{12}u_h, \partial_h \omega^u_h)\|_{L^2}^{\frac{3-s}{2}}d\tau\\
&\lesssim
\bigg\{\int_0^t \|(u_h, \partial_1 u_h, \partial_3 u_h)\|_{L^2}^2
 (1+\tau)^{-\frac{3-s}{1+s}\sigma}d\tau\bigg\}^{\frac{1+s}{4}}
\bigg\{\int_0^t \|(\partial_2 u_h, \partial_{12}u_h, \partial_h \omega^u_h)\|_{L^2}^2
 (1+\tau)^{\sigma}d\tau\bigg\}^{\frac{3-s}{4}}\\
&\lesssim
\delta \bigg\{\int_0^t (1+\tau)^{-s-\frac{3-s}{1+s}\sigma}d\tau\bigg\}^{\frac{1+s}{4}}
\le  \delta C_{s, \sigma},
\end{aligned}
\end{equation}
and
\begin{equation}\label{5212}
\begin{aligned}
&~~~~\int_0^t  \|u_3\|_{H^1_{tan}}^{\frac{s}{2}}
          \|\partial_h u\|_{H^1_{tan}}^{\frac{2-s}{2}}
          \|\omega^u_h\|_{L^2}^{\frac23}d\tau\\
&\lesssim
\bigg\{\int_0^t \|(u_3, \partial_h u_3, \omega^u_h)\|_{L^2}
  ^{\frac{2(3s+4)}{3(s+2)}}(1+\tau)^{-\frac{2-s}{2+s}\sigma}d\tau\bigg\}^{\frac{2+s}{4}}
\bigg\{\int_0^t \|\partial_h u\|_{H^1_{tan}}^2 (1+\tau)^\sigma d\tau\bigg\}
 ^{\frac{2-s}{4}}\\
&\lesssim
\delta^{\frac56}\bigg\{\int_0^t (1+\tau)^{-\frac{3s+4}{3(s+2)}s-\frac{2-s}{2+s}\sigma}
  d\tau\bigg\}^{\frac{2+s}{4}}\\
&\lesssim
\delta^{\frac56}\bigg\{
  \int_0^t (1+\tau)^{\frac{-6s-13}{6(s+2)}}d\tau\bigg\}^{\frac{2+s}{4}}
\le \delta^{\frac56} C_{s, \sigma},
\end{aligned}
\end{equation}
where $s\in (\frac{13}{14}, 1)$ and $\sigma=s-\frac{14s-13}{6(2-s)}$.
Similarly, it is easy to check that
\begin{equation}\label{5210}
\begin{aligned}
\int_0^t \|(\partial_h u_h,  \partial_h^2 u_h,
               \partial_h^3 u_h,\partial_3 \partial_h u_h)\|_{L^2}
           \|(\omega^u_h, \partial_3 u_h)\|_{L^2} d\tau
\le  \delta C_{s, \sigma},
\end{aligned}
\end{equation}
and
\begin{equation}\label{5211}
\begin{aligned}
\int_0^t \|\partial_h u_h\|_{H^2_{tan}}^{\frac12}
          \|\partial_h \partial_3 u_h\|_{L^2}^{\frac12}
          \|\omega^u_h\|_{L^2}^{\frac23}d\tau
\le \delta^{\frac56} C_{s, \sigma}.
\end{aligned}
\end{equation}
Substituting estimates \eqref{529}-\eqref{5211}
into \eqref{528}, it holds
\begin{equation*}
\begin{aligned}
\|\Lambda_h^{-s} \omega^u_h\|_{L^2}^2
+\int_0^t \|\partial_h \Lambda_h^{-s} \omega^u_h\|_{L^2}^2 d\tau
+\varepsilon \int_0^t  \|\partial_3 \Lambda_h^{-s} \omega^u_h\|_{L^2}^2 d\tau
\le
2\left(\|\Lambda_h^{-s} (\omega^u_h)|_{t=0}\|_{L^2}^2+E(0)^2\right)+C_{s,\sigma} \delta^2.
\end{aligned}
\end{equation*}
Therefore, we complete the proof of this lemma.
\end{proof}

Now, it is already to establish the decay rate estimate
for the global solution of system \eqref{eq51}.
\begin{lemm}\label{lemma53}
Under the assumptions of \eqref{index} and \eqref{decay-assumption},
the global solution of system \eqref{eq51} will obey
\begin{equation}\label{531}
(1+t)^{s}\left(\|u(t)\|_{H_{tan}^m}^2+\|\omega_h^u(t)\|_{H_{tan}^{m-1}}^2\right)
\le C C_0,
\end{equation}
where the constant $C_0 \overset{def}{=} \|\Lambda^{-s}_h u_0\|_{L^2}^2
+\|\Lambda_h^{-s} (\omega^u_h)|_{t=0}\|_{L^2}^2+E(0)+C_{s,\sigma} \delta^2$.
\end{lemm}
\begin{proof}
Due to the system \eqref{eq51}, it is easy to establish the following
differential inequality:
\begin{equation}\label{532}
\frac{d}{dt}\left(\|u\|_{H_{tan}^m}^2+\|\omega_h^u\|_{H_{tan}^{m-1}}^2\right)
+\left(\|\partial_h u\|_{H_{tan}^m}^2+\|\partial_h  \omega_h^u\|_{H_{tan}^{m-1}}^2\right)
+\varepsilon\left(\|\partial_3 u\|_{H_{tan}^m}^2+\|\partial_3 \omega_h^u\|_{H_{tan}^{m-1}}^2\right)
\le 0,
\end{equation}
which can be proved in \eqref{eq999} in Appendix \ref{claim-estimates}.
Using the interpolation inequality, we have for $s \in (0, 1)$,
\begin{equation}\label{533}
\|u\|_{H_{tan}^m}^2
\lesssim \|\Lambda_h^{-s }u\|_{H_{tan}^m}^{\frac{2}{1+s}}
         \|\partial_h u\|_{H_{tan}^m}^{\frac{2s}{1+s}}
\lesssim \left(\|\Lambda_h^{-s }u\|_{L^2}+\|u\|_{H_{tan}^m}\right)^{\frac{2}{1+s}}
         \|\partial_h u\|_{H_{tan}^m}^{\frac{2s}{1+s}},
\end{equation}
and
\begin{equation}\label{534}
\|\omega_h^u\|_{H_{tan}^{m-1}}^2
\lesssim \|\Lambda_h^{-s }\omega_h^u\|_{H_{tan}^{m-1}}^{\frac{2}{1+s}}
         \|\partial_h \omega_h^u\|_{H_{tan}^{m-1}}^{\frac{2s}{1+s}}
\lesssim \left(\|\Lambda_h^{-s }\omega_h^u\|_{L^2}+\|\omega_h^u\|_{H_{tan}^{m-1}}\right)^{\frac{2}{1+s}}
         \|\partial_h \omega_h^u\|_{H_{tan}^{m-1}}^{\frac{2s}{1+s}}.
\end{equation}
The combination of estimates \eqref{511}, \eqref{521} and Theorem \ref{th1} yields directly
\begin{equation}\label{535}
\begin{aligned}
&~~~~\|\Lambda_h^{-s }u\|_{L^2}^2+\|\Lambda_h^{-s }\omega_h^u\|_{L^2}^2
+\|u\|_{H_{tan}^m}^2+\|\omega_h^u\|_{H_{tan}^{m-1}}^2\\
&\lesssim \|\Lambda^{-s}_h u_0\|_{L^2}^2
+\|\Lambda_h^{-s} (\omega^u_h)|_{t=0}\|_{L^2}^2+E(0)+C_s \delta^2
\overset{def}{=}C_0,
\end{aligned}
\end{equation}
Substituting estimate \eqref{535} into \eqref{533} and \eqref{534}, then we have
\begin{equation*}
\|u\|_{H_{tan}^m}^2+\|\omega_h^u\|_{H_{tan}^{m-1}}^2
\lesssim C_0^{\frac{1}{1+s}}
        \left(\|\partial_h u\|_{H_{tan}^m}^{2}
         +\|\partial_h \omega_h^u\|_{H_{tan}^{m-1}}^2\right)^{\frac{s}{1+s}},
\end{equation*}
which, together with the differential inequality \eqref{532}, yields directly
\begin{equation*}
\frac{d}{dt}\left(\|u\|_{H_{tan}^m}^2+\|\omega_h^u\|_{H_{tan}^{m-1}}^2\right)
+C_0^{-\frac{1}{s}}\left(\|u\|_{H_{tan}^m}^2+\|\omega_h^u\|_{H_{tan}^{m-1}}^2\right)^{1+\frac{1}{s}}\le 0.
\end{equation*}
Then, we can obtain the decay estimate
\begin{equation*}
\|u(t)\|_{H_{tan}^m}^2+\|\omega_h^u(t)\|_{H_{tan}^{m-1}}^2
\lesssim C_0(1+t)^{-s},
\end{equation*}
or equivalently
\begin{equation*}
(1+t)^{s}\left(\|u(t)\|_{H_{tan}^m}^2+\|\omega_h^u(t)\|_{H_{tan}^{m-1}}^2\right)
\le C C_0.
\end{equation*}
Therefore, we complete the proof of this lemma.
\end{proof}

Finally, we will establish the time integral of viscous part
of velocity with suitable weight. This will help us control
the term in estimate \eqref{grow}.
\begin{lemm}\label{lemma54}
Under the assumptions of \eqref{index} and \eqref{decay-assumption},
the global solution of system \eqref{eq51} will obey
\begin{equation}\label{541}
\begin{aligned}
&(1+t)^{\sigma}\left(\|u\|_{H_{tan}^m}^2+\|\omega_h^u\|_{H_{tan}^{m-1}}^2\right)
 +\int_0^t (1+\tau)^{\sigma}
   \left(\|\partial_h u\|_{H_{tan}^m}^2+\|\partial_h  \omega_h^u\|_{H_{tan}^{m-1}}^2\right)d\tau\\
&+\varepsilon \int_0^t(1+\tau)^{\sigma}
   \left(\|\partial_3 u\|_{H_{tan}^m}^2+\|\partial_3 \omega_h^u\|_{H_{tan}^{m-1}}^2\right)d\tau
 \le C_{s, \sigma} C_0,
\end{aligned}
\end{equation}
where the constant $C_0$ is defined in Lemma \ref{lemma53}.
\end{lemm}
\begin{proof}
Multiplying \eqref{532} by $(1+t)^{\sigma}$, then we have
\begin{equation*}
\begin{aligned}
&\frac{d}{dt}[(1+t)^{\sigma}(\|u\|_{H_{tan}^m}^2+\|\omega_h^u\|_{H_{tan}^{m-1}}^2)]
 +(1+t)^{\sigma}
   (\|\partial_h u\|_{H_{tan}^m}^2+\|\partial_h  \omega_h^u\|_{H_{tan}^{m-1}}^2)\\
&+\varepsilon(1+t)^{\sigma}
   (\|\partial_3 u\|_{H_{tan}^m}^2+\|\partial_3 \omega_h^u\|_{H_{tan}^{m-1}}^2)
\le \sigma(1+t)^{\sigma-1}(\|u\|_{H_{tan}^m}^2+\|\omega_h^u\|_{H_{tan}^{m-1}}^2),
\end{aligned}
\end{equation*}
which, integrating over $[0, t]$, yields directly
\begin{equation*}
\begin{aligned}
&(1+t)^{\sigma}(\|u\|_{H_{tan}^m}^2+\|\omega_h^u\|_{H_{tan}^{m-1}}^2)
 +\int_0^t (1+\tau)^{\sigma}
   (\|\partial_h u\|_{H_{tan}^m}^2+\|\partial_h  \omega_h^u\|_{H_{tan}^{m-1}}^2)d\tau\\
&+\varepsilon \int_0^t(1+\tau)^{\sigma}
   (\|\partial_3 u\|_{H_{tan}^m}^2+\|\partial_3 \omega_h^u\|_{H_{tan}^{m-1}}^2)d\tau\\
&\lesssim  C_0+ C_0 \int_0^t(1+\tau)^{-(1+s-\sigma)}d\tau \le C_{s, \sigma} C_0.
\end{aligned}
\end{equation*}
Therefore, we complete the proof of this lemma.
\end{proof}

\begin{proof}[\textbf{\underline{Closed energy estimate}}]
Under the assumptions of \eqref{index} and \eqref{decay-assumption},
the estimates \eqref{531} and \eqref{541} yield
\begin{equation}\label{551}
(1+t)^{s}(\|u(t)\|_{H_{tan}^m}^2+\|\omega_h^u(t)\|_{H_{tan}^{m-1}}^2)\le C C_0,
\end{equation}
and
\begin{equation}\label{552}
\varepsilon \int_0^t(1+\tau)^{\sigma}
   (\|\partial_3 u\|_{H_{tan}^m}^2+\|\partial_3 \omega_h^u\|_{H_{tan}^{m-1}}^2)d\tau
+\int_0^t (1+\tau)^{\sigma}
   (\|\partial_h u\|_{H_{tan}^m}^2+\|\partial_h  \omega_h^u\|_{H_{tan}^{m-1}}^2)d\tau
\le C_{s, \sigma} C_0,
\end{equation}
where $C_0=\|\Lambda^{-s}_h u_0\|_{L^2}^2
+\|\Lambda_h^{-s} (\omega^u_h)|_{t=0}\|_{L^2}^2+E(0)+C_{s,\sigma} \delta^2$
($C_0$ is defined in Lemma \ref{lemma53}).
Then, the combination of estimates \eqref{551} and \eqref{552} yields directly
\begin{equation}\label{553}
\begin{aligned}
&~~~~(1+t)^{s}(\|u(t)\|_{H_{tan}^m}^2+\|\omega_h^u(t)\|_{H_{tan}^{m-1}}^2)
+\int_0^t (1+\tau)^{\sigma}
 (\|\partial_h u\|_{H_{tan}^m}^2+\|\partial_h  \omega_h^u\|_{H_{tan}^{m-1}}^2)d\tau\\
&~~~~+\varepsilon \int_0^t(1+\tau)^{\sigma}
   (\|\partial_3 u\|_{H_{tan}^m}^2+\|\partial_3 \omega_h^u\|_{H_{tan}^{m-1}}^2)d\tau\\
&\le C_{s, \sigma}(\|\Lambda^{-s}_h u_0\|_{L^2}^2
+\|\Lambda_h^{-s} (\omega^u_h)|_{t=0}\|_{L^2}^2+E(0))
+C_{s,\sigma} \delta^2.
\end{aligned}
\end{equation}
Now we choose the small constant $\delta\overset{def}{=}
4C_{s,\sigma}(\|\Lambda^{-s}_h u_0\|_{L^2}^2
+\|\Lambda_h^{-s} (\omega^u_h)|_{t=0}\|_{L^2}^2+E(0))
\le \min\{1, \frac{1}{4 C_{s,\sigma}}\}$, then we have
\begin{equation}\label{554}
\begin{aligned}
&~~~~(1+t)^{s}(\|u(t)\|_{H_{tan}^m}^2+\|\omega_h^u(t)\|_{H_{tan}^{m-1}}^2)
+\int_0^t (1+\tau)^{\sigma}
 (\|\partial_h u\|_{H_{tan}^m}^2+\|\partial_h  \omega_h^u\|_{H_{tan}^{m-1}}^2)d\tau\\
&~~~~+\varepsilon \int_0^t(1+\tau)^{\sigma}
   (\|\partial_3 u\|_{H_{tan}^m}^2+\|\partial_3 \omega_h^u\|_{H_{tan}^{m-1}}^2)d\tau\\
&\le C_{s,\sigma}(\|\Lambda^{-s}_h u_0\|_{L^2}^2
+\|\Lambda_h^{-s} (\omega^u_h)|_{t=0}\|_{L^2}^2+E(0))+\frac{\delta}{4}\\
&\le 2 C_{s,\sigma}(\|\Lambda^{-s}_h u_0\|_{L^2}^2
+\|\Lambda_h^{-s} (\omega^u_h)|_{t=0}\|_{L^2}^2+E(0))=\frac{\delta}{2}.
\end{aligned}
\end{equation}
The combination of estimates \eqref{decay-assumption} and
\eqref{554} will help us to establish the closed estimate.
Thus, we establish the time decay rate estimate \eqref{decay-1}
for the system \eqref{eq1-3}.
Similarly, one can establish similar times decay estimate \eqref{decay-2}
for the system \eqref{eq1-4}.
Therefore, we complete the proof of estimates \eqref{decay-1}
and \eqref{decay-2} in Theorem \ref{th2}.
\end{proof}

\subsection{Convergence rate of global solution}
In this subsection, we will establish the asymptotic behavior
for the quantity $u^\ep-u^0$ with respect to $\ep$ and time $t$.
Since the estimate \eqref{554} is independent of $\ep$, one can
take the same method to establish decay rate for the
global solution of system \eqref{eq52} as follows:
\begin{equation}\label{555}
\begin{aligned}
&~~~~(1+t)^{s}(\|u^0 (t)\|_{H_{tan}^m}^2+\|\omega_h^{u^0}(t)\|_{H_{tan}^{m-1}}^2)
+\int_0^t (1+\tau)^{\sigma}
 (\|\partial_h u^0\|_{H_{tan}^m}^2+\|\partial_h  \omega_h^{u^0}\|_{H_{tan}^{m-1}}^2)d\tau\\
&\le 2 C_{s,\sigma}(\|\Lambda^{-s}_h u_0\|_{L^2}^2
+\|\Lambda_h^{-s} (\omega^u_h)|_{t=0}\|_{L^2}^2+E(0)),
\end{aligned}
\end{equation}
where $\omega_h^{u^0}$ is defined by
$\omega_h^{u^0} \overset{def}{=}(\partial_2 u^0_3-\partial_3 u^0_2,
\partial_3 u^0_1-\partial_1 u^0_3)$.
Let us define $(\bar u, \bar p)\overset{def}{=}
(u^\ep-u^0, p^\ep-p^0)$, then $(\bar u, \bar p)$ will satisfy the system
\begin{equation}\label{eq53}
\left\{\begin{array}{*{4}{ll}}
\partial_t \bar u- \Delta_h \bar u+ \nabla \bar p
= \varepsilon \partial_3^2 u^\var
  -u^\ep \cdot \nabla \bar u- \bar u \cdot \nabla u^0
& {\rm in} ~~\mathbb{R}_+^3,\\
\nabla \cdot \bar u= 0 & {\rm in} ~~\mathbb{R}_+^3,\\
	 		\bar u_3= 0
\quad & {\rm on} ~~\mathbb{R}^2 \times \{x_3=0\},\\
\bar u|_{t=0}=0 & {\rm in} ~~\mathbb{R}_+^3.
	 	\end{array}\right.
	 \end{equation}
Now, we are already to establish the $L^2$-estimate for the quantity $\bar u$.
\begin{lemm}
Under the conditions of Theorem \ref{th2}, then it holds
\begin{equation}
\|\bar u\|_{L^2}
\le C_{s,\sigma}(\|\Lambda^{-s}_h u_0\|_{L^2}^2
+\|\Lambda_h^{-s} (\omega^u_h)|_{t=0}\|_{L^2}^2+E(0))
 \varepsilon^{\frac12},
\end{equation}
where the constant $C_{s,\sigma}$ is independent of $\ep$ and time $t$.
\end{lemm}
\begin{proof}
Multiplying Eq.$\eqref{eq53}_1$ by $\bar u$ and integrating by part,
it holds
\begin{equation*}
\begin{aligned}
\frac{1}{2}\frac{d}{dt}\|\bar u\|_{L^2}^2+\|\partial_h \bar u\|_{L^2}^2
=-\int_{\mathbb{R}_+^3} (\bar u \cdot \nabla) u^0 \cdot \bar u dx
-\varepsilon \int_{\mathbb{R}_+^3} \partial_3 u^\varepsilon \cdot \partial_3 \bar u dx.
\end{aligned}
\end{equation*}
Using the H${\rm \ddot{o}}$lder's and Cauchy's inequalities, we have
\begin{equation*}
\begin{aligned}
\bigg|\int_{\mathbb{R}_+^3} (\bar u \cdot \nabla) u^0 \cdot \bar u dx \bigg|
&\lesssim \|\bar u_h\|_{L^2}^{\frac12}\|\partial_1 {\bar u_h}\|_{L^2}^{\frac12}
          \|\partial_h u^0\|_{L^2}^{\frac12}\|\partial_3 \partial_h u^0\|_{L^2}^{\frac12}
          \|\bar u\|_{L^2}^{\frac12}\|\partial_2 {\bar u}\|_{L^2}^{\frac12}\\
         &~~~~+\|\bar u_3\|_{L^2}^{\frac12}\|\partial_3 {\bar u_3}\|_{L^2}^{\frac12}
          \|\partial_3 u^0\|_{L^2}^{\frac12}\|\partial_1 \partial_3 u^0\|_{L^2}^{\frac12}
          \|\bar u\|_{L^2}^{\frac12}\|\partial_2 {\bar u}\|_{L^2}^{\frac12}\\
&\lesssim \frac12 \|\partial_h {\bar u_h}\|_{L^2}^2
           +(\|\partial_3 \partial_h u^0\|_{L^2}\|\partial_h u^0\|_{L^2}
           +\|\partial_3 u^0\|_{L^2}\|\partial_1 \partial_3 u^0\|_{L^2})\|\bar u\|_{L^2}^2.
\end{aligned}
\end{equation*}
Therefore, it holds
\begin{equation*}
\frac{d}{dt}\|\bar u\|_{L^2}^2+\|\partial_h \bar u\|_{L^2}^2
\lesssim \|\partial_3 \partial_h u^0\|_{L^2}\|\nabla u^0\|_{L^2}\|\bar u\|_{L^2}^2
+\varepsilon \|\partial_3 u^\varepsilon\|_{L^2} \|\partial_3 \bar u\|_{L^2},
\end{equation*}
which, together with Gr\"{o}nwall's inequality, yields directly
\begin{equation}\label{grow}
\begin{aligned}
\|\bar u\|_{L^2}^2
&\lesssim \varepsilon \exp \bigg\{\int_0^t \|\partial_3 \partial_h u^0\|_{L^2}
                     \|\nabla u^0\|_{L^2} d \tau \bigg\}\\
         &~~~~\times \int_0^t \|\partial_3 u^\varepsilon\|_{L^2} \|\partial_3 \bar u\|_{L^2}
           \exp\bigg\{-\int_0^{\tau} \|\partial_3 \partial_h u^0\|_{L^2}
                     \|\nabla u^0\|_{L^2} d\varsigma\bigg\}d\tau\\
&\lesssim C_* C_{s,\sigma} \varepsilon \int_0^t \|\partial_3 u^\varepsilon\|_{L^2}
                  \|\partial_3 \bar u\|_{L^2}d\tau,
\end{aligned}
\end{equation}
where we have used the following estimate in the last inequality
\begin{equation*}
\begin{aligned}
\int_0^t \|\partial_3 \partial_h u^0\|_{L^2} \|\nabla u^0\|_{L^2} d \tau
&\le \bigg\{\int_0^t (1+\tau)^{\sigma}\|\partial_3 \partial_h u^0\|_{L^2}^2 d\tau\bigg\}^{\frac12}
\bigg\{\int_0^t  (1+\tau)^{-\sigma}\|\nabla u^0\|_{L^2}^2 d\tau\bigg\}^{\frac12}\\
&\lesssim C_*\bigg\{\int_0^t  (1+\tau)^{-\sigma-s}d\tau\bigg\}^{\frac12}
\lesssim C_*C_{s,\sigma}.
\end{aligned}
\end{equation*}
Here one uses the decay estimate \eqref{555} and
the constant $C_* \overset{def}{=}\|\Lambda^{-s}_h u_0\|_{L^2}^2
+\|\Lambda_h^{-s} (\omega^u_h)|_{t=0}\|_{L^2}^2+E(0)$.
Thus, we have
\begin{equation*}
\begin{aligned}
\|\bar u\|_{L^2}^2 &
\le  C_{s,\sigma} C_* \varepsilon
          \int_0^t \|\partial_3 u^\varepsilon\|_{L^2}
                  \|\partial_3 \bar u\|_{L^2}d\tau\\
&\le C_{s,\sigma} C_* \varepsilon^{\frac12}
 \left\{\varepsilon\int_0^t (1+\tau)^{\sigma}\|\partial_3 u^\varepsilon\|_{L^2}^2d\tau
 \right\}^{\frac12}
 \left\{\int_0^t (1+\tau)^{-\sigma}\|\partial_3 \bar u\|_{L^2}^2 d\tau
 \right\}^{\frac12}\\
& \le  C_{s,\sigma} C_* \varepsilon^{\frac12}
  \left\{\int_0^t (1+\tau)^{-(\sigma+s)}d\tau\right\}^{\frac12}
\le  C_{s,\sigma} C_* \varepsilon^{\frac12}
          [(1+t)^{1-(\sigma+s)}-1]^{\frac12},
\end{aligned}
\end{equation*}
which yields directly
\begin{equation*}
\|\bar u\|_{L^2}^2
\le C_{s,\sigma} C_* \varepsilon^{\frac12}.
\end{equation*}
Therefore, we complete the proof of this lemma.
\end{proof}

	\begin{sloppypar}
		\noindent\textbf{Data availability statement~} Data sharing is not applicable, as no data was collected or processed for the research reported in this article.
		\\
		
		\noindent\textbf{Conflicts of interest~} This work does not have any conflicts of interest.
		\\
		
				\noindent\textbf{Acknowledgements~}
Jincheng Gao was partially supported by National Key Research and Development Program of China (2021YFA1002100), Guangdong Basic and Applied Basic Research Foundation (2022A1515011798, 2021B1515310003) and
Guangzhou Science and Technology Progamme (2024A04J6410).
Jiahong Wu was partially supported by the National Science Foundation of the United
States (DMS 2104682, DMS 2309748).
Zheng-an Yao was partially supported by
National Natural Science Foundation of China (12126609) and
National Key Research and Development Program of China (2021YFA1002100).
	\end{sloppypar}

\begin{appendices}
\section{Some useful Sobolev inequalities}	\label{usefull-inequality}

Firstly, we will introduce the anisotropic Sobolev inequalities used frequently in Sections \ref{global-estimate}
and \ref{asymptotic-behavior}.

\begin{lemm}\label{lemma1}
For different numbers $i,j,k \in \{ 1,2,3\}$, the following inequalities hold when the right-hand side are all bounded:
\begin{equation}\label{a1}
\int_{\mathbb{R}_+^3}|fgh| dx
\lesssim
\Vert f \Vert_{L^2}^{\frac{1}{2}}
\Vert \partial_1 f \Vert_{L^2}^{\frac{1}{2}}
\Vert g \Vert_{L^2}^{\frac{1}{2}}
\Vert \partial_2 g \Vert_{L^2}^{\frac{1}{2}}
\Vert h \Vert_{L^2}^{\frac{1}{2}}
\Vert \partial_3 h \Vert_{L^2}^{\frac{1}{2}},
\end{equation}
\begin{equation}\label{a2}
			\int_{\mathbb{R}_+^3}
			|fgh| dx
			\lesssim
			\Vert f \Vert_{L^2}^{\frac{1}{4}}
			\Vert \partial_i f \Vert_{L^2}^{\frac{1}{4}}
			\Vert \partial_j f \Vert_{L^2}^{\frac{1}{4}}
			\Vert \partial_i \partial_j f \Vert_{L^2}^{\frac{1}{4}}
			\Vert g \Vert_{L^2}^{\frac{1}{2}}
			\Vert \partial_k g \Vert_{L^2}^{\frac{1}{2}}
			\Vert h \Vert_{L^2},
\end{equation}
\begin{equation}\label{a3}
			\Vert fg \Vert_{L^2}
			\lesssim
			\Vert f \Vert_{L^2}^{\frac{1}{4}}
			\Vert \partial_i f \Vert_{L^2}^{\frac{1}{4}}
			\Vert \partial_j f \Vert_{L^2}^{\frac{1}{4}}
			\Vert \partial_i \partial_j f \Vert_{L^2}^{\frac{1}{4}}
			\Vert g \Vert_{L^2}^{\frac{1}{2}}
			\Vert \partial_k g \Vert_{L^2}^{\frac{1}{2}},
\end{equation}
\begin{equation}\label{a4}
			\Vert f \Vert_{L^\infty}
			\lesssim
			\Vert f \Vert_{L^2}^{\frac{1}{2}}
			\Vert \partial_3 f \Vert_{L^2}^{\frac{1}{2}}
			+
			\Vert \partial_h f \Vert_{L^2}^{\frac{1}{2}}
			\Vert \partial_h \partial_3 f \Vert_{L^2}^{\frac{1}{2}}
			+
			\Vert \partial_h^2 f \Vert_{L^2}^{\frac{1}{2}}
			\Vert \partial_h^2 \partial_3 f \Vert_{L^2}^{\frac{1}{2}},
\end{equation}
\begin{equation}\label{a5}
\|Z_3 f\|_{L^2}
\lesssim \|f\|_{L^2}
          +\|f\|_{L^2}^{\frac34}\|Z_3^3 f\|_{L^2}^{\frac14}
          +\|f\|_{L^2}^{\frac23}\|Z_3^3 f\|_{L^2}^{\frac13}.
\end{equation}
\begin{equation}\label{a9}
\left\|\|f\|_{L^\infty(\mathbb{R}_+)}\right\|_{L^{\frac2s}(\mathbb{R}^2)}
\lesssim \left(\|f\|_{L^2}\|\partial_2 f\|_{L^2}
           +\|\partial_1 f\|_{L^2}\|\partial_{12}f\|_{L^2}\right)^{\frac{1-s}{2}}
           \|f\|_{L^2}^{\frac{2s-1}{2}}
           \|\partial_3 f\|_{L^2}^{\frac12}.
\end{equation}
	\end{lemm}
\begin{proof}
First of all, one can follow the idea in \cite[Lemma 1.2]{Wu2021Advance} to establish the inequalities \eqref{a1}-\eqref{a4} in Lemma \ref{lemma1}.
Thus, let us give the proof of inequality \eqref{a5}.
Indeed, integrating by part, it holds
\begin{equation*}
\|Z_3 f\|_{L^2}^2=-\int_{\mathbb{R}_+^3}(\varphi' Z_3 f+Z_3^2 f)f dx
\lesssim \|Z_3 f\|_{L^2}\|f\|_{L^2}
         +\|Z_3^2 f\|_{L^2}\|f\|_{L^2},
\end{equation*}
which yields directly
\begin{equation}\label{a10}
\|Z_3 f\|_{L^2}
\lesssim \|f\|_{L^2}
         +\|f\|_{L^2}^{\frac12}\|Z_3^2 f\|_{L^2}^{\frac12}.
\end{equation}
Applying estimate \eqref{a10} twice, it is easy to check that
\begin{equation*}
\begin{aligned}
\|Z_3^2 f\|_{L^2}
&\lesssim \|Z_3 f\|_{L^2}
         +\|Z_3 f\|_{L^2}^{\frac12}\|Z_3^3 f\|_{L^2}^{\frac12}\\
&\lesssim \|f\|_{L^2}
         +\|f\|_{L^2}^{\frac12}\|Z_3^2 f\|_{L^2}^{\frac12}
         +\left(\|f\|_{L^2}
         +\|f\|_{L^2}^{\frac12}\|Z_3^2 f\|_{L^2}^{\frac12}\right)^{\frac12}
         \|Z_3^3 f\|_{L^2}^{\frac12}\\
&\lesssim \frac12 \|Z_3^2 f\|_{L^2}+\|f\|_{L^2}
          +\|f\|_{L^2}^{\frac12}\|Z_3^3 f\|_{L^2}^{\frac12}
          +\|f\|_{L^2}^{\frac13}\|Z_3^3 f\|_{L^2}^{\frac23},
\end{aligned}
\end{equation*}
which implies directly
\begin{equation}\label{a11}
\|Z_3^2 f\|_{L^2}
\lesssim \|f\|_{L^2}
          +\|f\|_{L^2}^{\frac12}\|Z_3^3 f\|_{L^2}^{\frac12}
          +\|f\|_{L^2}^{\frac13}\|Z_3^3 f\|_{L^2}^{\frac23}.
\end{equation}
Then, the combination of estimates \eqref{a10} and \eqref{a11} yields directly
\begin{equation*}
\begin{aligned}
\|Z_3 f\|_{L^2}
&\lesssim \|f\|_{L^2}
         +\|f\|_{L^2}^{\frac12}
         \left\{\|f\|_{L^2}
          +\|f\|_{L^2}^{\frac12}\|Z_3^3 f\|_{L^2}^{\frac12}
          +\|f\|_{L^2}^{\frac13}\|Z_3^3 f\|_{L^2}^{\frac23}\right\}^{\frac12}\\
&\lesssim \|f\|_{L^2}
          +\|f\|_{L^2}^{\frac34}\|Z_3^3 f\|_{L^2}^{\frac14}
          +\|f\|_{L^2}^{\frac23}\|Z_3^3 f\|_{L^2}^{\frac13},
\end{aligned}
\end{equation*}
which yields the estimate \eqref{a5}.
Finally, let us establish the estimate \eqref{a9}.
Indeed, it holds
\begin{equation}\label{a6}
\begin{aligned}
\left\|\|f\|_{L^\infty(\mathbb{R}_+)}\right\|_{L^{\frac2s}(\mathbb{R}^2)}^{\frac2s}
\lesssim \left\|\|f\|_{L^2(\mathbb{R}_+)}^{\frac12}
 \|\partial_3 f\|_{L^2(\mathbb{R}_+)}^{\frac12}\right\|_{L^{\frac2s}(\mathbb{R}^2)}^{\frac2s}
\lesssim \|\partial_3 f\|_{L^2}^{\frac1s}
         \left\|\|f\|_{L^2(\mathbb{R}_+)}\right\|_{L^{\frac{2}{2s-1}}
         (\mathbb{R}^2)}^{\frac1s},
\end{aligned}
\end{equation}
and
\begin{equation}\label{a7}
\begin{aligned}
 \left\|\|f\|_{L^2(\mathbb{R}_+)}\right\|_{L^{\frac{2}{2s-1}}(\mathbb{R}^2)}
 ^{\frac{2}{2s-1}}
 &\lesssim \left\|\|f\|_{L^\infty(\mathbb{R}^2)}\right\|_{L^2(\mathbb{R}_+)}
 ^{\frac{4-4s}{2s-1}}\|f\|_{L^2}^2\\
 &\lesssim \left(\|f\|_{L^2}^{\frac12}\|\partial_2 f\|_{L^2}^{\frac12}
           +\|\partial_1 f\|_{L^2}^{\frac12}\|\partial_{12}f\|_{L^2}^{\frac12}\right)
           ^{\frac{4-4s}{2s-1}}\|f\|_{L^2}^2.
\end{aligned}
\end{equation}
The combination of estimates \eqref{a6} and \eqref{a7} gives
\begin{equation*}
\begin{aligned}
&\left\|\|f\|_{L^\infty(\mathbb{R}_+)}\right\|_{L^{\frac2s}(\mathbb{R}^2)}
\lesssim \|\partial_3 f\|_{L^2}^{\frac12}
         \left\|\|f\|_{L^2(\mathbb{R}_+)}\right\|_{L^{\frac{2}{2s-1}}
         (\mathbb{R}^2)}^{\frac12}\\
&\lesssim \left(\|f\|_{L^2}\|\partial_2 f\|_{L^2}
           +\|\partial_1 f\|_{L^2}\|\partial_{12}f\|_{L^2}\right)^{\frac{1-s}{2}}
           \|f\|_{L^2}^{\frac{2s-1}{2}}
           \|\partial_3 f\|_{L^2}^{\frac12}.
\end{aligned}
\end{equation*}
Therefore, we complete the proof of this lemma.
\end{proof}

First, we will state two elementary inequalities without proof
(i.e., Lemmas \ref{p} and \ref{H-L} below), one can
refer to \cite[Lemma 11]{Masmoudi2012ARMA}
and \cite[Theorem 1]{Stein1970} respectively.

\begin{lemm}[\cite{Masmoudi2012ARMA}]\label{p}
Consider the system
\begin{equation}\label{a12}
\partial_t u-\ep \Delta u+\nabla p=F,~\nabla \cdot u=0 ~\rm{in}~\mathbb{R}^3_+,
\end{equation}	
with the Navier boundary condition
\begin{equation*}\label{a13}
u_3=0,~\partial_3 u_h=2 \alpha u_h ~{\rm on} ~~ \mathbb{R}^2 \times \{ x_3=0 \},
\end{equation*}
where $F$ is some given source term. Applying divergence operator $\nabla \cdot $ to \eqref{a12}, we have
\begin{align*}
	\Delta p = \nabla \cdot F  ~\rm{in}~\mathbb{R}^3_+.
\end{align*}
Therefore, for $m \ge 2$, it holds
	\begin{equation*}\label{a14}
			\Vert \nabla p \Vert_{H_{co}^{m-1}} \lesssim
			\Vert F \Vert_{H_{co}^{m-1}}
			+\Vert \nabla \cdot F \Vert_{H_{co}^{m-2}}
            +\ep |\alpha| \Vert \nabla u_h \Vert_{H_{co}^{m-1}}^{\frac12}
             \Vert u_h \Vert_{H_{co}^{m-1}}^{\frac12}.
	\end{equation*}
\end{lemm}
In this paper, we only consider the boundary condition \eqref{eq46}
for the case of $\alpha=0$. Thus, under the assumption of boundary condition
\eqref{eq46}, we have the estimate of pressure as follows
	\begin{equation}\label{a15}
			\Vert \nabla p \Vert_{H_{co}^{m-1}} \lesssim
			\Vert F \Vert_{H_{co}^{m-1}}
			+\Vert \nabla \cdot F \Vert_{H_{co}^{m-2}}.
	\end{equation}
Next, we introduce the Hardy-Littlewood-Sobolev inequality
in \cite[pp. 119, Theorem 1]{Stein1970} as follow.
\begin{lemm}\label{H-L}
Let $0<\alpha<2 , 1<p<q<\infty, \frac{1}{q}+\frac{\alpha}{2}=\frac{1}{p}$, then
\begin{equation}\label{a16}
\|\Lambda^{-\alpha}f\|_{L^q}\lesssim \|f\|_{L^p}.
\end{equation}
\end{lemm}
In this paper, taking $q=2$ in \eqref{a16}, then
$
p=\frac{1}{\frac12+\frac{\alpha}{2}}=\frac{2}{1+\alpha}
$
should satisfy the condition
$$
1<p=\frac{2}{1+\alpha} <2.
$$
This implies the index $\alpha \in (0, 1)$.

\section{Proof of some claimed estimate}\label{claim-estimates}

In this section, we will give proof for the claimed estimates
\eqref{equivalent-norm} and \eqref{532}.

\begin{proof}[\textbf{\underline{Proof of equivalent norm \eqref{equivalent-norm}}}]
Due to the divergence-free condition and definition of vorticity, we have
\begin{equation*}
\begin{aligned}
\partial_3 u^\ep_1
&=\partial_3 u^\ep_1-\partial_1 u^\ep_3+\partial_1 u^\ep_3
=\omega_1^{u^\ep}+\partial_1 u^\ep_3,\\
\partial_3 u^\ep_2
&=\partial_3 u^\ep_2-\partial_2 u^\ep_3+\partial_2 u^\ep_3
=-\omega_2^{u^\ep}+\partial_2 u^\ep_3,\\
\partial_3 u^\ep_3
&=-\partial_1 u^\ep_1-\partial_2 u^\ep_2.
\end{aligned}
\end{equation*}
Then, it is easy to check that
\begin{equation}\label{b2}
\|u^\ep\|_{H^{m}_{co}}^2+\|\partial_3 u^\ep\|_{H^{m-1}_{co}}^2
\thicksim \|u^\ep\|_{H^{m}_{co}}^2+\|\omega_h^{u^\ep}\|_{H^{m-1}_{co}}^2.
\end{equation}
On the other hand, it is easy to check that
\begin{equation}\label{b3}
\|u^\ep \|_{H^m_{tan}}^2+\|\omega^{u^\ep}_h\|_{H^{m-1}_{co}}^2
\thicksim \|u^\ep \|_{H^m_{co}}^2+\|\omega^{u^\ep}_h\|_{H^{m-1}_{co}}^2.
\end{equation}
Obviously, in order to establish the relation \eqref{b3},
it only needs to establish the relation
\begin{equation}\label{b4}
\|u^\ep \|_{H^m_{co}}^2
\lesssim \|u^\ep \|_{H^m_{tan}}^2+\|\omega^{u^\ep}_h\|_{H^{m-1}_{co}}^2.
\end{equation}
Indeed, let us recall the definition of norm
$\|u^\ep \|_{H^m_{co}}^2{=}\sum_{0\le |\alpha|\le m}
\|Z^{\alpha_1}_1 Z^{\alpha_2}_2 Z^{\alpha_3}_3 u^\ep \|_{L^2}^2.
$
We will give the proof of estimate \eqref{b4} by induction.
Case $\alpha_3=0$: the norm relation \eqref{b4} obviously holds.
Case $\alpha_3=1$: one applies the divergence-free condition to obtain
\begin{equation}
\begin{aligned}
&~~~~\sum_{0\le |\alpha_1+\alpha_2|\le m-1}
\|Z^{\alpha_1}_1 Z^{\alpha_2}_2 Z_3 u^\ep \|_{L^2}^2\\
&\le \sum_{0\le |\alpha_1+\alpha_2|\le m-1}
(\|Z^{\alpha_1}_1 Z^{\alpha_2}_2 Z_3 u^\ep_h \|_{L^2}^2
 +\|Z^{\alpha_1}_1 Z^{\alpha_2}_2 \varphi (\partial_h \cdot u^\ep_h)\|_{L^2}^2)\\
&\le \sum_{0\le |\alpha_1+\alpha_2|\le m-1}
(\|Z^{\alpha_1}_1 Z^{\alpha_2}_2 \omega^{u^\ep}_h\|_{L^2}^2
+\|Z^{\alpha_1}_1 Z^{\alpha_2}_2 \partial_h u^\ep_3\|_{L^2}^2
+\|Z^{\alpha_1}_1 Z^{\alpha_2}_2 \varphi (\partial_h \cdot u^\ep_h)\|_{L^2}^2)\\
&\le  \|\omega^{u^\ep}_h\|_{H^{m-1}_{co}}^2
      +\|u^\ep \|_{H^{m}_{tan}}^2.
\end{aligned}
\end{equation}
Now, let us suppose $\alpha_3=k\ge 1$. It holds
\begin{equation}\label{control-assumption}
\sum_{0\le |\alpha_1+\alpha_2|\le m-k}
\|Z^{\alpha_1}_1 Z^{\alpha_2}_2 Z_3^k u^\ep \|_{L^2}^2
\le  \|\omega^{u^\ep}_h\|_{H^{m-1}_{co}}^2
      +\|u^\ep\|_{H^{m}_{tan}}^2.
\end{equation}
Let us give the proof for $\alpha_3=k+1$. It is easy to check that
\begin{equation*}
\begin{aligned}
&\sum_{0\le |\alpha_1+\alpha_2|\le m-(k+1)}
\|Z^{\alpha_1}_1 Z^{\alpha_2}_2 Z_3^{k+1} u^\ep\|_{L^2}^2\\
\le &\sum_{0\le |\alpha_1+\alpha_2|\le m-(k+1)}
(\|Z^{\alpha_1}_1 Z^{\alpha_2}_2 Z_3^{k} (\varphi \partial_3 u^\ep_h)\|_{L^2}^2
+\|Z^{\alpha_1}_1 Z^{\alpha_2}_2 Z_3^{k} (\varphi \partial_h \cdot u^\ep_h)\|_{L^2}^2)\\
\le &\sum_{0\le |\alpha_1+\alpha_2|\le m-(k+1)}
(\|Z^{\alpha_1}_1 Z^{\alpha_2}_2 Z_3^{k} (\varphi \omega^{u^\ep}_h)\|_{L^2}^2
+\| Z^{\alpha_1}_1 Z^{\alpha_2}_2 Z_3^{k} (\varphi \partial_h u^\ep_3)\|_{L^2}^2
+\| Z^{\alpha_1}_1 Z^{\alpha_2}_2 Z_3^{k} (\varphi \partial_h \cdot u^\ep_h)\|_{L^2}^2)\\
\lesssim & \|\omega^{u^\ep}_h\|_{H^{m-1}_{co}}^2
+\|u^\ep\|_{H^{m}_{tan}}^2,
\end{aligned}
\end{equation*}
where we have used the assumption \eqref{control-assumption} in the last inequality.
Then, with the help of induction, we complete the proof of estimate \eqref{b4}.
Then, the combination of equivalent estimates \eqref{b2} and \eqref{b3} yields directly
\begin{equation}\label{b5}
\|u^\ep\|_{H^{m}_{co}}^2+\|\partial_3 u^\ep\|_{H^{m-1}_{co}}^2
\thicksim \|u^\ep\|_{H^{m}_{co}}^2+\|\omega_h^{u^\ep}\|_{H^{m-1}_{co}}^2
\thicksim \|u^\ep \|_{H^m_{tan}}^2+\|\omega^{u^\ep}_h\|_{H^{m-1}_{co}}^2.
\end{equation}
Similar to the estimate \eqref{b5}, it is easy to check that
\begin{equation}\label{b6}
\begin{aligned}
&\|B^\ep\|_{H^{m}_{co}}^2+\|\partial_3 B^\ep\|_{H^{m-1}_{co}}^2
\thicksim \|B^\ep\|_{H^{m}_{co}}^2+\|\omega_h^{B^\ep}\|_{H^{m-1}_{co}}^2
\thicksim \|B^\ep \|_{H^m_{tan}}^2+\|\omega^{B^\ep}_h\|_{H^{m-1}_{co}}^2,\\
&\|\partial_h u^\ep\|_{H^{m}_{co}}^2
+\|\partial_3 \partial_h u^\ep\|_{H^{m-1}_{co}}^2
\thicksim \|\partial_h u^\ep \|_{H^{m}_{co}}^2
+\|\partial_h \omega_h^{u^\ep}\|_{H^{m-1}_{co}}^2
\thicksim \|\partial_h u^\ep\|_{H^m_{tan}}^2
+\|\partial_h \omega_h^{u^\ep}\|_{H^{m-1}_{co}}^2,\\
&\|\partial_3 B^\ep\|_{H^{m}_{co}}^2
+\|\partial_3^2 B^\ep\|_{H^{m-1}_{co}}^2
\thicksim \|\partial_3 B^\ep\|_{H^{m}_{co}}^2
+\|\partial_3 \omega_h^{B^\ep}\|_{H^{m-1}_{co}}^2
\thicksim \|\partial_3 B^\ep \|_{H^m_{tan}}^2
+\|\partial_3 \omega^{B^\ep}_h\|_{H^{m-1}_{co}}^2,\\
&\|\partial_3 u^\ep\|_{H^{m}_{co}}^2
+\|\partial_3^2 u^\ep\|_{H^{m-1}_{co}}^2
\thicksim \|\partial_3 u^\ep\|_{H^{m}_{co}}^2
+\|\partial_3 \omega_h^{u^\ep}\|_{H^{m-1}_{co}}^2
\thicksim \|\partial_3 u^\ep \|_{H^m_{tan}}^2
+\|\partial_3 \omega^{u^\ep}_h\|_{H^{m-1}_{co}}^2,\\
&\|\partial_h B^\ep\|_{H^{m}_{co}}^2
+\|\partial_3 \partial_h B^\ep\|_{H^{m-1}_{co}}^2
\thicksim \|\partial_h B^\ep\|_{H^{m}_{co}}^2
+\|\partial_h \omega_h^{B^\ep}\|_{H^{m-1}_{co}}^2
\thicksim \|\partial_h B^\ep \|_{H^m_{tan}}^2
+\|\partial_h \omega^{B^\ep}_h\|_{H^{m-1}_{co}}^2.
\end{aligned}
\end{equation}
Then, the combination of estimates \eqref{b5} and \eqref{b6} yield directly
\begin{equation}\label{b7}
E(t) \thicksim  X(t).
\end{equation}
Therefore, we complete the proof of claimed estimate \eqref{equivalent-norm}.
\end{proof}

\begin{proof}[\textbf{\underline{Proof of \eqref{532}}}]
Now, let us give the proof of claimed inequality \eqref{532} as follows
\begin{equation}\label{eq999}
\frac{d}{dt}(\|u\|_{H_{tan}^m}^2+\|\omega_h^u\|_{H_{tan}^{m-1}}^2)
+(\|\partial_h u\|_{H_{tan}^m}^2+\|\partial_h  \omega_h^u\|_{H_{tan}^{m-1}}^2)
+\varepsilon(\|\partial_3 u\|_{H_{tan}^m}^2+\|\partial_3 \omega_h^u\|_{H_{tan}^{m-1}}^2)
\le 0.
\end{equation}
Indeed, from the Eq.\eqref{3124}, \eqref{eqE1} and estimate \eqref{3117}, we have
\begin{equation}\label{Hmu}
\frac{1}{2}\frac{d}{dt}  \Vert u\Vert_{H^m_{tan}}^2
+\Vert \partial_h u \Vert_{H^m_{tan}}^2
+ \varepsilon \Vert \partial_3 u \Vert_{H^m_{tan}}^2
\lesssim
 ( \Vert u \Vert_{H_{tan}^m}+\Vert \omega_h^u \Vert_{H_{tan}^{m-1}} )
 (\Vert \partial_h u \Vert_{H_{tan}^m}^2
  +\Vert \partial_h \omega_h^u \Vert^{2}_{H_{tan}^{m-1}}).
\end{equation}
Furthermore, from the Eq.\eqref{eq2}, we have
\begin{equation}\label{b1}
\partial_t \omega_h^u + u \cdot \nabla \omega_h^u -\Delta_h \omega_h^u -\varepsilon \partial_3^2 \omega_h^u
= \omega^u \cdot \nabla u_h.
\end{equation}
The $L^2$-energy estimate of \eqref{b1} gives
\begin{equation}\label{L2omegau}
\frac{1}{2}\frac{d}{dt}\| \omega_h^u\|^2_{L^2}
		+ \|\partial_h w_h^u\|^2_{L^2}
		+ \varepsilon \|\partial_3 w_h^u\|^2_{L^2}
=
\int_{\mathbb{R}^3_+}
\omega^u \cdot \nabla u_h \cdot \omega_h^u dx
.
\end{equation}
Lemma \ref{lemma1} yields that
\begin{equation*}
	\begin{aligned}
&~~~~\int_{\mathbb{R}^3_+}
\omega^u \cdot \nabla u_h \cdot \omega_h^u dx
=
\int_{\mathbb{R}^3_+}
\omega^u_h \cdot \partial_h u_h \cdot \omega_h^u dx
+
\int_{\mathbb{R}^3_+}
\omega^u_3 \partial_3 u_h \cdot \omega_h^u dx
\\
&\lesssim
\Vert \omega^u_h \Vert_{L^2}^\frac{1}{2}
\Vert \partial_1 \omega^u_h \Vert_{L^2}^\frac{1}{2}
\Vert \partial_h u_h \Vert_{L^2}^\frac{1}{2}
\Vert \partial_3 \partial_h u_h \Vert_{L^2}^\frac{1}{2}
\Vert \omega_h^u \Vert_{L^2}^\frac{1}{2}
\Vert \partial_2 \omega_h^u \Vert_{L^2}^\frac{1}{2}
\\
&~~~~+
\Vert \omega^u_3 \Vert_{L^2}^\frac{1}{2}
\Vert \partial_3 \omega^u_3 \Vert_{L^2}^\frac{1}{2}
\Vert \partial_3 u_h \Vert_{L^2}^\frac{1}{2}
\Vert \partial_1 \partial_3 u_h \Vert_{L^2}^\frac{1}{2}
\Vert \omega_h^u \Vert_{L^2}^\frac{1}{2}
\Vert \partial_2 \omega_h^u \Vert_{L^2}^\frac{1}{2}
\\
&\lesssim
\Vert \omega^u_h \Vert_{L^2}
\Vert \partial_h \omega^u_h \Vert_{L^2}
\Vert \partial_h u_h \Vert_{L^2}^\frac{1}{2}
(
\Vert \partial_h \partial_h u_3 \Vert_{L^2}
+
\Vert \partial_h \omega_h^u \Vert_{L^2}
)
^\frac12
\\
&~~~~
+
\Vert \partial_h u \Vert_{L^2}^\frac{1}{2}
\Vert \partial_h \omega_h^u \Vert_{L^2}^\frac{1}{2}
(
\Vert \omega_h^u \Vert_{L^2}
+
\Vert \partial_h u \Vert_{L^2}
)^\frac{1}{2}
(
\Vert \partial_h \omega_h^u \Vert_{L^2}
+
\Vert \partial_h \partial_h u \Vert_{L^2}
)^\frac{1}{2}
\Vert \omega_h^u \Vert_{L^2}^\frac{1}{2}
\Vert \partial_h \omega_h^u \Vert_{L^2}^\frac{1}{2}
\\
&\lesssim
(
\Vert u \Vert_{H_{tan}^1}
+
\Vert \omega_h^u \Vert_{L^2}
)
(
\Vert \partial_h u \Vert_{H_{tan}^1}^2
+
\Vert \partial_h \omega_h^u \Vert_{L^2}^2
),
	\end{aligned}
\end{equation*}
where we have used the basic fact $\omega_3^u=\partial_1 u_2 - \partial_2 u_1$ and $\nabla \cdot \omega^u=0$.
Then, substituting this inequality into \eqref{L2omegau}, we have
\begin{equation}\label{L2}
	\frac{1}{2}\frac{d}{dt}\| \omega_h^u\|^2_{L^2}
	+ \|\partial_h w_h^u\|^2_{L^2}
	+ \varepsilon \|\partial_3 w_h^u\|^2_{L^2}
	\lesssim
	(
	\Vert u \Vert_{H_{tan}^1}
	+
	\Vert \omega_h^u \Vert_{L^2}
	)
	(
	\Vert \partial_h u \Vert_{H_{tan}^1}^2
	+
	\Vert \partial_h \omega_h^u \Vert_{L^2}^2
	)
	.
\end{equation}
The Eq.\eqref{b1} yields directly
\begin{equation}\label{omegau}
\begin{aligned}
&~~~~\frac{1}{2}\frac{d}{dt}\|\partial_h^{m-1} \omega_h^u\|^2_{L^2}
		+ \| \partial_h^{m-1} \partial_h w_h^u\|^2_{L^2}
		+ \varepsilon \| \partial_h^{m-1} \partial_3 w_h^u\|^2_{L^2}\\
&=- \int_{\mathbb{R}_+^3}
\partial_h^{m-1}(u\cdot\nabla \omega_h^u) \cdot \partial_h^{m-1} \omega_h^u dx
+
\int_{\mathbb{R}_+^3}\partial_h^{m-1}(\omega^u \cdot \nabla u_h)
 \cdot \partial_h^{m-1} \omega_h^u dx
\overset{def}{=} J_{4}+J_{5}.
\end{aligned}
\end{equation}
Using Lemma \ref{lemma1} repeatedly, we have
\begin{equation}\label{m1}
J_4\lesssim
(
\Vert u \Vert_{H^m_{tan}}
+
\Vert \omega_h^u \Vert_{H^{m-1}_{co}})
(
\Vert \partial_h u \Vert_{H^{m}_{tan}}^2
+
\Vert \partial_h \omega_h^u \Vert_{H^{m-1}_{tan}}^2),
\end{equation}
and
\begin{equation}\label{m2}
	\begin{aligned}
J_5
&\lesssim
(
\Vert u \Vert_{H^m_{tan}}
+
\Vert \omega_h^u \Vert_{H^{m-1}_{tan}}
)
(
\Vert \partial_h u \Vert_{H^m_{tan}}^2
+
\Vert \partial_h \omega_h^u \Vert_{H^{m-1}_{tan}}^2
).
	\end{aligned}
\end{equation}
Substituting \eqref{m1} and \eqref{m2} into \eqref{omegau}, we obtain
\begin{equation}\label{Hm-1omegau}
	\begin{aligned}
		&~~~~\frac{1}{2}\frac{d}{dt}\|\partial_h^{m-1} \omega_h^u\|^2_{L^2}
		+ \| \partial_h^{m-1} \partial_h w_h^u\|^2_{L^2}
		+ \varepsilon \| \partial_h^{m-1} \partial_3 w_h^u\|^2_{L^2}\\
		&\lesssim
		(
		\Vert u \Vert_{H^m_{tan}}
		+
		\Vert \omega_h^u \Vert_{H^{m-1}_{co}})
		(
		\Vert \partial_h u \Vert_{H^{m}_{tan}}^2
		+
		\Vert \partial_h \omega_h^u \Vert_{H^{m-1}_{tan}}^2).
	\end{aligned}
\end{equation}
The combination of \eqref{L2} and \eqref{Hm-1omegau} gives
\begin{equation}\label{Hm-1}
	\begin{aligned}
		&~~~~\frac{1}{2}\frac{d}{dt}\|\omega_h^u\|^2_{H^{m-1}_{tan}}
		+ \|\partial_h  w_h^u\|^2_{H^{m-1}_{tan}}
		+ \varepsilon \|\partial_3  w_h^u\|^2_{H^{m-1}_{tan}}\\
		&\lesssim
		(
		\Vert u \Vert_{H^m_{tan}}
		+
		\Vert \omega_h^u \Vert_{H^{m-1}_{co}})
		(
		\Vert \partial_h u \Vert_{H^{m}_{tan}}^2
		+
		\Vert \partial_h \omega_h^u \Vert_{H^{m-1}_{tan}}^2).
	\end{aligned}
\end{equation}
By \eqref{ns-estimate}, we combine \eqref{Hmu} and \eqref{Hm-1} and choose $\delta_1$ small enough to arrive \eqref{eq999}.
\end{proof}

\end{appendices}

		\phantomsection
		\addcontentsline{toc}{section}{\refname}
		\footnotesize
		\addtolength{\itemsep}{-1.0em}
		\setlength{\itemsep}{-4pt}
		\bibliographystyle{abbrv}
		\bibliography{Reference}
		
\end{document}